\documentclass[11pt,english]{article}
\usepackage[T1]{fontenc}
\usepackage[latin9]{inputenc}
\usepackage{geometry}
\usepackage{babel}
\usepackage{verbatim}
\usepackage{amsmath}
\usepackage{amsthm}
\usepackage{amssymb}
\usepackage[numbers]{natbib}
\usepackage{enumitem}
\usepackage[colorlinks,citecolor=blue,urlcolor=blue,linkcolor =blue,hypertexnames=false]{hyperref}

\makeatletter
\swapnumbers
\theoremstyle{plain}
\newtheorem{thm}{\protect\theoremname}
\theoremstyle{plain}
\newtheorem{lem}[thm]{\protect\lemmaname}
\theoremstyle{definition}
\newtheorem{rem}[thm]{\protect\remarkname}
\newtheorem{remarks}[thm]{\protect\remarksname}
\theoremstyle{definition}
\newtheorem{example}[thm]{\protect\examplename}
\theoremstyle{plain}
\newtheorem{cor}[thm]{\protect\corollaryname}
\theoremstyle{plain}
\newtheorem{prop}[thm]{\protect\propositionname}
\theoremstyle{plain}
\newtheorem{assumption}[thm]{\protect\assumptionname}

\numberwithin{equation}{section}
\numberwithin{thm}{section}

\usepackage{titlesec}
\usepackage[titletoc,toc,title]{appendix}
\usepackage{wasysym}
\usepackage{xcolor}
\usepackage{graphicx}

\providecommand{\assumptionname}{Assumption}

\providecommand{\examplename}{Example}
\providecommand{\lemmaname}{Lemma}
\providecommand{\propositionname}{Proposition}
\providecommand{\corollaryname}{Corollary}
\providecommand{\remarkname}{Remark}
\providecommand{\remarksname}{Remarks}
\providecommand{\theoremname}{Theorem}
\theoremstyle{definition}
\newtheorem{examples}[thm]{Examples}

\DeclareMathOperator{\E}{{\mathbb E}}

\DeclareMathOperator{\D}{{\mathbb D}}
\DeclareMathOperator{\R}{{\mathbb R}}

\DeclareMathOperator{\N}{{\mathbb N}}

\DeclareMathOperator{\PP}{{\mathbb P}}

\DeclareMathOperator{\Q}{{\mathbb Q}}

\DeclareMathOperator{\supp}{supp}

\DeclareMathOperator{\argmin}{argmin}
\DeclareMathOperator{\Var}{Var} \DeclareMathOperator{\Cov}{Cov}

\DeclareMathOperator{\dom}{dom}

\providecommand{\eps}{\varepsilon}
\renewcommand{\phi}{\varphi}

\renewcommand{\theta}{\vartheta}
\renewcommand{\subset}{\subseteq}

\providecommand{\abs}[1]{\lvert #1 \rvert}
\providecommand{\norm}[1]{\lVert #1 \rVert}
\providecommand{\bnorm}[1]{{\Bigl\lVert #1 \Bigr\rVert}}
\providecommand{\babs}[1]{{\Bigl\lvert #1 \Bigr\rvert}}
\providecommand{\scapro}[2]{\langle #1,#2 \rangle}

\providecommand{\wraum}{$(\Omega,{\scr F},\PP)$}
\providecommand{\fwraum}{$(\Omega,{\scr F},\PP,({\scr F}_t))$}

\renewcommand{\Im}{\operatorname{Im}}
\renewcommand{\le}{\leqslant}
\renewcommand{\leq}{\le}
\renewcommand{\ge}{\geqslant}
\renewcommand{\geq}{\ge}

\newtheorem*{assumption*}{\assumptionnumber}\providecommand{\assumptionnumber}{}

\newenvironment{assumptionWithArgs}[2]{%
  \stepcounter{thm}
  \renewcommand{\assumptionnumber}{\arabic{section}.\arabic{thm}($#1$;$#2$) Assumption}%
  \begin{assumption*}%
  \protected@edef\@currentlabel{\arabic{section}.\arabic{thm}}
 }{%
  \end{assumption*}
 }

\newcommand{\aswithargsref}[3]{\ref{#1}($#2$;$#3$)}

\makeatother

\providecommand{\corollaryname}{Corollary}
\providecommand{\examplename}{Example}
\providecommand{\lemmaname}{Lemma}
\providecommand{\propositionname}{Proposition}
\providecommand{\remarkname}{Remark}
\providecommand{\remarksname}{Remarks}
\providecommand{\theoremname}{Theorem}

\begin{document}
\global\long\def\epsilon{\varepsilon}

\global\long\def\E{\mathbb{E}}

\global\long\def\I{\mathbf{1}}

\global\long\def\N{\mathbb{N}}

\global\long\def\R{\mathbb{R}}

\global\long\def\C{\mathbb{C}}

\global\long\def\Q{\mathbb{Q}}

\global\long\def\P{\mathbb{P}}

\global\long\def\D{\Delta_{n}}

\global\long\def\dom{\operatorname{dom}}

\global\long\def\b#1{\mathbb{#1}}

\global\long\def\c#1{\mathcal{#1}}

\global\long\def\s#1{{\scriptstyle #1}}

\global\long\def\u#1#2{\underset{#2}{\underbrace{#1}}}

\global\long\def\r#1{\xrightarrow{#1}}

\global\long\def\mr#1{\mathrel{\raisebox{-2pt}{\ensuremath{\xrightarrow{#1}}}}}

\global\long\def\t#1{\left.#1\right|}

\global\long\def\l#1{\left.#1\right|}

\global\long\def\f#1{\lfloor#1\rfloor}

\global\long\def\sc#1#2{\langle#1,#2\rangle}

\global\long\def\abs#1{\lvert#1\rvert}

\global\long\def\bnorm#1{\Bigl\lVert#1\Bigr\rVert}

\global\long\def\wraum{(\Omega,\c F,\P)}

\global\long\def\fwraum{(\Omega,\c F,\P,(\c F_{t}))}

\global\long\def\norm#1{\lVert#1\rVert}

\global\long\def\theta{\vartheta}

\date{}

\author{Randolf Altmeyer, ~~~Markus Rei\ss  \\
\textit{ Humboldt-Universit\"at zu Berlin \footnote{Institut f\"ur Mathematik, Humboldt-Universit\"at zu Berlin, Unter den Linden 6, 10099 Berlin, Germany. Email: altmeyrx@math.hu-berlin.de, mreiss@math.hu-berlin.de.
\newline We are grateful to Sylvie Roelly, Nicolas Perkowski, Josef Janák  and two anonymous referees for very helpful comments and questions.
This research has been partially funded by Deutsche Forschungsgemeinschaft (DFG) - SFB1294/1 - 318763901.
}}}

\title{Nonparametric estimation for linear SPDEs\\
 from local measurements} 
 
\maketitle
\begin{abstract}
The coefficient function of the leading differential operator is estimated from observations of a linear stochastic partial differential equation (SPDE). The estimation is based on continuous time observations which are localised in space. For the asymptotic regime with fixed time horizon and with the spatial resolution of the observations tending to zero, we provide rate-optimal estimators and establish scaling limits of the deterministic PDE and of the SPDE on growing domains. The estimators are robust to lower order perturbations of the underlying differential operator and achieve the parametric rate even in the nonparametric setup with a spatially varying coefficient. A numerical example illustrates the main results.

\medskip



\end{abstract} 

\section{Introduction}

While there is a large amount of work on probabilistic, analytical
and recently also computational aspects of stochastic partial differential
equations (SPDEs), many natural statistical questions are open. With
this work we want to enlarge the scope of statistical methodology
in two major directions. First, we consider observations of a solution
path that are local in space and we ask whether the underlying differential
operator or rather its local characteristics can be estimated from
this local information only. Second, we allow the coefficients in
the differential operator to vary in space and we pursue nonparametric
estimation of the coefficient functions, as opposed to parametric
estimation approaches for finite-dimensional global parameters in
the coefficients. Naturally, both directions are intimately connected.

As a concrete model we consider the parabolic SPDE
\[
dX(t)=A_{\theta}X(t)dt+BdW(t),\quad t\in[0,T],
\]
with the second-order differential operator $A_{\theta}z:=\text{div}(\theta\nabla z)+\scapro{a}{\nabla z}+bz$
on some bounded domain $\Lambda\subset\R^{d}$ with Dirichlet boundary conditions, see Section \ref{sec:2}
for formal details. The coefficient functions $\theta,a,b$ are unknown
on $\Lambda$ and we aim at estimating $\theta:\Lambda\to\R^{+}$,
which models the diffusivity in a stochastic heat equation. The functions
$a,b$ as well as the operator $B$ in front of the driving space-time
white noise process $dW$ form an unknown nuisance part. Linear SPDEs of this form appear in many applications, including neuroscience (\citet{Walsh:1981bf}), oceanography (\citet{Frankignoul:1985bd}), geostatistics (\citet{Sigrist:2015jn}), surface growth  (\citet{10.2307/2397363}) and finance (\citet{Cont:2011ku}).

Measurements of a solution process $X$ necessarily have a minimal
spatial resolution $\delta>0$ and we dispose of the observations
$\scapro{X(t)}{K_{\delta,x_{0}}}$, where the solution is integrated in the spatial domain against
a kernel function $K_{\delta,x_{0}}$ with support of diameter
$\delta$ around some $x_{0}\in\Lambda$. We keep the time span $T$
fixed and construct an estimator, called \textit{proxy MLE}, which
for the resolution asymptotics $\delta\to0$ converges at rate $\delta$
to $\theta(x_{0})$ and satisfies a CLT, which we derive in the case of a local multiplication covariance operator $B$ in the SPDE. Another estimator, the so-called
\textit{augmented MLE}, will even converge under far more general
conditions and exhibit a smaller asymptotic variance, but requires
a second local observation process $\scapro{X(t)}{\Delta K_{\delta,x_{0}}}$
 in terms
of the Laplace operator $\Delta$. Clearly, if we have access to these
observations around all $x_{0}\in\Lambda$, then both estimators can
be used to estimate the diffusivity function $\theta$ nonparametrically
on all of $\Lambda$.

These results are statistically remarkable. First of all, even for
the parametric case that $\theta$ is a constant, it is not immediately
clear that $\theta$ is identified (i.e., exactly recovered) from
local observations in a shrinking neighbourhood around some $x_{0}\in\Lambda$
only. Probabilistically, this means that the local observation laws
are mutually singular for different values of $\theta$. What is more,
the bias-variance trade-off paradigm in nonparametric statistics does
not apply: asymptotic bias and standard deviation are both of order
$\delta$ and the CLT provides us even with a simple pointwise confidence
interval for $\theta$. The robustness of the estimators to lower
order parts in the differential operator and unknown $B$ is very
attractive for applications. The rate $\delta$ is shown to be the
best achievable rate in a minimax sense even for constant $\theta$
without nuisance parts.

The fundamental probabilistic structure behind these results is a
universal scaling limit of the observation process for $\delta\to0$.
At a highly localised level, the differential operator $A_{\theta}$
behaves like $\theta(x_{0})\Delta$, as expressed in Corollary \ref{thm:convToHeatEq}
below, and the construction of the estimators shows a certain scaling
invariance with respect to $B$. To study these scaling limits, we
need to consider the deterministic PDE on growing domains via the
stochastic Feynman-Kac approach and to deduce tight asymptotics for
the action of the semigroup and the heat kernels. Further tools like
the fourth moment theorem or the Feldman-Hajek Theorem rely on the
underlying Gaussian structure, but extensions to semi-linear SPDEs
seem possible.

Let us compare our localisation approach to the spectral approach,
introduced by \citet{Huebner:1993uy} and  then in \citet{Huebner1995}  for parametric estimation. In
the simplest case $A_{\theta}=\theta\Delta$ for some $\theta>0$
and $B$ commuting with $A_{\theta}$, the SPDE solution can be expressed
in the eigenbasis of the Laplace operator $\Delta$. If the first
$N$ coefficient processes (Fourier modes of $X$) are observed, then
a maximum-likelihood estimator for $\theta$ is asymptotically efficient
as $N\to\infty$. This approach has turned out to be very versatile,
allowing also for estimating time-dependent $\theta(t)$ nonparametrically
(\citet{Huebner:2000bi}) or to cover nonlinear SPDEs (\citet{Cialenco2011},
\citet{Pasemann:2019vw}). In particular, it helps to understand that the coefficient in the leading order of the differential operator
can be estimated with better rates than lower order coefficients. The methodology, however, is intrinsically
bound to observations in the spectral domain and to operators $A_{\theta}$
whose eigenfunctions, at least in the leading order, are independent
of $\theta$. In contrast, we work with local observations in space
and the unknown spectrum of the operators $A_{\theta}$ does not harm
us. More conceptually, we rely on the local action of the differential
operator $A_{\theta}$, while the spectral approach also applies to
an abstract operator in a Hilbert space setting.

Our case of spatially varying coefficients has been considered first
by \citet{Aihara:1988vqb} (with $a=b=0$) in a filtering problem.
The corresponding nonparametric estimation problem is then addressed
by \citet{Aihara:1989ih} with a sieve least squares estimator, but
they achieve consistency only for global observations with a growing
time horizon $T\to\infty$. In a stationary one-dimensional setting
\citet{Bibinger:2019de} ask whether the parameter $\theta>0$ can be
estimated when observing the solution only at $x_{0}$ over a fixed
time interval $[0,T]$. Interestingly, in the case $B=\sigma^{2}I$
the parameter $\theta$ cannot be recovered if the level $\sigma$
of the space-time white noise is unknown (see also lower bounds of \citet{Hildebrandt2019}). For a recent and exhaustive
survey on statistics for SPDEs we refer to \citet{Cialenco:2018en}.

In Section \ref{sec:2} the SPDE and the observation model are introduced
and in Section \ref{sec:3} the scaling properties along with the
resolution level $\delta$ are discussed. Section \ref{sec:Construction of the estimators}
derives our estimators via a least-squares and a likelihood approach
and provides basic insight into their error analysis. The main
convergence results as well as a minimax lower bound are presented
in Section \ref{sec:Main-results}. The findings are illustrated by
a numerical example in Section \ref{sec:Discussion-and-numerical}.
While the main steps in the proofs are presented together with the
results, all more technical arguments are delegated to the Appendix.

\section{The model}

\label{sec:2}

\subsection{Notation}

Let $\Lambda$ be a bounded open set in $\R^{d}$ with $C^{2}$-boundary
$\partial\Lambda$ and consider $L^{2}(\Lambda)$ with the usual $L^{2}$-norm
$\norm{\cdot}:=\norm{\cdot}_{L^{2}(\Lambda)}$. For any open set $U$ in $\R^d$ and any linear operator $A:L^2(U)\rightarrow L^2(U)$ let $\norm{A}_{L^2(U)}:=\norm{A}_{L^2(U)\rightarrow L^2(U)}$ denote the operator norm, and let $H^k(U)$ for $k\in\N$ be the $L^2$-Sobolev spaces. Define $H_{0}^{1}(\Lambda)$ as the closure of $C_{c}^{\infty}(\overline{\Lambda})$ in $H^{1}(\Lambda)$. We write $\sc{\cdot}{\cdot}_{\R^{d}}$
for the Euclidean inner product and $|\cdot|$ for the norm. Let us define a second order elliptic operator with
Dirichlet boundary conditions
\[
A_{\theta}=\Delta_{\theta}+A_{0},\,\,\,\,\c D(A_{\theta})=H_{0}^{1}(\Lambda)\cap H^{2}(\Lambda),
\]
where $\Delta_{\vartheta}z=\text{div}(\theta\nabla z)=\sum_{i=1}^{d}\partial_{i}(\vartheta\partial_{i}z)$
is the weighted Laplace operator with spatially varying diffusivity
$\theta\in C^{1+\alpha}(\overline{\Lambda})$ for $\alpha>0$, $\min_{x\in\overline{\Lambda}}\theta(x)>0$, and where $A_{0}z=\sc a{\nabla z}_{\R^{d}}+bz$ with functions $a\in C^{1+\alpha}(\overline{\Lambda};\R^d)$,
$b\in C^{\alpha}(\overline{\Lambda})$. The
regularity conditions on $\theta,a,b$ are such that the deterministic
PDE $\tfrac{d}{dt}u(t)=A_{\theta}^{*}u(t)$ with initial value $z\in L^{2}(\Lambda)$
has a sufficiently smooth solution (see proof of Proposition \ref{prop:semigroup_convergence}
below). Let $(S_{\theta}(t))_{t\geq0}$ denote the analytic semigroup on $L^{2}(\Lambda)$ generated by $A_{\theta}$ (cf. Theorem 3.1.3 of  \citet{lunardi2012analytic}), while $(e^{t\Delta})_{t\geq0}$ is the heat semigroup on $L^{2}(\R^{d})$ generated
by $\Delta=\Delta_1$ with domain $H^{2}(\R^{d})$.

\subsection{The SPDE model}

Throughout this work $T<\infty$ is fixed. Let $(\Omega,\c F,(\c F_{t})_{0\le t\le T},\P)$
be a filtered probability space with a cylindrical Brownian motion
$W$ on $L^{2}(\Lambda)$ ($dW$ is also referred to as \emph{space-time}
white noise), and let $B:L^{2}(\Lambda)\rightarrow L^{2}(\Lambda)$
be a bounded linear operator, which is not assumed to be trace class.
We study the linear stochastic partial differential equation
\begin{equation}
\begin{cases}
dX\left(t\right)=A_{\vartheta}X\left(t\right)dt+BdW\left(t\right),\,\,\,\,0<t\le T,\\
X\left(0\right)=X_{0},\\
X(t)|_{\partial\Lambda}=0,\,\,\,\,0<t\le T,
\end{cases}\label{eq:SPDE}
\end{equation}
with deterministic initial value $X_{0}\in L^{2}(\Lambda)$.

Our statistical analysis below relies on linear functionals of $X(t)$
rather than on $X(t)$ itself. We therefore use the weak solution
concept of \citet{DaPrato:2014wq}. If $\int_{0}^{T}\norm{S_{\theta}\left(t\right)B}_{HS(L^{2}(\Lambda))}^{2}dt<\infty$
with Hilbert-Schmidt norm $\norm{\cdot}_{HS(L^{2}(\Lambda))}$, then
the unique weak solution $(X(t))_{0\leq t\leq T}$ of the SPDE (\ref{eq:SPDE})
is given by the variation of constants formula, cf. Theorem 5.4 of
\citet{DaPrato:2014wq},
\begin{equation}
X(t)=S_{\theta}(t)X_{0}+\int_{0}^{t}S_{\theta}(t-s)B\,dW(s).\label{eq:weak_solution_X}
\end{equation}
It takes values in $L^{2}(\Lambda)$ and satisfies for $z\in H_{0}^{1}(\Lambda)\cap H^{2}(\Lambda)$
\begin{equation}
d\sc{X(t)}z=\sc{X(t)}{A_{\theta}^{*}z}dt+d\left\langle BW\left(t\right),z\right\rangle .\label{eq:weakFormulation}
\end{equation}
Clearly, for $z\in L^{2}(\Lambda)$
\begin{equation}
\sc{X(t)}z=\left\langle S_{\theta}\left(t\right)X_{0},z\right\rangle +\int_{0}^{t}\left\langle S_{\theta}^{*}\left(t-s\right)z,BdW\left(s\right)\right\rangle .\label{eq:weak_solution_l}
\end{equation}
If $\int_{0}^{T}\norm{S_{\theta}\left(t\right)B}_{HS(L^{2}(\Lambda))}^{2}dt=\infty$,
then the stochastic integral in (\ref{eq:weak_solution_X}) is well-defined
only in a space of distributions. For example, if $H^{-s}(\Lambda)$ is a fractional Sobolev space of negative order with $s>d/4$, then the natural
embedding $\iota:L^{2}(\Lambda)\rightarrow H^{-s}(\Lambda)$ is a
Hilbert-Schmidt operator such that $\int_{0}^{T}\norm{\iota S_{\theta}\left(t\right)B}_{HS(L^{2}(\Lambda),H^{-s}(\Lambda))}^{2}dt<\infty$,
and $X(t)$ takes values in $H^{-s}(\Lambda)$ (cf. Remark 5.6 of
\citet{Hairer:2009tt}). Still, (\ref{eq:weakFormulation}) and (\ref{eq:weak_solution_l})
remain valid, if $\sc{X(t)}z$ and $\sc{X(t)}{A_{\theta}^{*}z}$ are
interpreted as dual pairings between $H^{-s}(\Lambda)$ and its dual space
for $z\in C_{c}^{\infty}(\overline{\Lambda})$. %
\begin{comment}
by the dual pairings $[X(t),z]$, $[X(t),A_{\theta}^{*}z]$, where
$[u,v]:=\sc{(-\Delta)^{-s}u}{(-\Delta)^{s}v}$ for $u\in H^{-s}(\Lambda),$
$v\in H^{s}(\Lambda)$
\end{comment}
{}

On the other hand, denote the right hand side of the equation in (\ref{eq:weak_solution_l})
by $\ell(t,z)$ and observe that it is always well-defined for any $z\in L^{2}(\Lambda)$,
independent of the space in which (\ref{eq:weak_solution_X}) makes
sense, cf. Lemma 2.4.2 of \citet{liu2015stochastic}. The resulting process $\ell:=(\ell(t,z))_{0\leq t\leq T,z\in L^{2}(\Lambda)}$
thus extends the linear forms $z\mapsto\sc{X(t)}z$ from $C_c^\infty(\overline{\Lambda})$ to $L^{2}(\Lambda)$. It has the following properties.

\begin{prop} \label{prop:solution} $\ell$ is a Gaussian process with
mean function $(t,z)\mapsto\left\langle S_{\theta}\left(t\right)X_{0},z\right\rangle $
and covariance function at $0\leq t,t'\leq T$, $z,z'\in L^{2}(\Lambda)$
given by
\begin{equation}
\operatorname{Cov}(\ell(t,z),\ell(t',z'))=\int_{0}^{t\wedge t'}\sc{B^{*}S_{\theta}^{*}\left(t-s\right)z}{B^{*}S_{\theta}^{*}\left(t'-s\right)z'}ds.\label{eq:cov_function}
\end{equation}
Moreover, $\ell$ satisfies (\ref{eq:weakFormulation}) for $z\in H_{0}^{1}(\Lambda)\cap H^{2}(\Lambda)$,
if $\sc{X(t)}z$ and $\sc{X(t)}{A_{\theta}^{*}z}$ are replaced by
$\ell(t,z)$ and $\ell(t,A_{\theta}^{*}z)$.

\end{prop}
\begin{proof}
By (\ref{eq:weak_solution_l}), $\ell(t,z)$ for $z\in L^{2}(\Lambda)$
is Gaussian with mean $\left\langle S_{\theta}\left(t\right)X_{0},z\right\rangle $.
Itô's isometry (Proposition 4.28 of \citet{DaPrato:2014wq}) proves
(\ref{eq:cov_function}). If $z\in C_{c}^{\infty}(\overline{\Lambda})$,
then $\ell(t,z)=\sc{X(t)}z$ satisfies $d\,\ell(t,z)=\ell(t,A_{\theta}^{*}z)dt+d\left\langle BW\left(t\right),z\right\rangle $. This
extends to $z\in H_{0}^{1}(\Lambda)\cap H^{2}(\Lambda)$ by approximation
and continuity of $\ell:[0,T]\times L^{2}(\Lambda)\mapsto L^{2}(\P)$
from (\ref{eq:cov_function}).
\end{proof}
In the following, justified by this proposition, we write $\left\langle X(t),z\right\rangle $
for $0\leq t\leq T$ and $z\in L^{2}(\Lambda)$ instead of $\ell(t,z)$.

\subsection{Local observations}

Throughout this work let $x_{0}\in\Lambda$ be fixed. The following
rescaling will be useful in the sequel: for $z\in L^{2}(\R^{d})$
and $\delta>0$ set
\begin{align*}
\Lambda_{\delta,x_{0}} & :=\delta^{-1}(\Lambda-x_{0})=\{\delta^{-1}(x-x_{0})\,:\,x\in\Lambda\}\text{ and }\Lambda_{0,x_{0}}:=\R^{d},\\
z_{\delta,x_{0}}(x) & :=\delta^{-d/2}z(\delta^{-1}(x-x_{0})),\quad x\in\R^{d}.
\end{align*}

Fix a function
$K\in H^{2}(\R^{d})$, called kernel, with compact support in $\Lambda_{\delta,x_{0}}$.
The compact support ensures that $K_{\delta,x_{0}}$ is localized
around $x_{0}$ and $K_{\delta,x_{0}}\in H_{0}^{1}(\Lambda)\cap H^{2}(\Lambda)$,
$\norm{K_{\delta,x_{0}}}=\norm K_{L^{2}(\R^{d})}$. Local measurements
of $X$ at $x_{0}$ with resolution level $\delta$ until time $T$
are described by the real-valued processes $X_{\delta,x_{0}}=(X_{\delta,x_{0}}(t))_{0\le t\le T}$,
$X_{\delta,x_{0}}^{\Delta}=(X_{\delta,x_{0}}^{\Delta}(t))_{0\le t\le T}$,
\begin{align}
X_{\delta,x_{0}}(t) & =\left\langle X(t),K_{\delta,x_{0}}\right\rangle ,\label{EqXdelta}\\
X_{\delta,x_{0}}^{\Delta}(t) & =\left\langle X(t),\Delta K_{\delta,x_{0}}\right\rangle .\label{EqXdeltaDelta}
\end{align}
Note that it is sufficient to observe $X_{\delta,x}(t)$ for $x$
in a neighbourhood of $x_{0}$ in order to provide us with $X_{\delta,x_{0}}^{\Delta}(t)=\Delta X_{\delta,\cdot}(t)|_{x=x_{0}}$.
Examples for $K$ can be found in Section \ref{sec:Discussion-and-numerical}.

The process $X_{\delta,x_{0}}$ satisfies $X_{\delta,x_{0}}(0)=\langle X_{0},K_{\delta,x_{0}}\rangle$
and
\begin{equation}
dX_{\delta,x_{0}}\left(t\right)=\left\langle X\left(t\right),A_{\theta}^{*}K_{\delta,x_{0}}\right\rangle dt+\norm{B^{*}K_{\delta,x_{0}}}d\overline{W}\left(t\right)\label{eq:tested_eq}
\end{equation}
with the scalar Brownian motion $\overline{W}(t)=\scapro{BW(t)}{K_{\delta,x_{0}}}/\|B^{\ast}K_{\delta,x_{0}}\|$,
whenever $\norm{B^{*}K_{\delta,x_{0}}}>0$.

\section{Scaling assumptions}

\label{sec:3}

\subsection{Rescaled operators and semigroups}

Let us study how $A_{\theta}^{*}$ and $S_{\theta}^{*}(t)$
act on localized functions $z_{\delta,x_{0}}$. For this note first
that $A_{\theta}^{*}=\Delta_{\theta}+A_{0}^{*}$ with $A_{0}^{*}z=-\text{div}(az)+bz$
has domain $\c D(A_{\theta}^{*})=H_{0}^{1}(\Lambda)\cap H^{2}(\Lambda)$.
For $\delta>0$ define similarly the operator $A_{\theta,\delta,x_{0}}^{*}=\Delta_{\theta(x_{0}+\delta\cdot)}+A_{0,\delta,x_{0}}^{*}$
with domain $\c D(A_{\theta,\delta,x_{0}}^{*})=H_{0}^{1}(\Lambda_{\delta,x_{0}})\cap H^{2}(\Lambda_{\delta,x_{0}})$,
where for $z\in C_{c}^{\infty}({\Lambda}_{\delta,x_{0}})$
\begin{align}
A_{0,\delta,x_{0}}^{*}z & =-\delta\text{div}(a(x_{0}+\delta\cdot)z)+\delta^{2}b\left(x_{0}+\delta\cdot\right)z.\label{eq:A_0_star_delta}
\end{align}
The operator $A_{\theta,\delta,x_{0}}^{*}$
generates again an analytic semigroup $(S_{\theta,\delta,x_{0}}^{*}(t))_{t\geq0}$ on $L^{2}(\Lambda_{\delta,x_{0}})$ (Lemma 7.3.4 of \citet{Pazy:1983us}). The following scaling properties are fundamental for our analysis:

\begin{lem}
\label{lem:scaling} For $\delta>0$:
\begin{enumerate}
\item If $z\in H_{0}^{1}(\Lambda_{\delta,x_{0}})\cap H^{2}(\Lambda_{\delta,x_{0}})$,
then $A_{\vartheta}^{*}z_{\delta,x_{0}}=\delta^{-2}(A_{\theta,\delta,x_{0}}^{*}z)_{\delta,x_{0}}$.
\item If $z\in L^{2}(\Lambda_{\delta,x_{0}})$, then $S_{\theta}^{*}\left(t\right)z_{\delta,x_{0}}=(S_{\theta,\delta,x_{0}}^{*}(t\delta^{-2})z)_{\delta,x_{0}}$,
$t\geq0$.
\end{enumerate}
\end{lem}

\begin{proof}
It suffices to prove the result for $z\in C_{c}^{\infty}(\overline{\Lambda}_{\delta,x_{0}})$. In this case, (i) follows immediately, noting that $z_{\delta,x_0}\in C_c^\infty (\overline{\Lambda})$. For (ii) set  $w(t)=(S_{\theta,\delta,x_{0}}^{*}(t\delta^{-2})z)_{\delta,x_{0}}\in L^{2}(\Lambda)$.
As $(S_{\theta,\delta,x_{0}}^{*}(t))_{t\geq0}$ is an analytic semigroup,
we have $S_{\theta,\delta,x_{0}}^{*}(t)z\in\c D(A_{\theta,\delta,x_{0}}^{*})=H_{0}^{1}(\Lambda_{\delta,x_{0}})\cap H^{2}(\Lambda_{\delta,x_{0}})$
and so by (i)
\[
\frac{d}{dt}w(t)=\delta^{-2}(A_{\theta,\delta,x_{0}}^{*}S_{\theta,\delta,x_{0}}^{*}(t\delta^{-2})z)_{\delta,x_{0}}=A_{\theta}^{*}w(t).
\]
Since $w(0)=z_{\delta,x_{0}}$, we conclude that $w(t)=S_{\theta}^{*}(t)z_{\delta,x_{0}}$ from $u(t)=S_{\theta}^{*}(t)z_{\delta,x_{0}}$
being the unique solution in $C([0,\infty);H_{0}^{1}(\Lambda)\cap H^{2}(\Lambda))\cap C^{1}([0,\infty);L^{2}(\Lambda))$
of
\[
\frac{d}{dt}u(t) =A_{\theta}^{*}u(t),\,\,\,t\geq0,\,\,\,\,u(0)=z_{\delta,x_{0}}.\qedhere
\]
\end{proof}
Applying $S_{\theta}^{*}(t)$ to a localized function $z_{\delta,x_{0}}$
is therefore equivalent to applying a different semigroup, rescaled
in time and space, to the fixed function $z$.

\subsection{Scaling of $B$}

Just as with $A_{\theta}^{*}$ we also need that $B^{*}$ behaves
nicely when applied to localized functions. For this we shall assume
a scaling limit for $B^{*}$, which does not degenerate in combination
with $K$.

\begin{assumption} \label{assu:B} There are bounded linear operators
$B_{\delta,x_{0}},B_{0,x_{0}}:L^{2}(\R^{d})\rightarrow L^{2}(\R^{d})$
such that $B^{*}(z_{\delta,x_{0}})=(B_{\delta,x_{0}}^{*}z)_{\delta,x_{0}}$
for $z\in L^{2}(\R^{d})$ with support in $\Lambda_{\delta,x_{0}}$
and $B_{\delta,x_{0}}^{*}z\rightarrow B_{0,x_{0}}^{*}z$ for $z\in L^{2}(\R^{d})$
and $\delta\rightarrow0$. Introducing
\begin{align}
\Psi(z,z') & :=\int_{0}^{\infty}\left\langle B_{0,x_{0}}^{*}e^{s\Delta}z,B_{0,x_{0}}^{*}e^{s\Delta}z'\right\rangle _{L^{2}(\R^{d})}ds,\quad z,z'\in L^{2}(\R^{d}),\label{eq:Psi}
\end{align}
assume the \emph{non-degeneracy conditions} $\norm{B_{0,x_{0}}^{*}K}_{L^{2}(\R^{d})}>0$,
$\Psi(\Delta K,\Delta K)>0$.

\end{assumption}
\begin{rem}
\label{rem:psi} We shall see that after an appropriate rescaling
$\theta(x_{0})^{-1}\Psi(z,z')$ becomes the limiting covariance in
(\ref{eq:cov_function}) (cf. Proposition \ref{prop:cov_asymp} below).
$\Psi(\Delta K,\Delta K)$ is always nonnegative and finite because
\begin{align*}
 & \Psi\left(\Delta K,\Delta K\right)\leq\norm{B_{0,x_{0}}^{*}}_{L^{2}(\R^d)}^{2}\int_{0}^{\infty}\norm{e^{s\Delta}\Delta K}_{L^{2}(\R^{d})}^{2}ds\leq\tfrac{\norm{B_{0,x_{0}}^{*}}_{L^{2}(\R^{d})}^{2}}{2}\norm{\nabla K}_{L^{2}(\R^{d})}^{2},
\end{align*}
using $\norm{e^{s\Delta}\Delta K}_{L^{2}(\R^{d})}^{2}=\sc{e^{2s\Delta}\Delta K}{\Delta K}_{L^{2}(\R^{d})}$
and $\int_{0}^{\infty}e^{2s\Delta}\Delta K\,ds=-\frac{1}{2}K$.
\end{rem}

\begin{examples} \label{exa:B_multiplicationOp}
\mbox{}
\begin{enumerate}[label=(\alph*)]
\item For a bounded continuous function $\sigma:\R^{d}\to(0,\infty)$ define
the multiplication operator $M_{\sigma}:L^{2}(\Lambda)\rightarrow L^{2}(\Lambda)$,
$M_{\sigma}z(x):=(\sigma z)(x)=\sigma(x)z(x)$. With $B=B^{*}=M_{\sigma}$
the SPDE in (\ref{eq:SPDE}) can be written informally as
\[
\dot{X}\left(t,x\right)=A_{\vartheta}X\left(t,x\right)+\sigma(x)\dot{W}\left(t,x\right),\,\,\,\,0<t\le T,\,x\in\Lambda.
\]
Note that $B^*$ commutes with $A_{\theta}$ only if $\sigma$ is constant.
For $z\in L^{2}(\Lambda_{\delta,x_{0}})$ we find that $B^{*}z_{\delta,x_{0}}=(M_{\sigma(\delta\cdot+x_{0})}z)_{\delta,x_{0}}$
and so $B_{\delta,x_{0}}=M_{\sigma(\delta\cdot+x_{0})}$. Then $\norm{B_{\delta,x_{0}}^{*}z-\sigma(x_{0})z}_{L^{2}(\R^{d})}\rightarrow0$
for $z\in L^{2}(\R^{d})$, $\delta\rightarrow0$, and thus $B_{0,x_{0}}^{*}=M_{\sigma(x_{0})}$
is the multiplication operator on $L^{2}(\R^{d})$ with the constant
$\sigma(x_{0})$. For $z\in H^{2}(\R^{d})$, $z'\in L^{2}(\R^{d})$
we have (cf. Remark \ref{rem:psi})
\begin{equation}
\Psi(\Delta z,z')=-\frac{\sigma^{2}(x_{0})}{2}\scapro{z}{z'}_{L^{2}(\R^{d})}\label{eq:Psi_Delta_example}
\end{equation}
and integration by parts shows $\Psi(\Delta K,\Delta K)=\frac{\sigma^{2}(x_{0})}{2}\norm{\nabla K}_{L^{2}(\R^{d})}^{2}$.
The non-degeneracy conditions are clearly satisfied.
\item Let $\sigma$ be as in (a) and consider with bounded $\eta\in C^{2}(\R^{d})$,
$\min_{x\in\R^{d}}\eta(x)>0$, the perturbed multiplication operator
$B=B^{*}=M_{\sigma}+(-\Delta_{\eta})^{-\gamma}$, $\gamma>0$. By
functional calculus $B^{*}z_{\delta,x_{0}}=(B_{\delta,x_{0}}^{*}z)_{\delta,x_{0}}$
for $z\in L^{2}(\Lambda_{\delta,x_{0}})$ with $B_{\delta,x_{0}}=M_{\sigma(\delta\cdot+x_{0})}+\delta^{2\gamma}(-\Delta_{\eta(\delta\cdot+x_{0})})^{-\gamma}$
and $\norm{B_{\delta,x_{0}}^{*}z-\sigma(x_{0})z}_{L^{2}(\R^{d})}\rightarrow0$
for $z\in L^{2}(\R^{d})$, $\delta\rightarrow0$. $B_{0,x_{0}}$ and
$\Psi(\Delta K,\Delta K)$ are as in (a).
\item Assumption \ref{assu:B} excludes $B=(-\Delta)^{-\gamma}$, $\gamma>0$,
a typical choice to obtain smooth solutions $X$, cf. \citet[Chapter 5.5]{DaPrato:2014wq}.
Indeed, by (b) $B_{\delta,x_{0}}^{*}=\delta^{2\gamma}(-\Delta)^{-\gamma}$
and so $B_{0,x_{0}}^{*}=0$, violating the non-degeneracy conditions.
This problem can be solved by modifying the test function $K_{\delta,x_{0}}$.
For example, if $A_{\theta}=\theta\Delta$ for constant $\theta>0$
and $X_{0}\in\c D((-\Delta)^{\gamma})$, then assume we have access to $\sc{X(t)}{(-\Delta)^{\gamma}K_{\delta,x_{0}}}$,
$\sc{X(t)}{(-\Delta)^{\gamma}\Delta K_{\delta,x_{0}}}$ instead of (\ref{EqXdelta}),
(\ref{EqXdeltaDelta}). Since $B$ and $A_{\theta}$ commute, $\sc{X(\cdot)}{(-\Delta)^{\gamma}K_{\delta,x_{0}}}$
has the same distribution as $\sc{\tilde{X}(\cdot)}{K_{\delta,x_{0}}}$,
where $\tilde{X}$ corresponds to the SPDE (\ref{eq:SPDE}) with $B=I$
and $\tilde{X}_{0}=(-\Delta)^{\gamma}X_{0}$, and so Assumption \ref{assu:B}
is satisfied.
\end{enumerate}
\end{examples}

\subsection{\label{subsec:From-bounded-to}From bounded to unbounded domains}

Lemma \ref{lem:scaling} and Assumption \ref{assu:B} allow us to
rewrite the covariance function of $X_{\delta,x_{0}}$ for $t,t'\geq0$:
\begin{align}
 & \text{Cov}(\delta^{-1}X_{\delta,x_{0}}(t\delta^{2}),\delta^{-1}X_{\delta,x_{0}}(t'\delta^{2}))\nonumber \\
 & =\int_{0}^{t\wedge t'}\left\langle B_{\delta,x_{0}}^{*}S_{\theta,\delta,x_{0}}^{*}\left(t-s\right)K,B_{\delta,x_{0}}^{*}S_{\theta,\delta,x_{0}}^{*}\left(t'-s\right)K\right\rangle _{L^{2}(\Lambda_{\delta,x_{0}})}ds.
\label{eq:scaled_cov_function}
\end{align}
In order to see how this behaves when $\delta\rightarrow0$, note that the domain $\Lambda_{\delta,x_{0}}$ grows
and we find from (\ref{eq:A_0_star_delta}) that $A_{\theta,\delta,x_{0}}^{*}K\rightarrow\theta(x_{0})\Delta K$
in $L^{2}(\R^{d})$. This motivates the following result, proved in Appendix \ref{subsec:Analytical-results}.

\begin{prop}
\label{prop:semigroup_convergence} For $t>0$:
\begin{enumerate}
\item If $\delta>0$ and $z\in C(\overline{\Lambda}_{\delta,x_{0}})$, then
$|(S_{\theta,\delta,x_{0}}^{*}\left(t\right)z)(x)|\leq c_{3}e^{c_{1}\delta^{2}t}(e^{c_{2}t\Delta}\left|z\right|)(x)$
for all $x\in\Lambda_{\delta,x_{0}}$ with universal constants $c_{1},c_{2},c_{3}>0$.
\item If $z\in L^{2}(\R^{d})$, then $S_{\theta,\delta,x_{0}}^{*}\left(t\right)(z|_{\Lambda_{\delta,x_{0}}})\rightarrow e^{\theta(x_{0})t\Delta}z$
in $L^{2}(\R^{d})$ for $\delta\rightarrow0$.
\end{enumerate}
\end{prop}

This means that the solution of
\begin{align*}
\frac{d}{dt}u^{(\delta)}(t) & =(A_{\theta,\delta,x_{0}}^{*}u^{(\delta)})(t),\,\,\,\,u^{(\delta)}(0)=z,
\end{align*}
on $L^{2}(\Lambda_{\delta,x_{0}})$ with \emph{bounded domain} $\Lambda_{\delta,x_{0}}$
converges to the solution of the heat equation
\[
\frac{d}{dt}u(t)=\theta(x_{0})\Delta u(t),\,\,\,\,u(0)=z,
\]
on $L^{2}(\R^{d})$ with \emph{unbounded domain} $\R^{d}$. This
scaling limit,
which seems natural but is nevertheless non-trivial, lies at the heart
of the analysis for the covariance function. Yet, the convergence in Proposition \ref{prop:semigroup_convergence}(ii)
does not hold uniformly in $z$, which complicates the approximations in the covariance analysis.

\begin{comment}

To see this, let $A_{\theta}=\Delta+bI$ and $\tilde{A}_{1}=\Delta$
for constant $b\neq0$. Then
\begin{align*}
 & \sup_{z\in L^{2}(\Lambda_{\delta,x_{0}})}\norm{(A_{\theta,\delta,x_{0}}^{*}-\tilde{A}_{1,\delta,x_{0}}^{*})(-\tilde{A}_{1,\delta,x_{0}}^{*})^{-1}z}_{L^{2}(\Lambda_{\delta,x_{0}})}\\
 & =\delta^{2}|b|\sup_{z\in L^{2}(\Lambda_{\delta,x_{0}})}\norm{(-\tilde{A}_{1,\delta,x_{0}}^{*})^{-1}z}_{L^{2}(\Lambda_{\delta,x_{0}})}\geq|b||\lambda|^{-1},
\end{align*}
by choosing $z=\delta^{d/2}e_{\lambda}(\delta\cdot+x_{0})\in L^{2}(\Lambda_{\delta,x_{0}})$
for an eigenpair $(\lambda,e_{\lambda})$ of $\tilde{A}_{1}$, such
that $\tilde{A}_{1}z_{\delta,x_{0}}=\lambda z_{\delta,x_{0}}$ and
$\tilde{A}_{1,\delta,x_{0}}z=(\delta^{2}\lambda)z$ from Lemma \ref{lem:scaling}.{\color{red}So what?}
\end{comment}

Applying the proposition to \eqref{eq:scaled_cov_function} also implies a scaling limit for the SPDE in (\ref{eq:SPDE}),
where for simplicity a zero initial condition is assumed:
\begin{thm}
\label{thm:convToHeatEq} Let $X_{0}=0$ and set $Z_{\delta}(t,z):=\delta^{-1}\sc{X(t\delta^{2})}{(z|_{\Lambda_{\delta,x_{0}}})_{\delta,x_{0}}}$
for $t\geq0$, $z\in L^{2}(\R^{d})$. Under Assumption \ref{assu:B}
the finite dimensional distributions of $(Z_{\delta}(t,z))_{t\ge0,z\in L^{2}(\R^{d})}$
converge to those of $(Z_{0}(t,z))_{t\geq0,z\in L^{2}(\R^{d})}$,
$Z_{0}(t,z)=\sc{Y(t)}z_{L^{2}(\R^{d})}$, solving the stochastic heat
equation on $L^{2}(\R^{d})$ with space-time white noise $dW$ on $L^{2}(\R^{d})$:
\begin{align*}
\begin{cases}
dY\left(t\right)=\theta(x_{0})\Delta Y(t)dt+B_{0,x_{0}}dW(t),\,\,\,\,t>0,\\
Y\left(0\right)=0.
\end{cases}
\end{align*}
\end{thm}

\begin{proof}
According to \eqref{eq:scaled_cov_function} $Z_{\delta}$
is a centered Gaussian process with covariance function $\text{Cov}(Z_{\delta}(t,z),Z_{\delta}(t',z'))$
for $t,t'\geq0$, $z,z'\in L^{2}(\R^{d})$ equal to
\begin{align*}
 & \int_{0}^{t\wedge t'}\langle B_{\delta,x_{0}}^{*}S_{\theta,\delta,x_{0}}^{*}(t-s)(z|_{\Lambda_{\delta,x_{0}}}),B_{\delta,x_{0}}^{*}S_{\theta,\delta,x_{0}}^{*}(t'-s)(z'|_{\Lambda_{\delta,x_{0}}})\rangle _{L^{2}(\Lambda_{\delta,x_{0}})}ds.
\end{align*}
It is enough to show that this converges to
\[
\Cov(Z_{0}(t,z),Z_{0}(t',z'))=\int_{0}^{t\wedge t'}\sc{B_{0,x_{0}}^{*}e^{\theta(x_{0})(t-s)\Delta}z}{B_{0,x_{0}}^{*}e^{\theta(x_{0})(t'-s)\Delta}z'}_{L^{2}(\R^{d})}ds.
\]
Approximating $z$ by continuous functions, the semigroup bound in Proposition \ref{prop:semigroup_convergence}(i)
gives $\sup_{0<\delta\leq 1}\sup_{s\leq t}\norm{S_{\theta,\delta,x_{0}}^{*}(s)(z|_{\Lambda_{\delta,x_0}})}_{L^{2}(\Lambda_{\delta,x_0})}<\infty$,
while Assumption \ref{assu:B} implies $B_{\delta,x_0}^*u^{(\delta)}\to B_{0,x_0}^*u$ for any $u^{(\delta)}\to u$, invoking the uniform boundedness principle.
By Proposition \ref{prop:semigroup_convergence}(ii)
we have $S_{\theta,\delta,x_{0}}^{*}(s)(z|_{\Lambda_{\delta,x_{0}}})\rightarrow e^{\theta(x_{0})s\Delta}z$
in $L^{2}(\R^{d})$. Arguing in the same way with respect to $z'$, the dominated
convergence theorem shows the claim.
\end{proof}

This theorem demonstrates the strength of local measurements that
at small scales  the highest order differential operator dominates,
together with the local coefficient $\theta(x_{0})$ and the local
operator $B_{0,x_{0}}$ in the noise.

\subsection{The initial condition}

For $X_0$ we require the following scaling behaviour:
\begin{assumptionWithArgs}{z}{\beta}\label{assu:X0} For $\beta>0$
and $z\in H^{2}(\R^{d})$ with compact support in $\Lambda_{\delta,x_{0}}$
for $\delta>0$, the initial condition $X_{0}$ satisfies
\[
\int_{0}^{T}\langle S_{\theta}(t)X_{0},(\Delta z)_{\delta,x_{0}}\rangle^{2}dt=o(\ell_{d,2}(\delta)^{-1}\delta^{\beta}),\,\,\,\,\delta\rightarrow0,
\]
where $\ell_{d,2}(\delta)=\log(\delta^{-1})$ for $d=2$ and $\ell_{d,2}(\delta)=1$
otherwise.

\end{assumptionWithArgs}

Under this assumption the initial condition becomes negligible in the estimation procedure. It is true under general conditions.

\begin{lem}
\label{rem:X0_holds} Assumption \aswithargsref{assu:X0}{z}{\beta}
is satisfied for all $z\in H^{2}(\R^{d})$ with compact support in
$\Lambda_{\delta,x_{0}}$ for $\delta>0$ and
\begin{enumerate}
\item $\beta=2$ if $X_{0}\in L^{p}(\Lambda)$ for some $p>2$, in particular
if $X_{0}\in C(\overline{\Lambda})$,
\item $\beta=3$ if $X_{0}\in\c D(A_{\theta})$.
\end{enumerate}
\begin{proof}
This follows from Lemma \ref{lem:X0}(ii,iii) below, noting $\gamma(d,p)>2$ in (ii)  for $p>2$, $d\ge 1$.
\end{proof}
\end{lem}

\begin{comment}
\begin{rem}
$X_0\in L^2(\Lambda)$ is not enough for Assumption \aswithargsref{assu:X0}{z}{\beta} to hold. For example, for $d=1$, $\Lambda=(0,1)$ set
\[
X_0(x)=|x-x_0|^{-1/2}(|\log(x-x_0)|+1)^{-1}\I_{\{|x-x_0|\leq \eps\}},\,\,\,x\in\Lambda,
\]
for some small $\eps > 0$. Then $X_0\in L^2(\Lambda)$, $X_0\notin L^p(\Lambda)$ for any $p>2$, but $\langle S_{\theta}(t)X_{0},(\Delta z)_{\delta,x_{0}}\rangle^{2}\geq C_{T,z}\delta^{2}$ for a constant depending on $T,z$.
\end{rem}
\end{comment}

\section{\label{sec:Construction of the estimators}The two estimation methods}

\begin{comment}
 (Implicit ideas that should be made explicit and appear in one place, for instance right here:
\begin{itemize}
\item We derive two different estimators, the proxy MLE, using the local
measurement $X_{\delta,x_{0}}$, and the augmented MLE, using both
measurement processes $X_{\delta,x_{0}}$ and $X_{\delta,x_{0}}^{\Delta}$.
(already mentioned in intro and also implicitly below).
\item Both are consistent and satisfy central limit theorems, even in the
nonparametric setup and with general $A_{\theta}$ and $B$ (mentioned
below as well).
\item Neither the augmented nor the proxy MLE are\emph{ }actual\emph{ }MLEs.
\item The analysis of the augmented MLE is easier and it has smaller asymptotic
variance compared to the proxy MLE (intuitively clear, it uses more
information).
\item Still, the proxy MLE seems more desirable, as it uses less information.
One particular issue is, however, that it requires explicit knowledge
of $K$.
\end{itemize}
\end{comment}

\subsection{Motivation and construction}

We give two motivations for deriving the estimators in the parametric
case $A_{\theta}=\theta\Delta$ with constant $\theta>0$, $B=I$.
As we shall see later, these estimators will then work quite universally
for nonparametric specifications of $\theta$ and general $A_{\theta}$
and $B$.

\paragraph*{Least squares approach.}

In the deterministic situation of (\ref{eq:tested_eq}) without driving
noise (i.e. $A_{\theta}=\theta\Delta$ and $B=0$) we recover $\theta$
via $\dot{X}_{\delta,x_{0}}(t)=\theta X_{\delta,x_{0}}^{\Delta}(t)$
for all $t\in[0,T]$. A standard least-squares ansatz in the noisy
situation would therefore lead to an estimator $\hat{\theta}=\argmin_{\theta}\int_{0}^{T}(\dot{X}_{\delta,x_{0}}(t)-\theta X_{\delta,x_{0}}^{\Delta}(t))^{2}dt$.
While this itself is certainly not well defined, the corresponding
normal equations yield the feasible estimator
\[
\hat{\theta}_{\delta}^{LS}=\frac{\int_{0}^{T}X_{\delta,x_{0}}^{\Delta}(t)dX_{\delta,x_{0}}(t)}{\int_{0}^{T}X_{\delta,x_{0}}^{\Delta}(t)^{2}dt},
\]
compare with the approach by \citet{Maslowski:2013epa} for fractional
noise.

\paragraph*{Likelihood approach.}

Assume that only $X_{\delta,x_{0}}$ in \eqref{eq:tested_eq} is observed with $A_\theta=\theta\Delta$, $B=I$. Denote by $\P_{\theta}^{\delta,x_{0}}$ and
$\P_{0}$ the laws of $X_{\delta,x_{0}}$ and $\norm K_{L^{2}(\R^{d})}\overline{W}$
on the canonical path space $(C([0,T]),\norm{\cdot}_{\infty})$
equipped with its Borel sigma algebra. Typically, the likelihood of
$\P_{\theta}^{\delta,x_{0}}$ with respect to $\P_{0}$ is determined
via Girsanov's theorem. This is not immediate from (\ref{eq:tested_eq}),
because $X_{\delta,x_{0}}^{\Delta}$ cannot be obtained from $X_{\delta,x_{0}}$
for fixed $x_{0}$. Therefore we employ \citet[Theorem 7.17]{Liptser2001}
and write $X_{\delta,x_{0}}$ as the diffusion-type process
\[
dX_{\delta,x_{0}}(t)=\theta m_{\theta}(t)dt+\norm K_{L^{2}(\R^{d})}d\tilde{W}(t),\quad t\in[0,T],
\]
with a different scalar Brownian motion $\tilde{W}=(\tilde{W}(t))_{0\le t\le T}$,
adapted to the filtration generated by $X_{\delta,x_{0}}$, and
\begin{align*}
m_{\theta}(t) & =\E_{\theta}\left[\l{X_{\delta,x_{0}}^{\Delta}(t)}(X_{\delta,x_{0}}(s))_{0\leq s\leq t}\right].
\end{align*}
Girsanov's theorem in the form of \citet[Theorem 7.18]{Liptser2001}
applies and we find that $\P_{\theta}^{\delta,x_{0}}$ has with respect
to $\P_{0}$ the likelihood
\begin{align*}
\c L\left(\theta,X_{\delta,x_{0}}\right) & =\exp\bigg(\theta\int_{0}^{T}\frac{m_{\theta}(t)}{\norm K_{L^{2}(\R^{d})}^{2}}dX_{\delta,x_{0}}\left(t\right)-\frac{\theta^{2}}{2}\int_{0}^{T}\frac{m_{\theta}(t)^{2}}{\norm K_{L^{2}(\R^{d})}^{2}}dt\bigg).
\end{align*}
Computing the conditional expectation $m_{\theta}(t)$ is a non-explicit
filtering problem, even in the parametric case $A_{\theta}=\theta\Delta$.
In particular, $m_{\theta}$ depends on $\theta$ in a highly nonlinear
way. We pursue two different modifications:

\paragraph*{Augmented MLE.}

If we observe $X_{\delta,x_{0}}^{\Delta}$ additionally, then we can
just replace the conditional expectation $m_{\theta}(t)$ in the likelihood
by its argument $X_{\delta,x_{0}}^{\Delta}(t)$, which is in particular
independent of $\theta$. Maximizing this modified likelihood leads
to the augmented MLE
\begin{equation}
\hat{\vartheta}_{\delta}^{A}=\frac{\int_{0}^{T}X_{\delta,x_{0}}^{\Delta}(t)dX_{\delta,x_{0}}\left(t\right)}{\int_{0}^{T}X_{\delta,x_{0}}^{\Delta}(t)^{2}dt}.\label{eq:augMLE}
\end{equation}
We remark that $\hat{\vartheta}_{\delta}^{A}=\hat{\theta}_{\delta}^{LS}$.

\paragraph*{Proxy MLE.}

If we do not dispose of additional observations, we can approximate
$m_{\theta}(t)$ by the conditional expectation $\E_{\theta}[X_{\delta,x_{0}}^{\Delta}(t)\,|\,X_{\delta,x_{0}}(t)]$.
In our simplified setup with $A_{\theta}=\theta\Delta$ and $B=I$
the projected finite-dimensional process $(\sc{X(t)}{z_i})_{1\le i\le m}$ for $z_i\in L^{2}(\Lambda)$ admits a stationary solution
$\sc{X(t)}{z_i}=\int_{-\infty}^{t}\sc{S_{\theta}(t-s){z_i}}{dW(s)}$, $i=1,\ldots,m$,
with a two-sided cylindrical Brownian motion $(W(t))_{t\in\R}$, provided
the variances remain finite. Note that we need not require $X_0=\int_{-\infty}^{t}S_{\theta}^*(t-s)dW(s)$ to exist, but only that the finite-dimensional projection $(\sc{X_0}{z_i})_{1\le i\le m}$ follows the right law, which is always feasible.  If we choose $z_1=K_{\delta,x_0}$, $z_2=\Delta K_{\delta,x_0}$, then the process $(X_{\delta,x_{0}},X_{\delta,x_{0}}^{\Delta})$ is stationary with
\begin{align}
\Var(X_{\delta,x_{0}}(t)) & =\int_{-\infty}^{t}\scapro{S_{\theta}(2t-2s)K_{\delta,x_{0}}}{K_{\delta,x_{0}}}ds\nonumber \\
 & =\frac{1}{2\theta}\scapro{(-\Delta)^{-1}K_{\delta,x_{0}}}{K_{\delta,x_{0}}},\label{eq:varstat}\\
\Cov(X_{\delta,x_{0}}^{\Delta}(t),X_{\delta,x_{0}}(t)) & =\int_{-\infty}^{t}\scapro{S_{\theta}(2t-2s)\Delta K_{\delta,x_{0}}}{K_{\delta,x_{0}}}ds=\frac{-1}{2\theta}\norm K_{L^{2}(\R^{d})}^{2}.\nonumber
\end{align}
In general, $\scapro{(-\Delta)^{-1}K_{\delta,x_{0}}}{K_{\delta,x_{0}}}$
may not exist, but if we assume the existence of $\tilde{K}\in H^{4}(\R^{d})$
with $\Delta\tilde{K}=K$ and compact support in $\Lambda_{\delta,x_{0}}$,
then by the scaling in Lemma \ref{lem:scaling}, $\Var(X_{\delta,x_{0}}(t))=\frac{\delta^{2}}{2\theta}\norm{\nabla\tilde{K}}_{L^{2}(\R^{d})}^{2}<\infty$
follows. In this situation we therefore find that $\E_{\theta}[X_{\delta,x_{0}}^{\Delta}(t)\,|\,X_{\delta,x_{0}}(t)]$
equals
\[
\frac{\Cov(X_{\delta,x_{0}}^{\Delta}(t),X_{\delta,x_{0}}(t))}{\Var(X_{\delta,x_{0}}(t))}X_{\delta,x_{0}}(t)=\frac{-\delta^{-2}\norm K_{L^{2}(\R^{d})}^{2}}{\norm{\nabla\tilde{K}}_{L^{2}(\R^{d})}^{2}}X_{\delta,x_{0}}(t).
\]
This expression is again independent of $\theta$. Using it as an
approximation of $m_{\theta}(t)$ in the likelihood and neglecting
the boundary terms in the identity $2\int_{0}^{T}X_{\delta,x_{0}}(t)dX_{\delta,x_{0}}\left(t\right)=(X_{\delta,x_{0}}^{2}(T)-X_{\delta,x_{0}}^{2}(0))-\langle X_{\delta,x_{0}}\rangle_{T}$
with quadratic variation $\langle X_{\delta,x_{0}}\rangle_{T}$, we
obtain the proxy MLE
\begin{equation}
\hat{\theta}_{\delta}^{P}=\frac{\norm{\nabla\tilde{K}}_{L^{2}(\R^{d})}^{2}}{2\norm K_{L^{2}(\R^{d})}^{2}}\frac{\langle X_{\delta,x_{0}}\rangle_{T}}{\delta^{-2}\int_{0}^{T}X_{\delta,x_{0}}(t)^{2}dt}.\label{eq:proxyMLE}
\end{equation}
Note that the quadratic variation $\langle X_{\delta,x_{0}}\rangle_{T}=T\|B^*K_{\delta,x_0}\|^2$
is known to us from observing $X_{\delta,x_{0}}$ continuously in
time.
\begin{rem}
A sufficient condition for the existence of $\tilde{K}$ is $\int_{\R^{d}}K(x)dx=0$,
$\int_{\R^{d}}xK(x)dx=0$ by Lemma \ref{lem:v0}(iii) below.
\end{rem}

\subsection{Basic error decomposition }

\label{sec:erranalysis}

Let us discuss the basic error analysis for the augmented MLE $\hat{\vartheta}_{\delta}^{A}$
and the proxy MLE $\hat{\vartheta}_{\delta}^{P}$ in the general nonparametric
framework of Section \ref{sec:2}. Since we only use local measurements
around $x_{0}$, it might be expected that asymptotically we are lead to estimating
$\theta(x_{0})$. Let us point out that this is indeed true, but a priori not  clear because all values of $\theta(x)$ enter into the observations $X_{\delta,x_0}$ and it must be excluded that the resulting bias spoils the estimator.

\paragraph*{Augmented MLE.}

Consider $\hat{\vartheta}_{\delta}^{A}(x_{0})=\hat{\vartheta}_{\delta}^{A}$
from (\ref{eq:augMLE}). Then insertion of Equation (\ref{eq:tested_eq})
for $dX_{\delta,x_{0}}(t)$ yields the decomposition
\begin{equation}
\hat{\vartheta}_{\delta}^{A}(x_{0})=\vartheta\left(x_{0}\right)+\norm{B^{*}K_{\delta,x_{0}}}(\c I_{\delta}^{A})^{-1}M_{\delta}^{A}+(\c I_{\delta}^{A})^{-1}R_{\delta}^{A},\label{eq:augm_decomp}
\end{equation}
with
\begin{align*}
M_{\delta}^{A} & =\int_{0}^{T}X_{\delta,x_{0}}^{\Delta}(t)d\overline{W}(t)\text{ (martingale part),}\\
\c I_{\delta}^{A} & =\int_{0}^{T}X_{\delta,x_{0}}^{\Delta}(t)^{2}dt,\text{ (observed Fisher information),}\\
R_{\delta}^{A} & =\int_{0}^{T}X_{\delta,x_{0}}^{\Delta}(t)\left\langle X(t),\left(A_{\vartheta}^{*}-\theta\left(x_{0}\right)\Delta\right)K_{\delta,x_{0}}\right\rangle dt\text{ (remaining bias).}
\end{align*}
Let us note that $\c I_{\delta}^{A}$ is not the observed Fisher information
in a strict sense (due to the appearance of $m_{\theta}$ in the likelihood),
but it plays the same role, compare the analysis of the MLE for Ornstein-Uhlenbeck
processes in \citet{kutoyants2013statistical}. In particular, it
forms the quadratic variation of the martingale $M_{\delta}^{A}$.
In the specific case $A_{\theta}=\theta\Delta$ for some parametric
$\theta>0$ the term $R_{\delta}^{A}$ vanishes, otherwise it induces
a bias due to the variations of $\theta$ around $\theta(x_{0})$
and due to first and zero order differential operators that may appear
in $A_{\theta}$.

As the error structure suggests, the augmented MLE $\hat{\vartheta}_{\delta}^{A}(x_{0})$
is a consistent estimator for $\delta\to0$ if the observed Fisher
information satisfies $\c I_{\delta}^{A}\to\infty$. In the simple
stationary case of (\ref{eq:varstat}) we obtain $\E[\c I_{\delta}^{A}]=\frac{T}{2\theta}\scapro{(-\Delta)K_{\delta,x_{0}}}{K_{\delta,x_{0}}}$,
which by the scaling properties is of order $\delta^{-2}$. Physically,
this can be interpreted as an increase in energy in $X_{\delta,x_{0}}^{\Delta}$
under $\delta$-localisation due to the Laplacian in the drift, while
the energy from the space-time white noise remains unchanged. This
is in fact the same phenomenon as the increasing signal-to-noise ratio
for high Fourier modes in the spectral approach by \citet{Huebner1995}.

\paragraph*{Proxy MLE.}

Consider $\hat{\vartheta}_{\delta}^{P}(x_{0})=\hat{\vartheta}_{\delta}^{P}$
from (\ref{eq:proxyMLE}). The only stochastic part is
\begin{equation}
\c I_{\delta}^{P}:=\delta^{-2}\int_{0}^{T}X_{\delta,x_{0}}(t)^{2}dt\label{eq:IP}
\end{equation}
in the denominator. In the general model (\ref{eq:tested_eq}) we
shall see that $\c I_{\delta}^{P}$ converges to $\theta(x_{0})^{-1}T\Psi(K,K)$,
compare also Remark \ref{rem:psi} with $K=\Delta\tilde{K}$. Asking
for consistency $\hat{\vartheta}_{\delta}^{P}(x_{0})\to\theta(x_{0})$
leads to requiring the identity $\norm{\nabla\tilde{K}}_{L^{2}(\R^{d})}^{2}\norm{B_{0,x_{0}}^{*}K}_{L^{2}(\R^{d})}^{2}=2\norm K_{L^{2}(\R^{d})}^{2}\Psi(K,K)$.
This does not hold for any operator $B_{0,x_{0}}$. We therefore restrict
to our main specification $B=M_{\sigma}$, for which the identity
holds by (\ref{eq:Psi_Delta_example}). In contrast to the augmented
MLE, the proxy MLE works with the observation of $X_{\delta,x_{0}}$
alone, but asks for new structural assumptions on $B$ and $K$. If
they are not fulfilled, other likelihood approximations should be
pursued. Compare also the suboptimal behaviour of $\hat{\vartheta}_{\delta}^{P}(x_{0})$
for the kernel $K^{(2)}$ in the simulations of Section \ref{sec:Discussion-and-numerical}
below.

\begin{comment}
(Moved to the proof of the Theorem.) The main result for the proxy
MLE below is based on deriving a CLT for the quadratic functional
$\c I_{\delta}^{P}$ by very precise asymptotic moment calculations
and the \textit{fourth moment theorem} in Wiener chaos by \citet{Nualart:2005by}.
Fundamental for this analysis is that $X_{\delta,x_{0}}(t)$ and $X_{\delta,x_{0}}(s)$
quickly decorrelate as $\delta^{-2}\abs{t-s}\to\infty$, which is
also predicted by the scaling limit in Corollary \ref{thm:convToHeatEq}.
Finally, a CLT for $\hat{\vartheta}_{\delta}^{P}(x_{0})$ is easily
deduced via the delta method. Remark that this method of proof might
also cover time-discrete observations of $X_{\delta,x_{0}}$ if the
sampling frequency increases sufficiently fast as $\delta\to0$, but
this is not pursued here.
\end{comment}

\section{\label{sec:Main-results}Main results}

\subsection{\label{subsec:Augmented-MLE}Results for the augmented MLE}

Recall the function $\Psi$ from (\ref{eq:Psi}) and
the error decomposition (\ref{eq:augm_decomp}). We show first that
the observed Fisher information and the bias, after rescaling, converge
to deterministic quantities. The propositions are proved in Appendix \ref{app:1}.

\begin{prop}
\label{prop:pseudo_Fisher_info_A} Grant Assumptions \ref{assu:B}
and \aswithargsref{assu:X0}{K}{2}. Then for any $d\geq1$ as
$\delta\rightarrow0$
\begin{align*}
\delta^{2}\E[\c I_{\delta}^{A}] & \rightarrow T\theta(x_{0})^{-1}\Psi(\Delta K,\Delta K),\qquad
\c I_{\delta}^{A}/\E[\c I_{\delta}^{A}]  \r{\P}1.
\end{align*}
\end{prop}

\begin{prop}
\emph{\label{prop:bias_A_with_delta}} Grant Assumptions \ref{assu:B},
\aswithargsref{assu:X0}{K}{2}, and for $d=1$ assume $\theta\in C^{1+\alpha'}(\overline{\Lambda})$ for $\alpha'>1/2$
and $\int_{\R}K(x)dx=0$. Then for $\delta\rightarrow0$
\begin{align*}
\delta^{-1}(\c I_{\delta}^{A})^{-1}R_{\delta}^{A} & \r{\P}\mu^A\text{ with }\mu^{A}:=(\Psi(\Delta K,\Delta K))^{-1}\Psi(\Delta K,\beta),
\end{align*}
where $\beta(x)=\Delta(\scapro{\nabla\theta(x_{0})}{x}_{\R^{d}}K)(x)-\scapro{\nabla\theta(x_{0})-a(x_{0})}{\nabla K(x)}_{\R^{d}},$
 $x\in\R^{d}$.
\end{prop}

From this it follows that the augmented MLE $\hat{\vartheta}_{\delta}^{A}(x_{0})$
satisfies a central limit theorem with rate $\delta$.
\begin{thm}
\label{thm:augMLE_CLT} Grant Assumptions \ref{assu:B}, \aswithargsref{assu:X0}{K}{2},
and for $d=1$ assume $\theta\in C^{1+\alpha'}(\overline{\Lambda})$ for $\alpha'>1/2$ and $\int_{\R}K(x)dx=0$.
Then for $\delta\rightarrow0$
\begin{align}
\delta^{-1}\left(\hat{\vartheta}_{\delta}^{A}\left(x_{0}\right)-\theta\left(x_{0}\right)\right) & \r dN\left(\mu^{A},\theta(x_{0})\Sigma^{A}\right),\nonumber \\
\text{with } & \Sigma^{A}=T^{-1}(\Psi(\Delta K,\Delta K))^{-1}\norm{B_{0,x_{0}}^{*}K}_{L^{2}(\R^{d})}^{2},\label{eq:Sigma_A}
\end{align}
and with $\mu^{A}$ from Proposition \ref{prop:bias_A_with_delta}.
\end{thm}

\begin{proof}
In terms of $Y_{t}^{(\delta)}:=X_{\delta,x_{0}}^{\Delta}(t)/\E[\c I_{\delta}^{A}]^{1/2}$
we obtain $M_{\delta}^{A}/\E[\c I_{\delta}^{A}]^{1/2}=\int_{0}^{T}Y_{t}^{(\delta)}d\bar{W}(t)$,
the quadratic variation of which satisfies $\int_{0}^{T}(Y_{t}^{(\delta)})^{2}dt=\c I_{\delta}^{A}/\E[\c I_{\delta}^{A}]\r{\P}1$
by Proposition \ref{prop:pseudo_Fisher_info_A}. A standard continuous
martingale CLT, e.g. \citet[Theorem 1.19]{kutoyants2013statistical},
shows $M_{\delta}^{A}/\E[\c I_{\delta}^{A}]^{1/2}\xrightarrow{d}N(0,1)$.
Moreover,
\begin{equation}
\norm{B^{*}K_{\delta,x_{0}}}=\norm{B_{\delta,x_{0}}^{*}K}_{L^{2}(\Lambda_{\delta,x_{0}})}\to\norm{B_{0,x_{0}}^{*}K}_{L^{2}(\R^{d})}\label{eq:B_star_convergence}
\end{equation}
due to Assumption \ref{assu:B} and $\delta^{-1}(\c I_{\delta}^{A})^{-1}R_{\delta}^{A}\r{\P}\mu^{A}$
by Proposition \ref{prop:bias_A_with_delta}. We conclude by applying
Slutsky's lemma.
\end{proof}

\begin{remarks}\mbox{}
\begin{enumerate}

\item Both, bias and standard deviation of
$\hat{\vartheta}_{\delta}^{A}(x_{0})$, are of order $\delta$. The
asymptotic bias $\mu^{A}$ is independent of $T$, while the variance
$\Sigma^{A}$ decays in $T$.

\item $B$, $\nabla\theta$ and $a$ appear
in the limit only via the localized terms $B_{0,x_{0}}^{*}$, $\nabla\theta(x_{0})$,
$a(x_{0})$, while $b$ does not appear at all. This demonstrates
again the universality property of local measurements, in the spirit of
Theorem \ref{thm:convToHeatEq}.

\item The estimator and thus also its asymptotic bias and variance are invariant under constant scaling factors in the kernel. In fact, using the scaling such that $\norm{K_{\delta,x_0}}=\norm{K}_{L^2(\R^d)}$ is arbitrary, but simplifies the analysis.

\item The additional assumptions for the convergence of the remaining bias $R^{A}_{\delta}$ in $d=1$ allow for compensating the slower decay of the heat kernel compared to $d>1$, cf. Lemma A.6(ii,iii). This is not necessary for constant $\theta$ and in that case Theorem $\ref{thm:augMLE_CLT}$ holds without these assumptions.

\item When we dispose of observations at different locations $x$, then we can estimate $\theta(x)$ pointwise at each location $x$. In the case of multiplicative covariance $B=M_\sigma$ it can be shown that estimators at different locations become asymptotically independent. The argument relies on a multivariate martingale difference CLT, using that at points $x_0,x_1\in\Lambda$ the corresponding Brownian motions $\overline W_0,\overline W_1$ in \eqref{eq:tested_eq} are independent whenever $\supp(K_{\delta,x_0})\cap\supp(K_{\delta,x_1})=\varnothing$.
\end{enumerate}
\end{remarks}

From Proposition \ref{prop:bias_A_with_delta}
we see that $\mu^{A}$ vanishes if $A_{\theta}=\theta\Delta+b$ for
parametric $\theta>0$. Another important situation where $\mu^{A}=0$
is given next.
\begin{example}
\label{exa:cont} (Example \ref{exa:B_multiplicationOp}(a) ctd.)
Let $B=M_{\sigma}$ and recall the identities $\Psi(\Delta K,\Delta K)=\frac{\sigma(x_{0})^{2}}{2}\norm{\nabla K}_{L^{2}(\R^{d})}^{2}$,
$\Psi(\Delta K,\beta)=-\frac{\sigma(x_{0})^{2}}{2}\sc K{\beta}_{L^{2}(\R^{d})}$ with $\beta$ from Proposition \ref{prop:bias_A_with_delta}.
By Lemma \ref{lem:laplace_to_nabla} with $z=K$ this means
\begin{align*}
\langle K,\beta\rangle_{L^{2}(\R^{d})} & =-\langle\scapro{\nabla\theta(x_{0})}{x}_{\R^{d}},|\nabla K(x)|^{2}\rangle_{L^{2}(\R^{d})},
\end{align*}
and Theorem \ref{thm:augMLE_CLT} yields
\[
\mu^{A}=\frac{\int_{\R^{d}}\scapro{\nabla\theta(x_{0})}{x}_{\R^{d}}|\nabla K(x)|^{2}dx}{\norm{\nabla K}_{L^{2}(\R^{d})}^{2}},\,\,\,\,\Sigma^{A}=\frac{2\norm K_{L^{2}(\R^{d})}^{2}}{T\norm{\nabla K}_{L^{2}(\R^{d})}^{2}}.
\]
In particular, if $\nabla K$ is symmetric in the sense $|\nabla K(-x)|=|\nabla K(x)|$,
$x\in\R^{d}$, then the asymptotic bias vanishes:
\[
\delta^{-1}\left(\hat{\vartheta}_{\delta}^{A}\left(x_{0}\right)-\theta\left(x_{0}\right)\right)\r dN\left(0,\frac{2\theta(x_{0})\norm K_{L^{2}(\R^{d})}^{2}}{T\norm{\nabla K}_{L^{2}(\R^{d})}^{2}}\right).
\]
The rougher $K$ is, the
smaller is the asymptotic variance, which bears some similarity with
deconvolution problems.
\end{example}

If the asymptotic bias $\mu^{A}$ vanishes, we can construct a simple
confidence interval in terms of the augmented MLE. Note that in the
setting of Example \ref{exa:cont}, $\Sigma^{A}=2T^{-1}\norm K_{L^{2}(\R^{d})}^{2}\norm{\nabla K}_{L^{2}(\R^{d})}^{-2}$
is easily accessible.
\begin{cor}
\label{cor:confInterval_A} Assume the setting of Theorem \ref{thm:augMLE_CLT},
$\mu^{A}=0$ and let $\overline{\alpha}\in(0,1)$. Then the confidence interval for
$\theta(x_{0})$
\[
I_{1-\overline{\alpha}}^{A}=\left[\hat{\vartheta}_{\delta}^{A}(x_{0})-\delta(\hat{\vartheta}_{\delta}^{A}(x_{0})\Sigma^{A})^{1/2}q_{1-\overline{\alpha}/2},
\hat{\vartheta}_{\delta}^{A}(x_{0})+\delta(\hat{\vartheta}_{\delta}^{A}(x_{0})\Sigma^{A})^{1/2}q_{1-\overline{\alpha}/2}\right],
\]
with the standard normal $(1-\overline{\alpha}/2)$-quantile $q_{1-\overline{\alpha}/2}$,
has asymptotic coverage $1-\overline{\alpha}$ for $\delta\to0$.
\end{cor}

\begin{proof}
By Theorem \ref{thm:augMLE_CLT} and Slutsky's lemma applied for $\hat{\vartheta}_{\delta}^{A}\left(x_{0}\right)\xrightarrow{\PP}\theta(x_{0})$,
we have
\[
\delta^{-1}(\hat{\vartheta}_{\delta}^{A}\left(x_{0}\right)\Sigma^{A})^{-1/2}\left(\hat{\vartheta}_{\delta}^{A}\left(x_{0}\right)-\theta\left(x_{0}\right)\right)\r dN\left(0,1\right),\,\,\,\,\delta\rightarrow0,
\]
noting $\mu^{A}=0$. This yields $\PP(\theta(x_{0})\in I_{1-\overline{\alpha}}^{A})\to 1-\overline{\alpha}$.
\end{proof}

\subsection{\label{subsec:Proxy-MLE}Results for the proxy MLE}

In the setting described in Section \ref{sec:erranalysis}
we obtain a CLT for the quadratic functional $\c I_{\delta}^{P}$.
The proof uses very precise asymptotic
moment calculations and the fourth moment theorem
in Wiener chaos. Fundamental for this analysis is that $X_{\delta,x_{0}}(t)$
and $X_{\delta,x_{0}}(s)$ quickly decorrelate as $\delta^{-2}\abs{t-s}\to\infty$,
which is also predicted by the scaling limit in Corollary \ref{thm:convToHeatEq}.
Note that this method of proof might also cover time-discrete observations
of $X_{\delta,x_{0}}$ if the sampling frequency increases sufficiently
fast as $\delta\to0$, but this is not pursued here.

The next assumption gathers the conditions required for the analysis of the proxy MLE.

\begin{comment}
In the setting described in Section \ref{sec:erranalysis} we obtain
a CLT for the proxy MLE $\hat{\theta}_{\delta}^{P}(x_{0})$. It is
based on the delta method and a CLT for the quadratic functional $\c I_{\delta}^{P}$
using very precise asymptotic moment calculations and the \textit{\emph{fourth
moment theorem}} in Wiener chaos. Fundamental for this analysis is
that $X_{\delta,x_{0}}(t)$ and $X_{\delta,x_{0}}(s)$ quickly decorrelate
as $\delta^{-2}\abs{t-s}\to\infty$, which is also predicted by the
scaling limit in Corollary \ref{thm:convToHeatEq}. Note that this
method of proof might also cover time-discrete observations of $X_{\delta,x_{0}}$
if the sampling frequency increases sufficiently fast as $\delta\to0$,
but this is not pursued here.
\end{comment}

\begin{assumption}\label{ass:proxy}
Let $K=\Delta\tilde{K}$ for $\tilde{K}\in H^{4}(\R^{d})$
with compact support and let $B=M_{\sigma}$ with $\sigma\in C^{1}(\R^{d})$.
Grant Assumption \aswithargsref{assu:X0}{\tilde{K}}{3}, and
for $d=1$ assume $\theta\in C^{1+\alpha'}(\overline{\Lambda})$ for $\alpha'>1/2$ and $\int_{\R}\tilde K(x)dx=0$.
\end{assumption}

The following proposition is proved in Appendix \ref{app:1}.

\begin{prop}
\label{prop:I_P_CLT}
Grant Assumption \ref{ass:proxy}. Then for $\delta\rightarrow0$:
\begin{align*}
\delta^{-1}(\c I_{\delta}^{P}-\theta(x_{0})^{-1}C_{T,K}) & \r dN\bigg(-\theta^{-2}(x_{0})C_{T,K}\mu_{1}^{P},\theta^{-3}(x_{0})C_{T,K}^{2}\Sigma^{P}\bigg),\\
\text{with }C_{T,K} & =\frac{T}{2}\sigma^{2}(x_{0})\norm{\nabla\tilde{K}}_{L^{2}(\R^{d})}^{2},\\
\mu_{1}^{P} & =-\frac{\theta^{2}(x_{0})}{\sigma^{2}(x_{0})}\norm{\nabla\tilde{K}}_{L^2(\R^d)}^{-2}\Big\langle\sc{\nabla(\tfrac{\sigma^{2}}{\theta})(x_{0})}x_{\R^{d}},|\nabla\tilde{K}|^{2}\Big\rangle_{L^{2}(\R^{d})},\\
\Sigma^{P} & =\frac{4}{T}\norm{\nabla\tilde{K}}_{L^{2}(\R^{d})}^{-4}\int_{0}^{\infty}\norm{\nabla e^{(s/2)\Delta}\tilde{K}}_{L^{2}(\R^{d})}^{4}ds.
\end{align*}
\end{prop}

This yields asymptotic normality for the proxy MLE $\hat{\theta}_{\delta}^{P}(x_{0})$.

\begin{thm}
\label{thm:ProxyMLE_CLT} Grant Assumption \ref{ass:proxy}. Then for $\delta\rightarrow0$:
\begin{align*}
\delta^{-1}\left(\hat{\vartheta}_{\delta}^{P}\left(x_{0}\right)-\theta\left(x_{0}\right)\right) & \r dN\left(\mu_{1}^{P}+\mu_{2}^{P},\theta(x_{0})\Sigma^{P}\right),\\
\text{with }\mu_{2}^{P} & =\tfrac{\theta(x_{0})}{\sigma^{2}(x_{0})}\norm K_{L^{2}(\R^{d})}^{-2}\sc{\sc{\nabla\sigma^{2}(x_{0})}x_{\R^{d}}}{\left|K\right|^{2}}_{L^{2}(\R^{d})},
\end{align*}
and with $\mu_{1}^{P}$, $\Sigma^{P}$ from Proposition \ref{prop:I_P_CLT}.
\end{thm}

\begin{proof}
Recall the quadratic variation $\langle X_{\delta,x_{0}}\rangle_{T}=T\norm{\sigma K_{\delta,x_{0}}}^{2}$,
the constant $C_{T,K}$ from Proposition \ref{prop:I_P_CLT} and set
\begin{align*}
D_{T,K} & =\frac{T}{2}\norm{\nabla\tilde{K}}_{L^{2}(\R^{d})}^{2}\norm K_{L^{2}(\R^{d})}^{-2}.
\end{align*}
Write $\hat{\theta}_{\delta}^{P}(x_{0})=(\c I_{\delta}^{P})^{-1}D_{T,K}\norm{\sigma K_{\delta,x_{0}}}^{2}$
and decompose
\begin{align*}
\delta^{-1}(\hat{\vartheta}_{\delta}^{P}(x_{0})-\theta(x_{0})) & =(\c I_{\delta}^{P})^{-1}D_{T,K}\delta^{-1}(\norm{\sigma K_{\delta,x_{0}}}^{2}-D_{T,K}^{-1}C_{T,K})\\
 & \quad\quad\quad+C_{T,K}\delta^{-1}((\c I_{\delta}^{P})^{-1}-\theta(x_{0})C_{T,K}^{-1}).
\end{align*}
From the compact support of $K$ we infer for $\delta\to 0$
\begin{align*}
\tfrac{\norm{\sigma K_{\delta,x_{0}}}^{2}-D_{T,K}^{-1}C_{T,K}}{\delta} & =\sc{\tfrac{\sigma^{2}(x_{0}+\delta\cdot)-\sigma^{2}(x_{0})}{\delta}K}K_{L^{2}(\R^{d})}\to \sc{\sc{\nabla\sigma^{2}(x_{0})}x_{\R^{d}}}{\left|K\right|^{2}}_{L^{2}(\R^{d})}.
\end{align*}
Proposition \ref{prop:I_P_CLT} and the
delta method (\citet[Theorem 7]{Ferguson:1996up}) give
\begin{align*}
\delta^{-1}\left((\c I_{\delta}^{P})^{-1}-\theta(x_{0})C_{T,K}^{-1}\right) & \r dN\left(C_{T,K}^{-1}\mu_{1}^{P},\theta(x_{0})C_{T,K}^{-2}\Sigma^{P}\right),
\end{align*}
and, in particular, $(\c I_{\delta}^{P})^{-1}\r{\P}\theta(x_{0})C_{T,K}^{-1}$.
The theorem follows from Slutsky's lemma.
\end{proof}
The dependencies on $\delta,T,\theta,K$ in the CLT are similar as
for $\hat{\theta}_{\delta}^{A}(x_{0})$. It is interesting to note
that the asymptotic bias depends locally at $x_0$ on $\sigma^2, \theta$ and their gradients, while $a$, $b$ do not appear at all. The asymptotic bias vanishes when $\frac{\sigma^{2}}{\theta}$
and $\sigma^{2}$ are constant, but also similar to Example \ref{exa:cont}
if $|\nabla\tilde{K}(-x)|=|\nabla\tilde{K}(x)|$, $|K(-x)|=|K(x)|$,
$x\in\R^{d}$. As for the augmented MLE in Corollary \ref{cor:confInterval_A},
we obtain an asymptotic $(1-\overline{\alpha})$-confidence interval.

\begin{cor}
\label{cor:confInterval_P} Grant Assumption \ref{ass:proxy}, suppose
$\mu_{1}^{P}+\mu_{2}^{P}=0$ and let $\overline{\alpha}\in(0,1)$. Then the confidence
interval for $\theta(x_{0})$
\[
I_{1-\overline{\alpha}}^{P}=\left[\hat{\vartheta}_{\delta}^{P}(x_{0})-\delta(\hat{\vartheta}_{\delta}^{P}(x_{0})\Sigma^{P})^{1/2}q_{1-\overline{\alpha}/2},\hat{\vartheta}_{\delta}^{P}(x_{0})+\delta(\hat{\vartheta}_{\delta}^{P}(x_{0})\Sigma^{P})^{1/2}q_{1-\overline{\alpha}/2}\right],
\]
with the standard normal $(1-\overline{\alpha}/2)$-quantile $q_{1-\overline{\alpha}/2}$,
has asymptotic coverage $1-\overline{\alpha}$ for $\delta\to0$.
\end{cor}

Let us finally compare the  variance factor $\Sigma^{P}$ to $\Sigma^{A}$
from Theorem \ref{thm:augMLE_CLT}.

\begin{lem}
Under Assumption \ref{ass:proxy} the asymptotic variances of $\hat{\vartheta}_{\delta}^{P}$ and $\hat{\vartheta}_{\delta}^{A}$ always satisfy $\theta(x_0)\Sigma^{P}\geq \theta(x_0)\Sigma^{A}$.
\end{lem}

\begin{proof}
Using the tensor products $\Delta\otimes\Delta$, $f\otimes f$ and
$\Delta\oplus\Delta:=I\otimes\Delta+\Delta\otimes I$, we can write
for $f\in L^{2}(\R^{d})$, identifying $L^{2}(\R^{d})\otimes L^2(\R^{d})=L^{2}(\R^{2d})$,
\begin{align*}
\int_{0}^{\infty}\norm{e^{(s/2)\Delta}f}_{L^{2}(\R^{d})}^{4}ds & =\int_{0}^{\infty}\scapro{(e^{s\Delta}\otimes e^{s\Delta})(f\otimes f)}{f\otimes f}_{L^{2}(\R^{2d})}ds\\
 & =\int_{0}^{\infty}\scapro{e^{s(\Delta\oplus\Delta)}(f\otimes f)}{f\otimes f}_{L^{2}(\R^{2d})}ds\\
 & =\norm{(-\Delta\oplus\Delta)^{-1/2}(f\otimes f)}_{L^{2}(\R^{2d})}^{2},
\end{align*}
provided the last norm is finite, e.g. if $f=(-\Delta)^{1/2}\tilde{K}$.
With this $f$ we conclude via two duality arguments, using $\Delta\tilde K=K$:
\begin{align*}
\Sigma^{P} & =\frac{4}{T}\frac{\norm{(-\Delta\oplus\Delta)^{-1/2}(f\otimes f)}_{L^{2}(\R^{2d})}^{2}}{\norm{f\otimes f}_{L^{2}(\R^{2d})}^{2}}\ge\frac{4}{T}\frac{\norm{f\otimes f}_{L^{2}(\R^{2d})}^{2}}{\norm{(-\Delta\oplus\Delta)^{1/2}(f\otimes f)}_{L^{2}(\R^{2d})}^{2}}\\
 & =\frac{4}{T}\frac{\norm{f\otimes f}_{L^{2}(\R^{2d})}^{2}}{\scapro{(-\Delta\oplus\Delta)(f\otimes f)}{f\otimes f}_{L^{2}(\R^{2d})}}=\frac{2}{T}\frac{\norm{(-\Delta)^{1/2}\tilde{K}}_{L^{2}(\R^{d})}^{2}}{\norm K_{L^{2}(\R^{d})}^{2}}\\
 & \ge\frac{2}{T}\frac{\norm K_{L^{2}(\R^{d})}^{2}}{\norm{(-\Delta)^{1/2}K}_{L^{2}(\R^{d})}^{2}}=\Sigma^{A},
\end{align*}
which yields the result.
\end{proof}

Consequently, the proxy MLE has a larger variance than the augmented
MLE, but the loss in precision is not severe if $K$ has a well concentrated
Fourier spectrum (consider $\Delta$ as a multiplier in the Fourier domain).

Let us point out that in the one-dimensional parametric case with $A_\theta=\theta\partial_{xx}+a\partial_x+b$ and $B=M_\sigma$ for constant $\sigma$, \citet{Bibinger:2019de} construct least-squares estimators for $\sigma^2/\sqrt{\theta}$ and $a/\theta$ from discrete high frequency observations in time at two spatial points $x_1,x_2$. Compared to this, the proxy MLE uses spatial averages of the solution in infinitesimally small neighbourhoods, observed continuously in time, to estimate $\theta$ itself, without having to know $a$ or $\sigma$. A similar phenomenon has been observed by \citet{Cialenco:2019ci} for discrete observations when $a=0$, but they achieve only consistent estimation of $\theta$ and $\sigma$. A more profound comparison of both approaches would be highly desirable.

\subsection{Rate optimality}

Let us address the question of optimality of the above estimators
by providing a minimax lower bound. For minimax lower bounds it suffices
to consider a subclass of the original model and we thus assume here
that $X_{\delta,x_{0}}$ is observed with $A_{\theta}=\theta\Delta$,
$B=I$ and a stationary initial condition $X_{\delta,x_{0}}$. Then
the following result establishes that the rate of convergence $\delta$
is optimal and gives some lower bound for the dependence on $T$,
$\theta$ and $K$.
\begin{prop}
\label{prop:lowerBound}Assume $A_{\theta}=\theta\Delta$, $\theta>0$,
$B=I$, $K\in H^{1}(\R^{d})$ with compact support and that $X_{\delta,x_{0}}$
is stationary. For $\theta_{0}>0$ and $\delta\to0$ we have the asymptotic
local lower bound for the root mean squared error
\[
\inf_{\hat{\theta}}\sup_{\theta\in[\theta_{0},\theta_{0}(1+\delta)]}\E_{\theta}\left[(\hat{\theta}-\theta)^{2}\right]^{1/2}\ge\bar{c}\left(\frac{(\theta_{0}\wedge1)\norm{(I-\Delta)^{-1}K}_{L^{2}(\R^{d})}^{2}}{\sqrt{T}(\norm K_{L^{2}(\R^{d})}^{2}+\norm{\nabla K}_{L^{2}(\R^{d})}^{2})}\right)\delta,
\]
where $\Delta$ is the Laplace operator on $L^2(\R^d)$, $\bar{c}>0$ is some constant and the infimum is taken over
all estimators $\hat{\theta}$ based on observing $X_{\delta,x_{0}}$.
\end{prop}

\begin{proof}
The autocovariance function of the stationary process $(\delta^{-1}X_{\delta,x_{0}}(\delta^{2}t))_{t\in\R}$
is given by
\begin{align*}
c_{\theta,\delta}(t) & :=\delta^{-2}\E[X_{\delta,x_{0}}(\delta^{2}t)X_{\delta,x_{0}}(0)]\\
 & =\delta^{-2}\int_{-\infty}^{0}\scapro{S_{\theta}(\delta^{2}\abs t-s)K_{\delta,x_{0}}}{S_{\theta}(-s)K_{\delta,x_{0}}}ds\\
 & =\scapro{(-2A_{\theta,\delta,x_{0}})^{-1}S_{\theta,\delta,x_{0}}(\abs t)K}{K}_{L^{2}(\Lambda_{\delta,x_{0}})},
\end{align*}
using the scaling in Lemma \ref{lem:scaling} and $\frac{d}{ds}S_{\theta,\delta,x_{0}}(s)=A_{\theta,\delta,x_{0}}S_{\theta,\delta,x_{0}}(s)$
in the last line. The covariance operator for $\delta^{-1}X_{\delta,x_{0}}(\delta^{2}\cdot)$
on $L^{2}(\R)$ is obtained by convolution:
\begin{equation}
C_{\theta,\delta}f(t)=(c_{\theta,\delta}\ast f)(t),\quad t\in\R.\label{eq:cov_operator}
\end{equation}
The squared Hellinger distance $H^{2}(\theta,\theta_{0})$ between
two equivalent centered Gaussian measures can be bounded in terms
of the Hilbert-Schmidt norm of the covariance operators, see e.g.
the proof of the Feldman-Hajek Theorem in \citet[Theorem 2.25]{DaPrato:2014wq}.
For the laws of $(\delta^{-1}X_{\delta,x_{0}}(\delta^{2}t))_{t\in[0,T\delta^{-2}]}$
under $\theta_{0}$ and $\theta$ we can thus bound the corresponding
Hellinger distance via
\[
H^{2}(\theta,\theta_{0})\le\norm{C_{\theta_{0},\delta}^{-1}(C_{\theta,\delta}-C_{\theta_{0},\delta})}_{HS(L^{2}([0,T\delta^{-2}]))}^{2}.
\]
Since the Hellinger distance is invariant under bi-measurable bijective
transformations, $H(\theta,\theta_{0})$ denotes equally the Hellinger
distance between the observation laws of $(X_{\delta,x_{0}}(t))_{t\in[0,T]}$.

Let now $\theta_{\delta}=\theta_{0}+c\delta$ for some small $c>0$,
which we choose below, and assume that we can show $H^{2}(\theta_{\delta},\theta_{0})\le1$
for sufficiently small $\delta$. Then we obtain from the general
lower bound scheme in \citet{tsybakov2008introduction}, using his
Theorem 2.2(ii) and (2.9), that
\begin{equation}
\inf_{\hat{\theta}}\max_{\theta\in\{\theta_{0},\theta_{\delta}\}}\E_{\theta}\left[(\hat{\theta}-\theta)^{2}\right]\ge\tfrac{2-\sqrt{3}}{4}(\theta_{\delta}-\theta_{0})^{2}=\tfrac{2-\sqrt{3}}{4}c^{2}\delta^2.\label{eqLB}
\end{equation}
From this we will obtain the claimed lower bound.

In order to show $H^{2}(\theta_{\delta},\theta_{0})\le1$, denote
by $\iota_{1}:H^{1}([0,T\delta^{-2}])\to L^{2}([0,T\delta^{-2}])$
the Sobolev embedding operator. It is known from Maurin's Theorem,
see e.g. the proof of \citet[Theorem 6.61]{adams2003sobolev}, that
$\iota_{1}$ is Hilbert-Schmidt with
\[
\norm{\iota_{1}}_{HS(H^{1}([0,T\delta^{-2}]),L^{2}([0,T\delta^{-2}]))}^{2}\le K_{HS}T\delta^{-2}
\]
for some constant $K_{HS}>0$. By Hilbert-Schmidt norm calculus (in particular, $\norm{AB}_{HS(H_2,H_3)}\le\norm{A}_{HS(H_1,H_3)}\norm{B}_{H_2\to H_1}$ with obvious notation for the Hilbert-Schmidt and operator norms between Hilbert spaces $H_1,H_2,H_3$), the
implicit restriction of the covariance operators and by the covariance bound of Lemma \ref{lemH1Cov}
below we conclude for $\theta_{\delta}>\theta_{0}$ that
\begin{align*}
 & H^{2}(\theta_{\delta},\theta_{0})\le\norm{C_{\theta_{0},\delta}^{-1}(C_{\theta_{\delta},\delta}-C_{\theta_{0},\delta})}_{HS(L^{2}([0,T\delta^{-2}]))}^{2}\\
 & \le\norm{\iota_{1}}_{HS(H^{1}([0,T\delta^{-2}]),L^{2}([0,T\delta^{-2}]))}^{2}\norm{C_{\theta_{0},\delta}^{-1}(C_{\theta_{\delta},\delta}-C_{\theta_{0},\delta})}_{L^{2}([0,T\delta^{-2}])\to H^{1}([0,T\delta^{-2}])}^{2}\\
 & \le K_{HS}T\delta^{-2}\norm{C_{\theta_{0},\delta}^{-1}(C_{\theta_{\delta},\delta}-C_{\theta_{0},\delta})}_{L^{2}(\R)\to H^{1}(\R)}^{2}\\
 & \le K_{HS}T\left(\theta_{0}^{-2}+\theta_{0}^{-1}\frac{\norm K_{L^{2}(\R^{d})}^{2}+\norm{\nabla K}_{L^{2}(\R^{d})}^{2}}{\norm{(I-A_{1,\delta,x_0})^{-1}K}_{L^{2}(\Lambda_{\delta,x_0})}^{2}}\right)^{2}\frac{(\theta_{\delta}^{2}-\theta_{0}^{2})^{2}}{\delta^{2}}.
\end{align*}
Hence, $H^{2}(\theta_{\delta},\theta)\le1$ holds whenever
\[
\theta_{\delta}^{2}-\theta_{0}^{2}\le\frac{\theta_{0}^{2}}{\sqrt{K_{HS}T}}\left(1+\theta_{0}\frac{\norm K_{L^{2}(\R^{d})}^{2}+\norm{\nabla K}_{L^{2}(\R^{d})}^{2}}{\norm{(I-A_{1,\delta,x_0})^{-1}K}_{L^{2}(\Lambda_{\delta,x_0})}^{2}}\right)^{-1}\delta.
\]
Noting the convergence $\norm{(I-A_{1,\delta,x_0})^{-1}K}_{L^{2}(\Lambda_{\delta,x_0})}\to\norm{(I-\Delta)^{-1}K}_{L^{2}(\R^{d})}$
from Lemma \ref{lemH1Cov} below, we can thus find a sufficiently
small constant $c'>0$ such that, with
\[
c=c'\theta_{0}\frac{(1\wedge\theta_{0}^{-1})\norm{(I-\Delta)^{-1}K}_{L^{2}(\R^{d})}^{2}}{\sqrt{T}(\norm K_{L^{2}(\R^{d})}^{2}+\norm{\nabla K}_{L^{2}(\R^{d})}^{2})},
\]
(\ref{eqLB}) holds for $\theta_{\delta}=\theta_{0}+c\delta$. This
yields the result.
\end{proof}

\section{\label{sec:Discussion-and-numerical}A numerical example}

\begin{figure}
	\includegraphics[width=\textwidth]{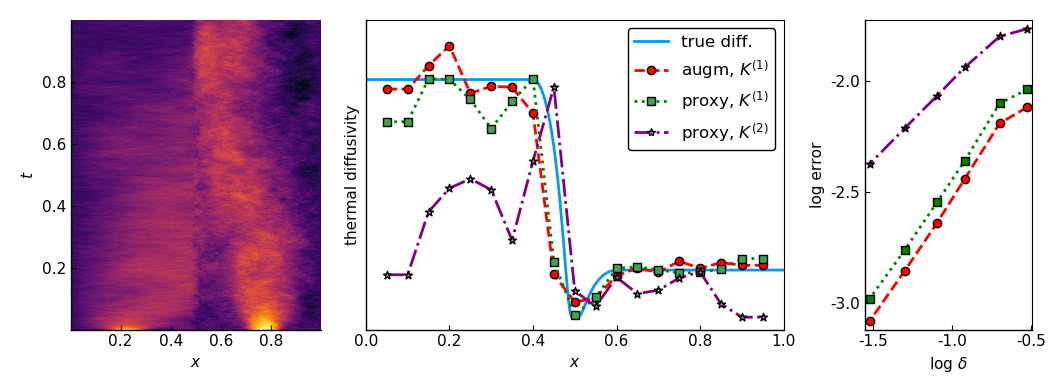}
	\caption{\label{fig:Estimation} (left) heat map for a typical realisation of $X(t,x)$; (center) true $\vartheta$ compared to $\hat{\vartheta}_{\delta}^{A}$ and $\hat{\vartheta}_{\delta}^{P}$ at $\delta=0.12$ with two different kernels; (right) $\log_{10}$-$\log_{10}$ plot of root mean squared estimation errors at $x_0=0.6$ for the estimators in the center.}
\label{fig:1}
\end{figure}

In this section we briefly illustrate the main results from above with simulation results.
Let $\Lambda=(0,1)$, $T=1$, and consider the stochastic heat equation
\[
dX(t)=\Delta_{\theta}X(t)dt+dW(t)
\]
with Dirichlet boundary conditions and with spatially varying diffusivity
$\theta$, which is smooth (true diffusivity in Figure \ref{fig:1}
(center)). Assume that $X_{0}$ is zero, except for two equally high
``peaks'' at $x=0.2$ and $x=0.8$. The heat map for a typical realisation
is presented in Figure \ref{fig:1} (left) and we see already qualitatively
that the heat diffusion is higher for $x\leq1/2$.

An approximate solution $\tilde{X}(t_{k},y_{j})\approx X(t_{k})(y_{j})$
is obtained with respect to a regular time-space grid $\{(t_{k},y_{j}):t_{k}=k/N,y_{j}=j/M,k=0,\dots,N,j=0,\dots,M\}$
by a semi-implicit Euler scheme and a finite difference approximation
of $\Delta_{\theta}$ (\citet[Section 10.5]{Lord:2014wd}). Since the solution
is tested against functions $K_{\delta,x_{0}}$ and $\Delta K_{\delta,x_{0}}$
with small support, $M$ needs to be relatively large, while it is
well-known that accurate simulation requires $N\asymp M^{2}$, see
\citet[p. 458]{Lord:2014wd}. We therefore choose $M=2000$, $N=10^{6}$.

Consider the kernels $K^{(1)}=\varphi'''$, $K^{(2)}=\varphi'$ with
a smooth bump function
\[
\varphi(x)=\exp(-\frac{12}{1-x^{2}}),\,\,\,\,x\in(-1,1).
\]
For $\delta\in\{0.03,0.05,0.08,0.12,0.2,0.3\}$ and $x_{0}\in(0,1)$
on a regular grid we obtain approximate local measurements $\tilde{X}_{\delta,x_{0}}$,
$\tilde{X}_{\delta,x_{0}}^{\Delta}$ for $K^{(1)}$ and $K^{(2)}$,
respectively, from which the augmented MLE $\hat{\theta}_{\delta}^{A}(x_{0})$
and the proxy MLE $\hat{\theta}_{\delta}^{P}(x_{0})$ are computed.
For $x_{0}$ near the boundary and $i=1,2$ set
\[
K_{\delta,x_{0}}^{(i)}=\begin{cases}
K_{\delta,\delta}^{(i)}, & x_{0}<\delta,\\
K_{\delta,1-\delta}^{(i)}, & x_{0}>1-\delta.
\end{cases}
\]

Figure \ref{fig:1} (center) shows pointwise estimation results for
$\theta(x_{0})$ at $\delta=0.12$ and for different $x_{0}$, while
Figure \ref{fig:1} (right) presents a $\log_{10}$-$\log_{10}$ plot of root mean squared estimation errors at $x_{0}=0.6$ for $\delta\rightarrow0$,
obtained by 5.000 Monte-Carlo runs.

Already at the relatively large resolution $\delta=0.12$ both $\hat{\theta}_{\delta}^{A}(x_{0})$
and $\hat{\theta}_{\delta}^{P}(x_{0})$ perform surprisingly well.
For $K^{(1)}$ both estimators are close together and achieve after
a burn-in phase the convergence rate $\delta$, as predicted by Theorems
\ref{thm:augMLE_CLT} and \ref{thm:ProxyMLE_CLT}. Note that $K^{(1)}=\Delta\tilde{K}$
for $\tilde{K}=\varphi'$ and $\int_{\R}\tilde{K}(x)dx=0$ such that the
assumptions of Theorem \ref{thm:ProxyMLE_CLT} are satisfied. With
respect to $K^{(2)}$ those assumptions are not met and indeed $\hat{\theta}_{\delta}^{P}(x_{0})$
deviates considerably from $\theta(x_{0})$, but still seems to be
consistent with rate of convergence dropping to about $\delta^{3/4}$.
Estimation by $\hat{\theta}_{\delta}^{A}(x_{0})$ is unaffected by choosing $K^{(2)}$ instead of $K^{(1)}$ (not
shown).

\appendix

\section{Proofs}

For a better understanding we structure the appendix such that the
proofs for the main theorems of Section \ref{sec:Main-results}
are given in Section \ref{app:1}. Only afterwards, we provide the
technical tools used for the main proofs. Section \ref{subsec:Analytical-results}
contains analytical results for rescaled semigroups and heat kernels,
while Section \ref{subsec:asympForCovars} assembles precise asymptotics
for variance and covariance expressions.

From now on, without loss of generality replace $\Lambda$ with $\Lambda-x_{0}$
and assume $x_{0}=0$. In particular, we estimate $\theta(0)$ and
ease notation by removing the subindex $x_{0}$ and write $\Lambda_{\delta}=\Lambda_{\delta,x_{0}}$,
$z_{\delta}=z_{\delta,x_{0}}$ and $X_{\delta}=X_{\delta,x_{0}}$.
Unless stated otherwise, all limits are for $\delta\rightarrow0$.
$C$ always denotes a generic positive constant, which may depend
on $T$, if not made explicit otherwise, and changes from line to line.
$A\lesssim B$ means $A\leq CB$. For $z\in L^1(\R^d)\cap L^2(\R^d)$ define the norm
\[
\norm{z}_{L^{1}\cap L^{2}(\R^{d})}:=\norm{z}_{L^{1}(\R^{d})}+\norm{z}_{L^{2}(\R^{d})},
\]
and for $z$ with partial derivatives up to second order in $L^1(\R^d)\cap L^2(\R^d)$ set
\[
\norm{z}_{W_{1,2}^{2}(\R^{d})}:=\norm{z + |\nabla z| + \Delta z}_{L^1\cap L^2(\R^d)}.
\]
We write throughout $\sc{X(t)}z=\sc{\tilde{X}(t)}z+\sc{S_{\theta}(t)X_{0}}z$ for
$z\in L^{2}(\Lambda)$ with $\sc{\tilde{X}(t)}z$ being defined as
$\sc{X(t)}z$, but with $X_{0}=0$. Note that $\E[\sc{\tilde{X}(t)}z]=0$ and $\E[\sc{X(t)}z]=\sc{S_{\theta}(t)X_{0}}z$. Set also $\tilde{X}_{\delta}(t)=\sc{\tilde{X}(t)}{K_{\delta}}$, $\tilde{X}_{\delta}^{\Delta}(t)=\sc{\tilde{X}(t)}{\Delta K_{\delta}}$. We will use frequently, without explicit mention, that $\Delta K_{\delta}=\delta^{-2}(\Delta K)_{\delta}$
by Lemma \ref{lem:scaling}.

\subsection{Proofs for Section \ref{sec:Main-results}}

\label{app:1}
\begin{proof}[Proof of Proposition \ref{prop:pseudo_Fisher_info_A}]
We show the result first for  $\tilde{\c I}_{\delta}^{A}=\int_{0}^{T}\tilde{X}_{\delta}^{\Delta}(t)^{2}dt$. Propositions
\ref{prop:cov_asymp}(ii) and \ref{prop:var_with_delta}(ii) below
with $w^{(\delta)}=\Delta K$, $z=K$ yield
\[
\E[\delta^{2}\tilde{\c I}_{\delta}^{A}]=\delta^{2}\int_{0}^{T}\Var(\tilde{X}_{\delta}^{\Delta}(t))dt\rightarrow T\theta(0)^{-1}\Psi(\Delta K,\Delta K), \quad \text{Var}\big(\delta^2\tilde{\c I}_{\delta}^{A}\big)\to 0.
\]
In particular, $\Var(\tilde{\c I}_{\delta}^{A})/\E[\tilde{\c I}_{\delta}^{A}]^{2}\r{\P}0$
and thus $\tilde{\c I}_{\delta}^{A}/\E[\c{\tilde{I}}_{\delta}^{A}]\r{\P}1$.
To finish the proof, decompose
\begin{equation}\label{eq.ItildeI}
\delta^2\c I_{\delta}^{A}=\delta^2\c{\tilde{I}}_{\delta}^{A}+\int_{0}^{T}\delta^2\E[X_{\delta}^{\Delta}(t)]^{2}dt+2\int_{0}^{T}\delta^2 \tilde{X}_{\delta}^{\Delta}(t)\E[X_{\delta}^{\Delta}(t)]\,dt.
\end{equation}
Assumption \aswithargsref{assu:X0}{K}{2} gives
\begin{equation}
\int_{0}^{T}\delta^2\E[X_{\delta}^{\Delta}(t)]^{2}dt=\delta^{-2} \int_{0}^{T}\scapro{S_{\theta}(t)X_{0}}{(\Delta K)_{\delta}}^{2}dt\to 0.\label{eq:X_delta_0}
\end{equation}
By the Cauchy-Schwarz inequality, the cross-term in \eqref{eq.ItildeI} is therefore also negligible and the result follows.
\begin{comment}
For the general case By , , and so it is enough to show $\c I_{\delta}^{A}/\tilde{\c I}_{\delta}^{A}\r{\P}1$,
which follows by the first part if $\E[(\c I_{\delta}^{A}-\tilde{\c I}_{\delta}^{A})^{2}]/\E[\tilde{\c I}_{\delta}^{A}]^{2}\rightarrow0$.
Applying again Proposition \ref{prop:pseudo_Fisher_info_A}, the result
is thus obtained by Equation and
\begin{align*}
 & \E\left[\delta^{4}(\c I_{\delta}^{A}-\tilde{\c I}_{\delta}^{A})^{2}\right]\lesssim\E\left[\left(\delta^{2}\int_{0}^{T}\tilde{X}_{\delta}^{\Delta}(t)\E[X_{\delta}^{\Delta}(t)]dt\right)^{2}\right]+\left(\delta^{2}\int_{0}^{T}\E[X_{\delta}^{\Delta}(t)]^{2}dt\right)^{2}\\
 & \quad\lesssim\E\left[\delta^{2}\tilde{\c I}_{\delta}^{A}\right]\left(\delta^{2}\int_{0}^{T}\E[X_{\delta}^{\Delta}(t)]^{2}dt\right)+\left(\delta^{2}\int_{0}^{T}\E[X_{\delta}^{\Delta}(t)]^{2}dt\right)^{2}=o(1).
\end{align*}
\end{comment}
\end{proof}

\begin{proof}[Proof of Proposition \ref{prop:bias_A_with_delta}]
Define $\tilde{R}_\delta^A$ as $R_{\delta}^A$, but with respect to $\tilde{X}(\cdot)$. In terms of $\beta^{(\delta)}:=\delta^{-1}(A_{\vartheta,\delta}^{*}-\theta(0)\Delta)K$
we have $\delta R_{\delta}^{A}=\int_{0}^{T}X_{\delta}^{\Delta}(t)\langle X(t),\beta_{\delta}^{(\delta)}\rangle dt$ and $\delta \tilde{R}_{\delta}^{A}=\int_{0}^{T}\tilde{X}_{\delta}^{\Delta}(t)\langle \tilde{X}(t),\beta_{\delta}^{(\delta)}\rangle dt$.
$\beta^{(\delta)}$ and $\beta$ correspond to $v^{(\delta)}$ and
$v$ from Lemma \ref{lem:v0} below with $z=K$, and therefore $\beta^{(\delta)}\rightarrow\beta$
in $L^{2}(\R^{d})$. Decompose $R_{\delta}^{A}=\tilde{R}_{\delta}^{A}+V_{1,\delta}+V_{2,\delta}$,
where
\begin{align*}
V_{1,\delta} & =\int_{0}^{T}\sc{S_{\theta}(t)X_{0}}{\delta^{-2}(\Delta K)_{\delta}}(\sc{\tilde{X}(t)}{\beta_{\delta}^{(\delta)}}+\sc{S_{\theta}(t)X_{0}}{\beta_{\delta}^{(\delta)}})dt,\\
V_{2,\delta} & =\int_{0}^{T}\tilde{X}_{\delta}^{\Delta}(t)\sc{S_{\theta}(t)X_{0}}{\beta_{\delta}^{(\delta)}}dt.
\end{align*}
We infer $\delta V_{2,\delta}\r{\P}0$
from $\E[\delta^2\tilde{\c I}_{\delta}^{A}]=O(1)$ by Proposition
\ref{prop:pseudo_Fisher_info_A} and from
\[
\int_{0}^{T}\sc{S_{\theta}(t)X_{0}}{\beta_{\delta}^{(\delta)}}^{2}dt=O(\delta\norm{\beta^{(\delta)}}_{L^1\cap L^2(\R^d)}^{2})\to 0
\]
by Lemma \ref{lem:X0}(i) below with $u=\beta^{(\delta)}$.
This, (\ref{eq:X_delta_0}) and the Cauchy-Schwarz inequality together
with
\[
\int_{0}^{T}\E[\sc{\tilde{X}(t)}{\beta_{\delta}^{(\delta)}}^{2}]dt=\int_{0}^{T}\int_{0}^{t}\norm{B^{*}S_{\theta}^{*}\left(s\right)\beta_{\delta}^{(\delta)}}^{2}ds\lesssim\norm{\beta_{\delta}^{(\delta)}}^{2}<\infty,
\]
also imply $\delta V_{1,\delta}\r{\P}0$. By Proposition \ref{prop:pseudo_Fisher_info_A}
it therefore suffices to show
\begin{align*}
\E[\delta\tilde{R}_{\delta}^{A}] & \rightarrow T\theta\left(0\right)^{-1}\Psi(\Delta K,\beta),\,\,\,\,\text{Var}(\delta\tilde{R}_{\delta}^{A})\to0.
\end{align*}
The convergence of $\E[\delta\tilde{R}_{\delta}^{A}]$ follows for
$d\geq2$ from Proposition \ref{prop:cov_asymp}(iii) below with $w^{(\delta)}=\Delta K$,
$z=K$, $u^{(\delta)}=\beta^{(\delta)}$, $u=\beta$. For $d=1$, $\theta \in C^{1+\alpha'}(\overline{\Lambda})$ for $\alpha'>1/2$ and $\int_{\R}K(x)dx=0$,
it follows from Lemma \ref{lem:v0}(ii) that there
is a compactly supported $\tilde{\beta}\in H^{2}(\R)$ with $\beta=\Delta\tilde{\beta}$, $\norm{\beta^{(\delta)}-\Delta\tilde{\beta}}_{L^1\cap L^2(\R)}\leq C \delta^{\alpha'}$.
Then, by polarisation and Proposition \ref{prop:cov_asymp}(ii), $\E[\delta\tilde{R}_{\delta}^{A}]$
converges to
\begin{align*}
 & \frac{T}{4\theta(0)}\left(\Psi(\Delta(K+\tilde{\beta}),\Delta(K+\tilde{\beta}))-\Psi(\Delta(K-\tilde{\beta}),\Delta(K-\tilde{\beta}))\right)=\frac{T}{\theta(0)}\Psi(\Delta K,\beta).
\end{align*}
Next, $\Var(\delta\tilde{R}_{\delta}^{A})=\Var(\int_{0}^{T}X_{\delta}^{\Delta}(t)\langle X(t),\beta_{\delta}^{(\delta)}\rangle dt)\rightarrow0$
follows for $d\geq2$ by Proposition \ref{prop:var_with_delta}(i)
below with $z=K$, $u^{(\delta)}=\beta^{(\delta)}$, $u=\beta$.  If
$\theta\in C^{1+\alpha'}(\overline{\Lambda})$ and $\int_{\R}K(x)dx=0$, then $\Var(\delta\tilde{R}_{\delta}^{A})\rightarrow0$
by Proposition \ref{prop:var_with_delta}(ii) with $z=K$, $w^{(\delta)}=\beta^{(\delta)}$,
$m=\tilde{\beta}$.%
\end{proof}

\begin{proof}[Proof of Proposition \ref{prop:I_P_CLT}]
Define  $\tilde{\c I}_{\delta}^{P}$
as $\c I_{\delta}^{P}$, but with respect to $\tilde{X}(\cdot)$. By Assumption
\aswithargsref{assu:X0}{\tilde{K}}{3} and $K_{\delta}=(\Delta\tilde{K})_{\delta}$
we have $\delta^{-3}\int_{0}^{T}\sc{S_{\theta}(t)X_{0}}{K_{\delta}}^{2}dt\to 0$ whence $\delta^{-1}(\c I_{\delta}^{P}-\tilde{\c I}_{\delta}^{P})\to 0$ follows by $\tilde{\c I}_{\delta}^{P}=O_{\P}(\delta)$ and the Cauchy-Schwarz inequality.

It remains to prove the result for $\tilde{\c I}_{\delta}^{P}$. Note that
\begin{equation}\label{eq:proofIPCLT}
\delta^{-1}\left(\tilde{\c I}_{\delta}^{P}-\theta(0)^{-1}C_{T,K}\right)=Z_{\delta}+\delta^{-1}\Big(\delta^{-2}\int_{0}^{T}\Var(\tilde{X}_{\delta}(t))dt
-\theta(0)^{-1}C_{T,K}\Big).
\end{equation}
with $Z_{\delta}:=\delta^{-3}\int_{0}^{T}(\tilde{X}_{\delta}(t)^{2}-\E[\tilde{X}_{\delta}(t)^{2}])dt$.
$\tilde{X}_{\delta}$ is a centered Gaussian process and $Z_{\delta}$
is an element of the second Wiener chaos. By the fourth moment
theorem (\citet[Theorem 1]{Nualart:2005by}) it suffices to prove
$\text{Var}(Z_{\delta})\r{}\Sigma$ and $\E[Z_{\delta}^{4}]\rightarrow3\Sigma^{2}$
to conclude  $Z_{\delta}\r dN(0,\Sigma)$. Propositions \ref{prop:var_with_delta}(iv) and
\ref{prop:eightMom} below (with $w^{(\delta)}=\Delta\tilde{K}$) provide exactly these convergences with
\[
\Sigma=\frac{4T}{\theta^{3}(0)}\int_{0}^{\infty}\Psi(e^{s\Delta}\Delta\tilde{K},\Delta\tilde{K})^{2}ds=\frac{T\sigma^{4}(0)}{\theta^{3}(0)}\int_{0}^{\infty}\norm{\nabla e^{(s/2)\Delta}\tilde{K}}_{L^{2}(\R^{d})}^{4}ds,
\]
where the last identity is (\ref{eq:Psi_Delta_example}).
The claim follows from applying  Proposition \ref{prop:cov_asymp_special} below with $z=\tilde{K}$ to the second term  in \eqref{eq:proofIPCLT} and Slutsky's lemma.
\end{proof}

\begin{lem}
\label{lemH1Cov} Assume the setting of Proposition \ref{prop:lowerBound} and recall the operator $C_{\theta,\delta}$ from (\ref{eq:cov_operator}). We have for $\theta>\theta_{0}>0$
\[
\norm{C_{\theta_{0},\delta}^{-1}(C_{\theta,\delta}-C_{\theta_{0},\delta})}_{L^{2}(\R)\to H^{1}(\R)}\le\frac{\theta^{2}-\theta_{0}^{2}}{\theta_{0}^{2}}\left(1+\theta_{0}\frac{\norm K_{L^{2}(\R^{d})}^{2}+\norm{\nabla K}_{L^{2}(\R^{d})}^{2}}{\norm{(I-A_{1,\delta})^{-1}K}_{L^{2}(\Lambda_{\delta})}^{2}}\right).
\]
Moreover, we have $\norm{(I-A_{1,\delta})^{-1}K}_{L^{2}(\Lambda_{\delta})}\to\norm{(I-\Delta)^{-1}K}_{L^{2}(\R^{d})}$
for $\delta\to0$, where $\Delta$ is the Laplace operator on $L^{2}(\R^{d})$.
\end{lem}

\begin{proof}
For simplicity write in the following proof $\theta\Delta$ and $e^{\theta t\Delta}$
instead of $A_{\theta,\delta}$ and $S_{\theta,\delta}(t)$. In the
Fourier domain, the convolution operator $C_{\theta,\delta}$ is given
by
\begin{align*}
{\mathcal{F}}c_{\theta,\delta}(w) & =\int_{0}^{\infty}\scapro{(-2\theta\Delta)^{-1}e^{\theta t\Delta}K}{K}_{L^{2}(\Lambda_{\delta})}(e^{iwt}+e^{-iwt})dt\\
 & =\scapro{(-2\theta\Delta)^{-1}\int_{0}^{\infty}(e^{t(\theta\Delta+iwI)}+e^{t(\theta\Delta-iwI)})K\,dt}{K}_{L^{2}(\Lambda_{\delta})}\\
 & =\scapro{(-2\theta\Delta)^{-1}(-(\theta\Delta+iwI)^{-1}-(\theta\Delta-iwI)^{-1})K}{K}_{L^{2}(\Lambda_{\delta})}\\
 & =\scapro{(\theta^{2}\Delta^{2}+w^{2}I)^{-1}K}{K}_{L^{2}(\Lambda_{\delta})}.
\end{align*}
The operator $C_{\theta_{0},\delta}^{-1}(C_{\theta,\delta}-C_{\theta_{0},\delta})$
is expressed in the Fourier domain by multiplication with
\[
\frac{{\mathcal{F}}c_{\theta,\delta}(w)-{\mathcal{F}}c_{{\theta_{0}},\delta}(w)}{{\mathcal{F}}c_{\theta_{0},\delta}(w)}=(\theta^{2}-\theta_{0}^{2})\frac{\scapro{(\theta^{2}\Delta^{2}+w^{2}I)^{-1}\Delta^{2}(\theta_{0}^{2}\Delta^{2}+w^{2}I)^{-1}K}{K}_{L^{2}(\Lambda_{\delta})}}{\scapro{(\theta_{0}^{2}\Delta^{2}+w^{2}I)^{-1}K}{K}_{L^{2}(\Lambda_{\delta})}}.
\]
Using the description of $H^{1}(\R)$ in the Fourier domain and functional
calculus for the Laplacian $\Delta$ on $L^{2}(\Lambda_{\delta})$
yields therefore for $\theta>\theta_{0}$ that
\begin{align*}
 & \norm{C_{\theta_{0},\delta}^{-1}(C_{\theta,\delta}-C_{\theta_{0},\delta})}_{L^{2}(\R)\to H^{1}(\R)}\\
 & =\sup_{w\in\R}\babs{(1+w^{2})^{1/2}\frac{{\mathcal{F}}c_{\theta,\delta}(w)-{\mathcal{F}}c_{{\theta_{0}},\delta}(w)}{{\mathcal{F}}c_{\theta_{0},\delta}(w)}}\\
 & \le(\theta^{2}-\theta_{0}^{2})\sup_{w'\in\R}\babs{(1+(\theta_{0}w')^{2})^{1/2}\frac{\scapro{\Delta^{2}(\theta_{0}^{2}\Delta^{2}+\theta_{0}^{2}(w')^{2}I)^{-2}K}{K}_{L^{2}(\Lambda_{\delta})}}{\scapro{(\theta_{0}^{2}\Delta^{2}+\theta_{0}^{2}(w')^{2}I)^{-1}K}{K}_{L^{2}(\Lambda_{\delta})}}}\\
 & =\frac{\theta^{2}-\theta_{0}^{2}}{\theta_{0}^{2}}\sup_{w\in\R}\babs{(1+(\theta_{0}w)^{2})^{1/2}\frac{\norm{w^{-1}\Delta(w^{-2}\Delta^{2}+I)^{-1}K}_{L^{2}(\Lambda_{\delta})}^{2}}{\norm{(w^{-2}\Delta^{2}+I)^{-1/2}K}_{L^{2}(\Lambda_{\delta})}^{2}}}\\
 & \le\frac{\theta^{2}-\theta_{0}^{2}}{\theta_{0}^{2}}\left(1+\theta_{0}\left(1\vee\frac{\sup_{w>1}\norm{w^{-1/2}\Delta(w^{-2}\Delta^{2}+I)^{-1}K}_{L^{2}(\Lambda_{\delta})}^{2}}{\inf_{w>1}\norm{(w^{-2}\Delta^{2}+I)^{-1/2}K}_{L^{2}(\Lambda_{\delta})}^{2}}\right)\right)\\
 & \le\frac{\theta^{2}-\theta_{0}^{2}}{\theta_{0}^{2}}\left(1+\theta_{0}\left(1\vee\frac{\norm{(-\Delta)^{1/2}K}_{L^{2}(\Lambda_{\delta})}^{2}}{\norm{(I+\Delta^{2})^{-1/2}K}_{L^{2}(\Lambda_{\delta})}^{2}}\right)\right),
\end{align*}
where we used in the last line $w^{-1/2}{\lambda}(1+w^{-2}\lambda^{2})^{-1}\le{\lambda}^{1/2}$
for all $\lambda,w>0$. For this and similar arguments note that by spectral calculus with  a self-adjoint operator $A$, e.g. $-\Delta$, we have $\norm{f(A)K}\le \norm{g(A)K}$ whenever $\abs{f}\le \abs{g}$ for bounded $f,g$ on the spectrum of $A$. Since $\scapro{-\Delta K}{K}_{L^{2}(\Lambda_{\delta})}=\norm{\nabla K}_{L^{2}(\R^{d})}^{2}$,
the numerator is independent of $\delta$. For the denominator write
again $A_{1,\delta}=\Delta$ and note similarly  $(I+A_{1,\delta}^{2})^{-1/2}\ge(I-A_{1,\delta})^{-1}$,
where we have explicitly, cf. \citet[Chapter 2.6]{Pazy:1983us},
\[
(I-A_{1,\delta})^{-1}K=\int_{0}^{\infty}e^{-t}S_{1,\delta}(t)K\,dt.
\]
Proposition \ref{prop:semigroup_convergence} yields then first, approximating $K$ by continuous functions, that $\norm{S_{1,\delta}(t)K}_{L^{2}(\Lambda_{\delta})}\lesssim\norm K_{L^{2}(\R^{d})}$
uniformly in $\delta$, and second, the convergence
\[
\norm{(I-A_{1,\delta})^{-1}K}_{L^{2}(\Lambda_{\delta})}\rightarrow\norm{(I-\Delta)^{-1}K}_{L^{2}(\R^{d})},\,\,\,\delta\rightarrow0.\qedhere
\]
\end{proof}

\subsection{\label{subsec:Analytical-results}Analytical results}
Recall that the  solution of the heat equation $\frac{d}{dt}u(t)=\lambda\Delta u(t)$,
$\lambda>0$, on $\R^{d}$ with initial value $w\in L^{2}(\R^{d})$
is given by the convolution
\begin{equation}
u(t)=e^{\lambda t\Delta}w=q_{\lambda t}*w,\label{eq:fundamental_solution}
\end{equation}
with the heat kernel $q_{t}(x)=(4\pi t)^{-d/2}\exp(-|x|^{2}/(4t))$, $x\in\R^{d}$.
\begin{lem}
\label{lem:LaplaceSemigroup_props}We have for $u\in L^{2}(\R^{d})$,
$t>0$:
\begin{enumerate}
\item $\norm{e^{t\Delta}u}_{L^{2}(\R^{d})}\lesssim(1\wedge t^{-d/4})\norm u_{L^1\cap L^2(\R^d)}$,
$\norm{\Delta e^{t\Delta}u}_{L^{2}(\R^{d})}\lesssim t^{-1}\norm u_{L^{2}(\R^{d})}$.
\item $xe^{\theta(0)t\Delta}u(x)=-2\theta(0)t\nabla e^{\theta(0)t\Delta}u(x)+e^{\theta(0)t\Delta}(xu)(x)$, $x\in\R^d$.
\item $\norm{|x|^{2}e^{\theta(0)t\Delta}u}_{L^{2}(\R^{d})}\lesssim(1\vee t)(1\wedge t^{-d/4})\norm{(1+|x|+|x|^{2})u}_{L^1\cap L^2(\R^d)}$.
\end{enumerate}
\end{lem}

\begin{proof}
(i). For the second part use functional calculus. The first part follows
from
\[
\norm{e^{t\Delta}u}_{L^{2}(\R^{d})}=\norm{q_{t}*u}_{L^{2}(\R^{d})}\lesssim  \min\Big(\norm{u}_{L^2(\R^d)},t^{-d/4}\norm{u}_{L^1(\R^d)}\Big).
\]
(ii). Let $i\in\{1,\dots,d\}$. The result follows from
\begin{align*}
 & x_{i}\left(e^{\theta(0)t\Delta}u\right)(x)=x_{i}(q_{\theta(0)t}*u)(x)\\
 & \quad=\theta(0)t\int_{\R^{d}}\frac{x_{i}-y_{i}}{\theta(0)t}q_{\theta(0)t}(x-y)u(y)dy+\int_{\R^{d}}y_{i}q_{\theta(0)t}(x-y)u(y)dy\\
 & \quad=-2\theta(0)t(\partial_{i}q_{\theta(0)t}*u)(x)+(q_{\theta(0)t}*(x_{i}u))(x).
\end{align*}
(iii). Applying the proof in (ii) twice for $i\in\{1,\dots,d\}$ gives
\begin{align*}
 & x_{i}^{2}\left(e^{\theta(0)t\Delta}u\right)(x)=-2x_{i}\theta(0)t(\partial_{i}q_{\theta(0)t}*u)(x)+x_{i}(q_{\theta(0)t}*(x_{i}u))(x)\\
 & =4\theta^{2}(0)t^{2}(\partial_{ii}^{2}q_{\theta(0)t}*u)(x)-2\theta(0)t(\partial_{i}q_{\theta(0)t}*(x_{i}u))(x)+(q_{\theta(0)t}*(x_{i}^{2}u))(x).
\end{align*}
Summing over $i$ with $v_{i}=e^{\theta(0)t\Delta}(x_{i}u)$ obtain
from this for $\norm{|x|^{2}e^{\theta(0)t\Delta}u}_{L^{2}(\R^{d})}$
up to a constant the upper bound
\begin{align*}
 & t^{2}\norm{\Delta e^{\theta(0)t\Delta}u}_{L^{2}(\R^{d})}+t\sum_{i=1}^{d}\norm{\partial_{i}v_{i}}_{L^{2}(\R^{d})}+\norm{e^{\theta(0)t\Delta}(|x|^{2}u)}_{L^{2}(\R^{d})}.
\end{align*}
Using $e^{\theta(0)t\Delta}=e^{\theta(0)(t/2)\Delta}e^{\theta(0)(t/2)\Delta}$
and the two statements in (i) yield for the first and last terms the
claimed bound. For the second term integration by parts implies $\norm{\partial_{i}v_{i}}_{L^{2}(\R^{d})}^{2}\leq\sc{-\Delta v_{i}}{v_{i}}_{L^{2}(\R^{d})}\leq\norm{\Delta v_{i}}_{L^{2}(\R^{d})}\norm{v_{i}}_{L^{2}(\R^{d})}$.
The result follows again from applying (i).
\end{proof}

\begin{lem}
\label{lem:laplace_to_nabla} If $z\in H^{2}(\R^{d})$
has compact support, then $\sc z{\partial_{i}z}_{L^{2}(\R^{d})}=0$, $\left\langle x_{i}\Delta z,z\right\rangle_{L^{2}(\R^{d})}=-\left\langle x_{i},|\nabla z|^{2}\right\rangle _{L^{2}(\R^{d})}$ for $i=1,\ldots,d$.
If $z\in H^{4}(\R^{d})$, then also $\sc{\Delta z}{e^{t\Delta}\Delta\partial_{i}z}_{L^{2}(\R^{d})}=0$,
$t\geq0$.
\end{lem}

\begin{proof}
Integration by parts gives $\langle z,\partial_{i}z\rangle_{L^{2}(\R^{d})}=-\langle\partial_{i}z,z\rangle_{L^{2}(\R^{d})}$ (argue with compactly supported $z$ first, then extend by continuity),
implying $\sc z{\partial_{i}z}_{L^{2}(\R^{d})}=0$ and $\langle x_{i}\partial_{jj}z,z\rangle_{L^{2}(\R^{d})}=-\langle x_{i},(\partial_{j}z)^{2}\rangle_{L^{2}(\R^{d})}$
for $j=1,\ldots,d$. The last part follows from the first one for
$\tilde{z}_{t}=e^{(t/2)\Delta}z\in H^{2}(\R^{d})$ using
\[
\left\langle e^{t\Delta}\Delta^{2}z,\partial_{i}z\right\rangle _{L^{2}(\R^{d})}=-\frac{d^{2}}{dt^{2}}\left\langle e^{t\Delta}z,\partial_{i}z\right\rangle _{L^{2}(\R^{d})}=-\frac{d^{2}}{dt^{2}}\left\langle \tilde{z}_{t},\partial_{i}\tilde{z}_{t}\right\rangle _{L^{2}(\R^{d})}.\qedhere
\]
\end{proof}
The upper bounds in the next Proposition are well-known for analytic
semigroups. The main difficulty is to ensure that they hold for growing
domains, uniformly in $\delta>0$.

\begin{prop}
\label{prop:semigroup_bounds}There exist universal constants $M_{0},M_{1}$
such that for $\delta,t>0$
\begin{enumerate}
\item $\norm{S_{\theta,\delta}^{*}(t)}_{L^{2}(\Lambda_{\delta})}\le M_{0}e^{C\delta^{2}t}$,
\item $\norm{tA_{\theta,\delta}^{*}S_{\theta,\delta}^{*}(t)}_{L^{2}(\Lambda_{\delta})}\le M_{1}e^{C\delta^{2}t}$.
\end{enumerate}
\end{prop}

\begin{proof}
The claimed bounds in the statement follow from Proposition 2.1.1
of \citet{lunardi2012analytic}, if we can show
\begin{equation}
\norm{(\lambda I-A_{\theta,\delta}^{*})^{-1}}_{L^{2}(\Lambda_{\delta})}\le\frac{M}{\left|\lambda-w\right|},\label{eq:resolvent_inequality}
\end{equation}
with $w=c_{1}\delta^{2}$ for all $\lambda\in\Sigma_{\sigma,w}:=\{\rho\in\mathbb{C}:|\text{arg}(\rho-w)|<\sigma\}\backslash\{w\}$
and with constants $c_{1},M>0,\sigma\in(\pi/2,\pi)$ independent of
$\delta$. Since the self-adjoint operator $\Delta_{\theta(\delta\cdot)}$
has strictly negative spectrum for all $\delta>0$ (cf. \citet[Section 6.5]{Evans:2010wj}),
by functional calculus (\ref{eq:resolvent_inequality}) holds indeed
for $\Delta_{\theta(\delta\cdot)}$ with $w=0$, $M=1$ and any $\sigma\in(\pi/2,\pi)$.

In order to extend this to $A_{\theta,\delta}^{*}$, we consider it
as a perturbation of $\Delta_{\theta(\delta\cdot)}$. We show first
that $A_{0,\delta}^{*}$ is $\Delta_{\theta(\delta\cdot)}$-bounded,
i.e.
\begin{equation}
\norm{A_{0,\delta}^{*}v}_{L^{2}(\Lambda_{\delta})}\le c_{2}\epsilon\norm{\Delta_{\theta(\delta\cdot)}v}_{L^{2}(\Lambda_{\delta})}+\left(\frac{1}{4\epsilon}+c_{3}\right)\delta^{2}\norm v_{L^{2}(\Lambda_{\delta})}\label{eq:A_bounded}
\end{equation}
for $\epsilon>0$, $v\in H_{0}^{1}(\Lambda_{\delta})\cap H^{2}(\Lambda_{\delta})$
and absolute constants $c_{2},c_{3}>0$. For this note that $\norm{A_{0,\delta}^{*}v}_{L^{2}(\Lambda_{\delta})}$
is upper bounded by
\begin{align*}
 & \norm{\delta\sc{a(\delta\cdot)}{\nabla v}}_{L^{2}(\Lambda_{\delta})}+\delta^{2}\left(\norm{v\,\text{div}\left(a(\delta\cdot)\right)}_{L^{2}(\Lambda_{\delta})}+\norm b_{\infty}\norm v_{L^{2}(\Lambda_{\delta})}\right).
\end{align*}
Moreover, $\norm{\delta\sc{a(\delta\cdot)}{\nabla v}}_{L^{2}(\Lambda_{\delta})}$
is upper bounded by
\begin{align}
 & \delta d^{1/2}\sup_{i=1,\dots,d}\norm{a_{i}}_{\infty}\left(\sum_{i=1}^{d}\norm{\partial_{i}v}_{L^{2}(\Lambda_{\delta})}^{2}\right)^{1/2}\nonumber \\
 & \qquad\le\delta\frac{d^{1/2}\sup_{i=1,\dots,d}\norm{a_{i}}_{\infty}}{\min_{x}\theta(x)^{1/2}}\sc{(-\Delta_{\theta\left(\delta\cdot\right)}v)}v_{L^{2}(\Lambda_{\delta})}^{1/2}\nonumber \\
 & \qquad\le\frac{d^{1/2}\sup_{i=1,\dots,d}\norm{a_{i}}_{\infty}}{\min_{x}\theta(x)^{1/2}}\norm{\Delta_{\theta\left(\delta\cdot\right)}v}_{L^{2}(\Lambda_{\delta})}^{1/2}\delta\norm v_{L^{2}(\Lambda_{\delta})}^{1/2}\label{eq:A0_upper_bound}\\
 & \qquad\le c_{2}\epsilon\norm{\Delta_{\theta\left(\delta\cdot\right)}v}_{L^{2}(\Lambda_{\delta})}+\frac{\delta^{2}}{4\epsilon}\norm v_{L^{2}(\Lambda_{\delta})},\nonumber
\end{align}
with $c_{2}:=\frac{d\sup_{i=1,\dots,d}\norm{a_{i}}_{\infty}^{2}}{\min_{x}\theta(x)}$,
where we used in the last line the basic inequality $xy\le\epsilon x^{2}+\frac{1}{4\epsilon}y^{2}$
for $x,y>0$. This shows (\ref{eq:A_bounded}) with $c_{3}:=\sum_{i=1}^{d}\norm{\partial_{i}a_{i}}_{\infty}+\norm b_{\infty}$.

Choosing $\epsilon$ sufficiently small, the proof of Lemma III.2.6
in \citet{engel1999one} implies (\ref{eq:resolvent_inequality})
for all $\lambda\in\Sigma_{\sigma,0}\cap\{\rho\in\mathbb{C}:|\rho|>c_{4}\delta^{2}\}$
with $c_{4}=\frac{(4\epsilon)^{-1}+c_{3}}{1-2c_{2}\epsilon}$, $\sigma=3\pi/4$
and $M'>0$ instead of $M$. Setting $w=(1+c_{5})c_{4}\delta^{2}$,
for a suitable constant $c_{5}>0$ to be determined later, and assuming
that for these $\lambda$
\begin{equation}
\lambda+w\in\Sigma_{\sigma,0}\cap\{\rho\in\mathbb{C}:|\rho|>c_{4}\delta^{2}\},\,\,\,\,|\lambda+w|\geq C|\lambda|,\label{eq:lambda_bounds}
\end{equation}
with a universal constant $C$, we can therefore conclude for any
$\lambda\in\Sigma_{\sigma,0}\cap\{\rho\in\mathbb{C}:|\rho|>c_{4}\delta^{2}\}$
that
\begin{equation}
\norm{((\lambda+w)I-A_{\theta,\delta}^{*})^{-1}}_{L^{2}(\Lambda_{\delta})}\le\frac{M'}{\left|\lambda+w\right|}\leq\frac{M'C}{\left|\lambda\right|}.\label{eq:resolvent_inequality_2}
\end{equation}
In order to obtain (\ref{eq:resolvent_inequality}) from this let
$\lambda\in\Sigma_{\sigma,w}$ such that $\lambda-w\in\Sigma_{\sigma,0}$.
Assume that we can also show
\begin{equation}
\left|\lambda-w\right|>c_{4}\delta^{2}.\label{eq:lambda_bounds_2}
\end{equation}
Then the result follows from (\ref{eq:resolvent_inequality_2}) with
$c_{1}=(1+c_{5})c_{4}$, $M=M'C$, because
\[
\norm{(\lambda I-A_{\theta,\delta}^{*})^{-1}}_{L^{2}(\Lambda_{\delta})}=\norm{\left(((\lambda-w)+w)I-A_{\theta,\delta}^{*}\right)^{-1}}_{L^{2}(\Lambda_{\delta})}\le\frac{M'C}{\left|\lambda-w\right|}.
\]
We are left with showing (\ref{eq:lambda_bounds}) and (\ref{eq:lambda_bounds_2}).
For (\ref{eq:lambda_bounds}) note that $\lambda\in\Sigma_{\sigma,0}$
already yields $\lambda+w\in\Sigma_{\sigma,0}$, because $w>0$, while
the inequality $|\lambda+w|>c_{4}\delta^{2}$ holds clearly, if $|\text{Im}(\lambda)|>c_{4}\delta^{2}$.
On the other hand, $|\text{arg}(\lambda)|<\sigma$ implies $|\text{Re}(\lambda)|<c_{5}|\text{Im}(\lambda)|$
for a constant $c_{5}>0$ and thus, if $|\text{Im}(\lambda)|\leq c_{4}\delta^{2}$,
then
\begin{equation}
\left|\lambda+w\right|\geq w-|\text{Re}(\lambda)|\geq w-c_{5}|\text{Im}(\lambda)|>c_{4}\delta^{2}.\label{eq:lambda_bounds_3}
\end{equation}
In order to find the constant $C$ in (\ref{eq:lambda_bounds}), note
that $|\lambda+w|\geq|\lambda|$ holds always if $\text{Re}(\lambda)\geq0$,
and that $|\lambda+w|\geq\frac{1}{2}|\lambda|$ whenever $2w\leq|\lambda|$.
Let now $\text{Re}(\lambda)<0$ and $|\lambda|<2w$ such that by (\ref{eq:lambda_bounds})
$\left|\lambda+w\right|>c_{4}\delta^{2}=\frac{2w}{2(1+c_{5})}>C\left|\lambda\right|$,
with $C:=\frac{1}{2(1+c_{5})}$. Finally, with respect to (\ref{eq:lambda_bounds_2}),
$|\lambda-w|>c_{4}\delta^{2}$ holds always, if $|\text{Im}(\lambda)|>c_{4}\delta^{2}$.
On the other hand, $|\text{arg}(\lambda-w)|<\sigma$ implies $|\text{arg}(\lambda)|<\sigma$
and hence for $|\text{Im}(\lambda)|\leq c_{4}\delta^{2}$, as in (\ref{eq:lambda_bounds_3}),
$|\lambda-w|\geq w-|\text{Re}(\lambda)|>c_{4}\delta^{2}$.
\end{proof}
With these preparations we can proceed to proving Proposition \ref{prop:semigroup_convergence}.

\begin{proof}[Proof of Proposition \ref{prop:semigroup_convergence}]
(i). The proof is based on giving a stochastic representation for
$S_{\theta,\delta}^{*}(t)z$ via the Feynman-Kac formulas. Without
loss of generality let $\theta\in C^{1+\alpha}(\R^{d})$, $a\in C^{1+\alpha}(\R^d;\R^d)$, $b\in C^{\alpha}(\R^{d})$, $\alpha>0$, with $\min_{x\in\R^{d}}\theta(x)>0$.
Then for $f\in C_{c}^{2}(\R^{d})$
\begin{equation}
A_{\theta,\delta}^{*}f(x)=\theta(\delta x)\Delta f(x)+\sc{\tilde{a}_{\delta}(x)}{\nabla f(x)}_{\R^{d}}+\tilde{b}_{\delta}(x)f(x),\,\,\,\,x\in\R^{d},\label{eq:A_adjoint}
\end{equation}
where $\tilde{a}_{\delta}=\delta(\nabla\theta(\delta\cdot)-a(\delta\cdot))\in C^{\alpha}(\R^{d})$,
$\tilde{b}_{\delta}=\delta^{2}(b(\delta\cdot)-\text{div}(a(\delta\cdot))\in C^{\alpha}(\R^{d})$.
By \citet[Theorem 5.4.22]{Karatzas1991} we can find a process $Y^{(\delta)}=(Y_{t}^{(\delta)})_{t\geq0}$
being a weak solution of the $d$-dimensional stochastic differential
equation
\[
dY_{t}^{(\delta)}=\tilde{a}_{\delta}(Y_{t}^{(\delta)})dt+\sqrt{2}\theta(\delta Y_{t}^{(\delta)})^{1/2}d\tilde{W}_{t},\,\,\,\,Y_{0}^{(\delta)}=x\in\R^{d},
\]
on a filtered probability space $(\tilde{\Omega},\tilde{\c F},(\tilde{\c F}_{t})_{t\geq0},\tilde{\P})$
carrying a scalar Brownian motion $(\tilde{W}_{t})_{t\geq0}$. We
show below for $x\in\R^{d}$
\begin{align}
(S_{\theta,\delta}^{*}\left(t\right)z)\left(x\right) & =\tilde{\E}_{x}\left[z(Y_{t}^{(\delta)})\exp\Big(\int_{0}^{t}\tilde{b}_{\delta}(Y_{s}^{(\delta)})ds\Big)\I_{\left\{ t<\tau_{\delta}(Y^{(\delta)})\right\} }\right],\label{eq:feynman_Kac}
\end{align}
where $\tilde{\P}_{x}$ and $\tilde{\E}_{x}$ indicate the initial
value and $\tau_{\delta}(Y^{(\delta)}):=\inf\{t\geq0:Y_{t}^{(\delta)}\notin\Lambda_{\delta}\}$.
Assume first this holds true. Denote the transition densities of $Y^{(\delta)}$
by $p_{\delta,t}(x,y)$, $x,y\in\R^{d}$. According to \citet[Eq. (1.4)]{Sheu:1991ep}
we have $p_{\delta,t}(x,y)\leq c_{3}q_{c_{2}t}(x-y)$ for universal
constants $c_{2},c_{3}>0$. Then by (\ref{eq:feynman_Kac}), using
$\norm{\tilde{b}_{\delta}}_{\infty}\leq c_{1}\delta^{2}$ for some
constant $c_{1}>0$, it follows
\begin{align*}
\left|(S_{\theta,\delta}^{*}\left(t\right)z)\left(x\right)\right| & \leq e^{c_{1}t\delta^{2}}\tilde{\E}_{x}\left[|z(Y_{t}^{(\delta)})|\right]=e^{c_{1}t\delta^{2}}\int_{\R^{d}}|z(y)|p_{\delta,t}(x,y)dy\\
 & \le c_{3}e^{c_{1}t\delta^{2}}q_{c_{2}t}*|z|(x)=c_{3}e^{c_{1}t\delta^{2}}(e^{c_{2}t\Delta}\left|z\right|)(x).
\end{align*}
We are left with showing (\ref{eq:feynman_Kac}). The proof is similar
to \citet[Theorem 6.5.2]{1975iv} and extends \citet[Theorem 7.44]{peres2010},
which applies only to Brownian motion. It is enough to consider $x\in\overline{\Lambda}_{\delta}$,
because otherwise $(S_{\theta,\delta}^{*}(t)z)(x)=0$ and $\I_{\{t<\tau_{\delta}(Y_{\delta})\}}=0$
$\tilde{\P}_{x}$-a.s. and so (\ref{eq:feynman_Kac}) holds trivially.
The function $u(t)=u(t,\cdot)$ with $u(t,x):=(S_{\theta,\delta}^{*}(t)z)(x)$
for $t\geq0$, $x\in\overline{\Lambda}_{\delta}$, is the unique solution
in $L^{2}(\Lambda_{\delta})$ of
\[
\begin{cases}
\frac{d}{dt}u\left(t\right)=A_{\theta,\delta}^{*}u\left(t\right),\,\,\,\,t>0,\\
u(0)=z,\,\,\,\,u(t)|_{\partial\Lambda_{\delta}}=0,\,\,\,t\geq0,
\end{cases}
\]
where the derivative is taken in $L^{2}(\Lambda_{\delta})$. Classical
PDE theory yields $u\in C([0,\infty),\overline{\Lambda}_{\delta})\cap C^{1,2}([\epsilon,\infty),\overline{\Lambda}_{\delta})$
for any $\epsilon>0$, see for example \citet[Theorem 6.3.6]{1975iv}
(here we use the regularity assumptions on $\theta, a, b$).
Set $h(t)=\exp(\int_{0}^{t}\tilde{b}_{\delta}(Y_{s}^{(\delta)})ds)$
and let $\rho=\inf\{t\geq0:Y_{t}^{(\delta)}\notin U\}$ for a compact
set $U\subset\Lambda_{\delta}$. Set $g(t',x)=u(t-t',x)$, $0\leq t'\leq t$.
Noting that $A_{\theta,\delta}^{*}-\tilde{b}_{\delta}$ generates
the transition semigroup of $Y^{(\delta)}$ and $h'(t)=\tilde{b}_{\delta}(Y_{t}^{(\delta)})h(t)$,
Itô's formula shows for any $0\leq t'<t$
\begin{align*}
 & g\left(t'\wedge\rho,Y_{t'\wedge\rho}^{(\delta)}\right)h\left(t'\wedge\rho\right)=g\left(0,Y_{0}^{(\delta)}\right)+\int_{0}^{t'\wedge\rho}\bigg(A_{\theta,\delta}^{*}g(s,\cdot)(Y_{s}^{(\delta)})\\
 & \qquad-\frac{d}{ds}g(s,Y_{s}^{(\delta)})\bigg)h(s)ds+\int_{0}^{t'\wedge\rho}\left\langle \nabla g\left(s,Y_{s}^{(\delta)}\right)h\left(s\right),d\tilde{W}_{s}\right\rangle _{\R^{d}}.
\end{align*}
Using the previous display, the second term vanishes. Taking expectations
and letting $t'\rightarrow t$ yields therefore
\begin{align*}
 & \left(S_{\theta,\delta}^{*}\left(t\right)z\right)\left(x\right)=g\left(0,Y_{0}^{(\delta)}\right)=\tilde{\E}_{x}\left[u\left(t-t\wedge\rho,Y_{t\wedge\rho}^{(\delta)}\right)h\left(t\wedge\rho\right)\right]\\
 & \quad=\tilde{\E}_{x}\left[z(Y_{t}^{(\delta)})h\left(t\right)\I_{\left\{ \rho>t\right\} }\right]+\tilde{\E}_{x}\left[u(t-\rho,Y_{\rho}^{(\delta)})h\left(\rho\right)\I_{\left\{ \rho\le t\right\} }\right].
\end{align*}
If $U$ exhausts $\Lambda_{\delta}$, then $\rho\rightarrow\tau_{\delta}(Y^{(\delta)})$
and $Y_{\rho}^{(\delta)}\rightarrow0$ such that $u(t-\rho,Y_{\rho}^{(\delta)})\rightarrow0$.
This implies (\ref{eq:feynman_Kac}).

(ii). We can assume $z\in C(\overline{\Lambda}_{\delta})$ for sufficiently
small $\delta$. Indeed, for $z\in L^{2}(\R^{d})$ let $z^{(\epsilon)}\in C_{c}(\R^{d})$
converge to $z$ in $L^{2}(\R^{d})$ as $\epsilon\rightarrow0$. For
small $\delta$ we have $z^{(\epsilon)}\in C(\overline{\Lambda}_{\delta})$.
Applying Proposition \ref{prop:semigroup_bounds}(i) to $S_{\theta,\delta}^{*}\left(t\right)(z|_{\Lambda_{\delta}}-z^{(\epsilon)})$,
Lemma \ref{lem:LaplaceSemigroup_props}(i) to $e^{\theta(0)t\Delta}(z-z^{(\epsilon)})$
we have
\begin{align*}
 & \norm{(S_{\theta,\delta}^{*}\left(t\right)(z|_{\Lambda_{\delta}})-e^{\theta(0)t\Delta}z}_{L^{2}(\R^{d})}\\
 & \quad\lesssim(e^{c_{1}\delta^{2}t}+1)\norm{z-z^{(\epsilon)}}_{L^{2}(\R^{d})}+\norm{(S_{\theta,\delta}^{*}\left(t\right)-e^{\theta(0)t\Delta})z^{(\epsilon)}}_{L^{2}(\R^{d})}.
\end{align*}
Using the statement with respect to $z^{(\epsilon)}$, and letting
first $\delta\rightarrow0$ and then $\epsilon\rightarrow0$, the
last line tends to zero.

For $z\in C(\overline{\Lambda}_{\delta})$ it is enough to show $(S_{\theta,\delta}^{*}(t)z)(x)\rightarrow(e^{\vartheta(0)\Delta t}z)(x)$
pointwise for $x\in\R^{d}$. $L^{2}(\R^{d})$-covergence follows then
from (i) and dominated convergence. Using the notation from (i) we
have the representation $(e^{\vartheta(0)t\Delta}z)\left(x\right)=\tilde{\E}_{x}[z(Y_{t}^{(0)})]$
for $Y_{t}^{(0)}=x+\sqrt{2}\theta(0)^{1/2}\tilde{W}_{t}$. (\ref{eq:feynman_Kac})
therefore allows us to write $S_{\theta,\delta}^{*}\left(t\right)z-e^{\theta(0)t\Delta}z=:T_{1}+T_{2}+T_{3}$
with
\begin{align*}
T_{1} & =\tilde{\E}_{x}\left[z(Y_{t}^{(\delta)})-z(Y_{t}^{(0)})\right],\\
T_{2} & =\tilde{\E}_{x}\left[z(Y_{t}^{(\delta)})\left(\exp\left(\int_{0}^{t}\tilde{b}_{\delta}(Y_{s}^{(\delta)})ds\right)-1\right)\I_{\left\{ t<\tau_{\delta}(Y^{(\delta)})\right\} }\right],\\
T_{3} & =-\tilde{\E}_{x}\left[z(Y_{t}^{(\delta)})\I_{\left\{ t\geq\tau_{\delta}(Y^{(\delta)})\right\} }\right].
\end{align*}
We shall show that $T_{i}\rightarrow0$, $i=1,2,3$. The transition
semigroup of $(Y_{t}^{(0)})_{t\geq0}$ is generated by $\c A^{(0)}=\theta(0)\Delta$.
Since $\c A^{(\delta)}f\rightarrow\c A^{(0)}f$ uniformly on $\R^{d}$
for $f\in C_{c}^{\infty}(\R^{d})$ as $\delta\rightarrow0$, it follows
from \citet[Theorem 19.25]{kallenberg2002foundations} that $Y^{(\delta)}\r dY^{(0)}$
with respect to the uniform topology on compacts in $\R_{+}$. This
yields $T_{1}\rightarrow0$. As $z$ is bounded and $\sup_{s>0}|\tilde{b}_{\delta}(Y_{s}^{(\delta)})|\lesssim\delta^{2}$,
we also have $|T_{2}|\lesssim e^{Ct\delta^{2}}\delta^{2}$ and $|T_{3}|\lesssim\tilde{\P}_{x}(\tau_{\delta}(Y^{(\delta)})\leq t)$.
To see why $\tilde{\P}_{x}(\tau_{\delta}(Y^{(\delta)})\leq t)\rightarrow0$
holds let $Z_{s}^{(\delta)}=Y_{s}^{(\delta)}-x-\int_{0}^{s}\tilde{a}_{\delta}(Y_{s'}^{(\delta)})ds'$
and observe that $|\int_{0}^{s}\tilde{a}_{\delta}(Y_{s'}^{(\delta)})ds'|\lesssim\delta t$
such that
\begin{align*}
\tilde{\P}_{x}(\tau_{\delta}(Y^{(\delta)})\leq t) & \leq\sum_{i=1}^{d}\tilde{\P}_{x}\left(\max_{0\leq s\leq t}|Z_{s}^{(\delta,i)}|\geq C\delta^{-1}\right),
\end{align*}
where $Z^{(\delta)}=(Z^{(\delta,i)})_{1\leq i\leq d}$. Since each
$Z^{(\delta,i)}$ is a continuous martingale vanishing at $0$ such
that $\langle Z^{(\delta,i)}\rangle_{s}=2\int_{0}^{s}\theta(\delta Y_{s'}^{(\delta)})ds'\leq cs$,
$c>0$, uniformly in $i=1,\dots,d$, we find for some scalar Brownian
motion $(\tilde{B}_{s})_{s\geq0}$ and $\tilde{c}>0$
\begin{align*}
\tilde{\P}_{x}(\tau_{\delta}(Y^{(\delta)})\leq t) & \leq d\tilde{\P}_{x}\left(\max_{0\leq s\leq t}|\tilde{B}_{cs}|\geq C\delta^{-1}\right)\lesssim e^{-\tilde{c}\delta^{-2}t^{-1}},
\end{align*}
because the density of the running maximum of a Brownian motion decays exponentially
(\citet[Chapter 2.8]{Karatzas1991}). This yields $T_{3}\rightarrow0$.
\end{proof}

\begin{lem}
\label{lem:v0} Let $z\in H^{2}(\R^{d})$ have compact support in $\Lambda_{\delta'}$ for some $\delta'>0$. For $0<\delta\le\delta'$ set $v^{(\delta)}:=\delta^{-1}(A_{\theta,\delta}^{*}-\theta(0)\Delta)z$ and define 
\[
v:=\Delta(\scapro{\nabla\theta(0)}{x}_{\R^{d}}z)(x)-\scapro{\nabla\vartheta(0)+a(0)}{\nabla z(x)}_{\R^{d}}.
\]
Then the following holds:
\begin{enumerate}
\item $\norm{v^{(\delta)}}_{L^{1}\cap L^{2}(\R^{d})}\leq C\norm{z}_{W_{1,2}^{2}(\R^{d})}$
and $v^{(\delta)}\rightarrow v$ in $L^{2}(\R^{d})$ for $\delta\rightarrow0$.

\item If $\int_{\R^{d}}z(x)dx=0$, then $v=\Delta m$ for $m\in H^{2}(\R^{d})$ with compact support. Moreover, if $\theta\in C^{1+\alpha'}(\overline{\Lambda})$ for $0\leq \alpha' \leq 1$, then
$\norm{v^{(\delta)}-v}_{L^1\cap L^2(\R^d)}\leq C\delta^{\alpha'}\norm z_{W_{1,2}^{2}(\R^{d})}$.
\item If  $\int_{\R^{d}}z(x)dx=0$
and $\int_{\R^{d}}xz(x)dx=0$, then $z=\Delta m$ for $m\in H^{4}(\R^{d})$
with compact support.
\end{enumerate}
\end{lem}

\begin{proof}
(i). Without loss of generality let $\theta,a,b$ as well as the partial derivatives of $\theta,a$ be bounded on $\R^d$. Then for $x\in\R^d$
\begin{align}
v^{(\delta)}(x) & =\frac{\vartheta\left(\delta x\right)-\vartheta\left(0\right)}{\delta}\Delta z(x)+\left\langle \nabla\theta(\delta x)-a(\delta x),\nabla z(x)\right\rangle _{\R^{d}}\label{eq:v_delta}\\
 & +\delta\left(b\left(\delta x\right)-(\text{div}a)(\delta x)\right)z(x).\nonumber
\end{align}
From this obtain the upper bound on $v^{(\delta)}$ and the convergence in $L^2(\R^d)$ to
\begin{align*}
& \scapro{\nabla\theta(0)}{x}_{\R^{d}}\Delta z(x)-\scapro{-\nabla\vartheta(0)+a(0)}{\nabla z(x)}_{\R^{d}}=v(x).
\end{align*}

(ii). In order to find $m$, as $z$ has
compact support, it suffices to find a compactly supported function
$g\in H^{2}(\R^{d})$ with $\Delta g=\scapro{\nabla\vartheta(0)+a(0)}{\nabla z}_{\R^{d}}$
in $L^{2}(\R^{d})$ and to set $m:=\scapro{\nabla\theta(0)}{x}_{\R^{d}}z(x)-g(x)$.
Using the Fourier transform ${\cal F}g(\omega)=\int_{\R^{d}}g(x)e^{i\scapro{x}{\omega}}dx$
this means by usual Fourier calculus
\[
-\abs{\omega}^{2}{\cal F}g(\omega)=\scapro{\nabla\vartheta(0)+a(0)}{i\omega}_{\R^{d}}{\cal F}z(\omega),\quad\omega\in\R^{d}.
\]
By the compact support of $z$ and $\int_{\R^{d}}z(x)dx=0$ the Fourier
transform ${\cal F}z$ is analytic with ${\cal F}z(0)=0$. We can
thus define
\[
g(x):={\cal F}^{-1}[u](x)\text{ with }u(\omega):=\scapro{\nabla\vartheta(0)+a(0)}{-\abs{\omega}^{-1}i\omega}_{\R^{d}}\frac{{\cal F}z(\omega)}{\abs{\omega}}
\]
as the inverse Fourier transform of the $L^{2}$-function $u$. Noting
$z\in H^{2}(\R^{d})$ and $\abs{u(\omega)}\lesssim\abs{{\cal F}z(\omega)}$
for $\abs{\omega}\to\infty$, we see $g\in H^{2}(\R^{d})$ and $\Delta g=\scapro{\nabla\vartheta(0)+a(0)}{\nabla z}_{\R^{d}}$
in $L^{2}(\R^{d})$.

For compactness of $g$ we use the Paley-Wiener Theorem (\citet[Theorem II.7.22]{rudin2006functional})
to deduce from the compact support of $z$ that ${\cal F}z$ can be
extended to an entire function on $\C^{d}$, satisfying the exponential
growth condition $\abs{{\cal F}z(\omega)}\le\gamma_{N}(1+\abs{\omega})^{-N}\exp(r\abs{\Im(\omega)})$,
$\omega\in\C^{d}$, for all $N\in\N$ and suitable positive constants
$\gamma_{N}$, $r$. Hence, $u$ is the quotient of an entire function
and $\abs{\omega}^{2}$, which is also entire. A meromorphic function
with removable singularity extends continuously to an entire function.
Consequently, we can work with an entire function $u$, which by definition
satisfies the same exponential growth condition. A reverse application
of the Paley-Wiener Theorem shows that $g$ has compact support.

Finally, we can assume that $\nabla \theta,a$  are uniformly $\alpha'$-H{\"o}lder continuous on $\R^d$. The upper bound on $\norm{v^{(\delta)}-v}_{L^1\cap L^2(\R^d)}$ follows then for $x\in\R^d$ using \eqref{eq:v_delta}  from
\begin{eqnarray*}
& |v^{(\delta)}(x) - v(x)| \lesssim \left|\frac{\vartheta(\delta x)-\theta(0)}{\delta} - \langle\nabla \theta(0),x\rangle_{\R^d}\right| |\Delta z(x)| \\
& \qquad \qquad+ \left|\nabla \theta(\delta x) -\nabla \theta(0) + a(0)-a(\delta x)\right| |\nabla z(x)| + \delta |z(x)|.
\end{eqnarray*}

(iii). The argument is similar to (ii). As above, the Fourier transform
$\c Fz$ is analytic with $\c Fz(0)=0$. The assumption $\int_{\R^{d}}x_{i}z(x)dx=0$ gives
also $\partial_{i}(\c Fz)(0)=0$, $i=1,\dots,d$. It follows for
\[
m(x):={\cal F}^{-1}[u](x)\text{ with }u(\omega):=-\frac{{\cal F}z(\omega)}{\abs{\omega}^{2}}
\]
that $m\in H^{4}(\R^{d})$ and $\Delta m=z$. A Paley-Wiener
argument as in (ii) shows that $m$ has compact support.
\end{proof}

The following heat kernel bounds will be used frequently. The conditions
in (iii) are essential for $d=1$ to improve on (ii).

\begin{lem}
\label{lem:S_delta_bounds} Let the functions $u,w\in L^{2}(\R^{d})$, $z\in H^{2}(\R^{d})$
have compact support in $\Lambda_\delta$ for some $\delta>0$. Then for $0<t\leq T\delta^{-2}$:
\begin{enumerate}
\item $\norm{S_{\theta,\delta}^{*}\left(t\right)u}_{L^{2}(\Lambda_{\delta})}\leq e^{CT}(1\wedge t^{-d/4})\norm u_{L^1\cap L^2(\R^d)}.$
\item If $\norm{w-\Delta z}_{L^1\cap L^2(\R^d)}\leq C\delta^{\alpha'}\norm{z}_{W_{1,2}^2(\R^d)}$ for $0\leq \alpha'\leq 1$, then
\[\norm{S_{\theta,\delta}^{*}(t)w}_{L^{2}(\Lambda_{\delta})}\leq e^{CT}(1\wedge t^{-\alpha'/2-d/4})\norm{z}_{W^2_{1,2}(\R^{d})}.\]
\item If $\theta\in C^{1+\alpha'}(\overline{\Lambda})$ for $0\leq \alpha'\leq 1$ and $\int_{\R^{d}}z(x)dx=0$,
then
\[\norm{S_{\theta,\delta}^{*}(t)\Delta z}_{L^{2}(\Lambda_{\delta})}\leq e^{CT}(1\wedge t^{-1/2-\alpha'/2-d/4})\norm{z}_{W^2_{1,2}(\R^{d})}.\]
\end{enumerate}
\end{lem}

\begin{proof}
(i). The semigroup bound in Proposition \ref{prop:semigroup_convergence}(i),
applied to a sequence of continuous functions approximating $u$,
and Lemma \ref{lem:LaplaceSemigroup_props}(i) show for $t\leq T\delta^{-2}$
\begin{align*}
\norm{S_{\theta,\delta}^{*}\left(t\right)u}_{L^{2}(\Lambda_{\delta})} & \lesssim e^{CT}\norm{e^{c_{2}t\Delta}|u|}_{L^{2}(\R^{d})}\lesssim e^{CT}(1\wedge t^{-d/4})\norm u_{L^1\cap L^2(\R^d)}.
\end{align*}
(ii). Write $w=u^{(\delta)}+\theta(0)^{-1}A_{\theta,\delta}^{*}z$ for $u^{(\delta)}=(w-\Delta z)-\theta(0)^{-1}\delta v^{(\delta)}$ and $v^{(\delta)}=\delta^{-1}(A_{\theta,\delta}^{*}-\theta(0)\Delta)z$ such that
\begin{align*}
\norm{S_{\theta,\delta}^{*}(t)w}_{L^{2}\left(\Lambda_{\delta}\right)} & \leq\norm{S_{\theta,\delta}^{*}\left(t\right)u^{(\delta)}}_{L^{2}(\Lambda_{\delta})}+\theta(0)^{-1}\norm{A_{\theta,\delta}^{*}S_{\theta,\delta}^{*}(t)z}_{L^{2}(\Lambda_{\delta})}.
\end{align*}
The second term is up to a constant bounded by $e^{CT}(1\wedge t^{-1})\norm{S_{\theta,\delta}^{*}(t/2)z}_{L^{2}(\Lambda_{\delta})}$, using Proposition \ref{prop:semigroup_bounds}(ii) for $t\leq T\delta^{-2}$
and $S_{\theta,\delta}^{*}(t)=S_{\theta,\delta}^{*}(t/2)S_{\theta,\delta}^{*}(t/2)$. Applying (i) to $u=z$ gives the upper bound $e^{CT}(1\wedge t^{-1-d/4})\norm{z}_{L^2(\R^d)}$. The result follows from applying (i) to $u=u^{(\delta)}$ in the last display and  noting that $\norm{u^{(\delta)}}_{L^1\cap L^2(\R^d)}\leq e^{CT}(1\wedge t^{-\alpha'/2})\norm{z}_{W^{2}_{1,2}(\R^{d})}$ by Lemma \ref{lem:v0}(i) and $\delta \leq (T/t)^{1/2}$, as well as adjusting the constant $C$.
\begin{comment}
\begin{align*}
\norm{S_{\theta,\delta}^{*}(t)w}_{L^{2}\left(\Lambda_{\delta}\right)} & \leq\norm{S_{\theta,\delta}^{*}\left(t\right)u^{(\delta)}}_{L^{2}(\Lambda_{\delta})}+\theta(0)^{-1}\norm{A_{\theta,\delta}^{*}S_{\theta,\delta}^{*}(t)z}_{L^{2}(\Lambda_{\delta})} \\
& \lesssim \norm{S_{\theta,\delta}^{*}\left(t\right)u^{(\delta)}}_{L^{2}(\Lambda_{\delta})}+e^{CT}(1\wedge t^{-1})\norm{S_{\theta,\delta}^{*}(t/2)z}_{L^{2}(\Lambda_{\delta})},
\end{align*}
using Proposition \ref{prop:semigroup_bounds}(ii) for $t\leq T\delta^{-2}$
and $S_{\theta,\delta}^{*}(t)=S_{\theta,\delta}^{*}(t/2)S_{\theta,\delta}^{*}(t/2)$. The result follows then from applying (i) to $u=u^{(\delta)}$ and $u=z$, noting that $\norm{u^{(\delta)}}_{L^1\cap L^2(\R^d)}\lesssim e^{CT}(1\wedge t^{-1/2})\norm{z}_{W^{2}_{1,2}(\R^{d})}$ by Lemma \ref{lem:v0}(i) and $\delta \leq (T/t)^{1/2}$, as well as adjusting the constant $C$.
\end{comment}

(iii). Following the proof of (ii) for $w=\Delta z$ it is enough to show the improved upper bound $\norm{S_{\theta,\delta}^{*}\left(t\right)u^{(\delta)}}_{L^{2}(\Lambda_{\delta})}\lesssim e^{CT} (1\wedge t^{-1/2-\alpha'/2-d/4})\norm{z}_{W_{1,2}^2(\R^d)}$. Lemma \ref{lem:v0}(ii) shows the existence of a compactly
supported $m\in H^{2}(\R^{d})$ such that $\norm{v^{(\delta)}-\Delta m}_{L^1\cap L^2(\R^d)}\leq C\delta^{\alpha'}\norm z_{W_{1,2}^{2}(\R^{d})}$. With $\tilde{u}^{(\delta)}=v^{(\delta)}-\Delta m$ write $u^{(\delta)}=\theta(0)^{-1}\delta\tilde{u}^{(\delta)}+\theta(0)^{-1}\delta\Delta m$. Applying (i) to $u=\tilde{u}^{(\delta)}$ and (ii) to $z=m$ yields
\[
\norm{S_{\theta,\delta}^{*}\left(t\right)u^{(\delta)}}_{L^{2}(\Lambda_{\delta})}\lesssim e^{CT} (\delta^{1+\alpha'}(1\wedge t^{-d/4}) + \delta (1\wedge t^{-1/2-d/4})) \norm{z}_{W_{1,2}^2(\R^d)}.
\]
For $\delta \leq (T/t)^{1/2}$ the order in $t$ is $1\wedge t^{-1/2-\alpha'/2-d/4}$, as claimed.
\end{proof}
\begin{lem}
\label{lem:X0} Let $u\in L^{2}(\R^{d})$, $z\in H^{2}(\R^{d})$ have
compact support in $\Lambda_{\delta}$ for some $\delta>0$. Using $\ell_{d,2}(\delta)$ as in Assumption \ref{assu:X0}, we have:
\begin{enumerate}
\item $\int_{0}^{T}\left\langle S_{\theta}\left(t\right)X_{0},u_{\delta}\right\rangle ^{2}dt\leq e^{CT}\norm{X_0}^2\ell_{d,2}(\delta)\delta^{2\wedge d}\norm u_{L^1\cap L^2(\R^d)}^{2}$.
\item If $X_{0}\in L^{p}(\Lambda)$, $p\geq2$, $1/p+1/p'=1$, then, with $\gamma(d,p):=2\frac{1+d/p}{1+d/2}+d(1-\frac{2}{p})$,
\[
 \int_{0}^{T}\langle S_{\theta}(t)X_{0},(\Delta z)_{\delta}\rangle^{2}dt\leq e^{CT}\norm{X_{0}}_{L^{p}(\Lambda)}^{2}\delta^{\gamma(d,p)}(\norm{\Delta z}_{L^{p'}(\R^{d})}^{2}+\norm{z}^2_{W_{1,2}^{2}(\R^{d})}).
\]
\item If $X_{0}\in\c D(A_{\theta})$, then 
\[\int_{0}^{T}\langle S_{\theta}(t)X_{0},(\Delta z)_{\delta}\rangle^{2}dt\leq e^{CT}(\norm{X_0}^2 + \norm{A_{\theta}X_0}^2)\ell_{d,2}(\delta)\delta^{3}\norm{z}^2_{W_{1,2}^{2}(\R^{d})}.
\]
\end{enumerate}
\end{lem}

\begin{proof}
(i). By Lemma \ref{lem:S_delta_bounds}(i) and the scaling in Lemma
\ref{lem:scaling} we find
\begin{align*}
\int_{0}^{T}\left\langle S_{\theta}\left(t\right)X_{0},u_{\delta}\right\rangle ^{2}dt & \leq\norm{X_{0}}^{2}\delta^{2}\int_{0}^{T\delta^{-2}}\norm{S_{\theta,\delta}^{*}\left(t\right)u}_{L^{2}(\Lambda_{\delta})}^{2}dt\\
 & \lesssim e^{CT}\norm{X_{0}}^{2}\delta^{2}\int_{0}^{T\delta^{-2}}(1\wedge t^{-d/2})dt\norm u_{L^1\cap L^2(\R^d)}^{2}.
\end{align*}
The claim follows, because the integral has order $O(1)$ for $d\ge3$,
order $O(\log(T\delta^{-2}))$ for $d=2$ and order $O(T^{1/2}\delta^{-1})$
for $d=1$.

(ii). It is enough to consider continuous $X_{0}$. Using the Hölder
inequality and Proposition \ref{prop:semigroup_convergence}(i), we
obtain
\begin{align*}
\sc{S_{\theta}(t)X_{0}}{(\Delta z)_{\delta}}^{2} & \leq\norm{S_{\theta}(t)X_{0}}_{L^{p}(\Lambda)}^{2}\norm{\left(\Delta z\right)_{\delta}}_{L^{p'}(\Lambda)}^{2}\\
 & \lesssim\norm{q_{c_{2}t}*\left|X_{0}\right|}_{L^{p}(\R^{d})}^{2}\delta^{d(2/p'-1)}\norm{\Delta z}_{L^{p'}(\R^{d})}^{2}.
\end{align*}
Here, $\norm{q_{c_{2}t}*\left|X_{0}\right|}_{L^{p}(\R^{d})}\leq \norm{X_0}_{L^p(\Lambda)}$. For $\epsilon>0$, Lemmas \ref{lem:scaling} and \ref{lem:S_delta_bounds}(ii) show
\begin{align*}
\int_{\epsilon}^{T}\langle S_{\theta}(t)X_{0},\left(\Delta z\right)_{\delta}\rangle^{2}dt & \leq\norm{X_{0}}^{2}\int_{\epsilon}^{T}\norm{S_{\theta,\delta}^{*}(t\delta^{-2})\Delta z}_{L^{2}(\Lambda_{\delta})}^{2}dt\\
 & \lesssim e^{CT}\norm{X_{0}}^{2}\int_{\epsilon}^{T}(t\delta^{-2})^{-1-d/2}dt\,\norm{z}_{W_{1,2}^{2}(\R^d)}^2.
\end{align*}
Splitting up the integral and adjusting the constant $C$ yields thus
\begin{align*}
\int_{0}^{T}\langle S_{\theta}(t)X_{0},\left(\Delta z\right)_{\delta}\rangle^{2}dt & \lesssim e^{CT}\norm{X_{0}}^{2}(\delta^{d(2/p'-1)}\epsilon+\delta^{2+d}\epsilon^{-d/2})\\
 & \quad \cdot(\norm{\Delta z}_{L^{p'}(\R^{d})}^{2}+\norm{z}_{W_{1,2}^{2}(\R^d)}^2).
\end{align*}
The claim follows with $\norm{X_{0}}\lesssim\norm{X_{0}}_{L^{p}(\Lambda)}$ and $\epsilon=\delta^{2\frac{1+d/p}{1+d/2}}$.

(iii). With $v^{(\delta)}:=\delta^{-1}(A_{\theta,\delta}^{*}-\theta(0)\Delta)z$
as in Lemma \ref{lem:v0} and using the scaling in Lemma \ref{lem:scaling},
write $\theta(0)(\Delta z)_{\delta}=-\delta v_{\delta}^{(\delta)}+\delta^{2}A_{\theta}^{*}z_{\delta}$.
Then
\[
\langle S_{\theta}(t)X_{0},(\Delta z)_{\delta}\rangle^{2}\lesssim\delta^{2}\langle S_{\theta}(t)X_{0},v_{\delta}^{(\delta)}\rangle^{2}+\delta^{4}\langle S_{\theta}(t)A_{\theta}X_{0},z_{\delta}\rangle^{2},
\]
and the claim follows from applying (i) with $u=v$ and $u=z$ (with
$A_{\theta}X_{0}\in L^{2}(\Lambda)$ instead of $X_{0}$).
\end{proof}

\subsection{\label{subsec:asympForCovars}Asymptotic results for the covariances}

The general idea for the proofs in this section is to apply the scaling
in Lemma \ref{lem:scaling} to the covariance function as in Section
\ref{subsec:From-bounded-to} and to deduce a limit for the integral
using the heat kernel bounds and the convergence of the semigroups
from the last section.
\begin{prop}
\label{prop:cov_asymp} Grant Assumption \ref{assu:B}. Consider functions $z\in H^{2}(\R^{d})$, $u\in L^{2}(\R^{d})$,
$(w^{(\delta)})_{\delta>0}$, $(u^{(\delta)})_{\delta>0}\subset L^{2}(\R^{d})$
with compact support in $\Lambda_{\delta'}$ for some $\delta'>0$. Assume for $0<\delta\leq\delta'$  that $\norm{w^{(\delta)}-\Delta z}_{L^1\cap L^2(\R^d)}\leq C \delta^{\alpha'}$ for $\alpha'>1/2$, 
$\norm{u^{(\delta)}-u}_{L^{2}(\R^{d})}\rightarrow0$ as $\delta\rightarrow0$. Then with $\Psi$ from \eqref{eq:Psi}:
\begin{enumerate}
\item $\delta^{-2}\Var(\sc{\tilde{X}(t)}{w_{\delta}^{(\delta)}})\rightarrow\theta(0)^{-1}\Psi(\Delta z,\Delta z)$,
$t>0$.
\item $\delta^{-2}\int_{0}^{T}\Var(\sc{\tilde{X}(t)}{w_{\delta}^{(\delta)}})dt\rightarrow\ T\theta(0)^{-1}\Psi(\Delta z,\Delta z)$.
\item If $d\geq2$, then $\delta^{-2}\int_{0}^{T}\Cov(\sc{\tilde{X}(t)}{w_{\delta}^{(\delta)}},\sc{\tilde{X}(t)}{u_{\delta}^{(\delta)}})dt\rightarrow\ T\theta(0)^{-1}\Psi(\Delta z,u)$.
\end{enumerate}
\end{prop}

\begin{proof}

(i). By \eqref{eq:scaled_cov_function} for $t=t'$ we have
\[
\delta^{-2}\Var(\langle \tilde{X}(t),w_{\delta}^{(\delta)}\rangle)=\int_{0}^{t\delta^{-2}}\norm{B_{\delta}^{*}S_{\theta,\delta}^{*}\left(s\right)w^{(\delta)}}_{L^{2}(\Lambda_{\delta})}^{2}ds = \int_{0}^{\infty}f_{\delta}(s)ds,
\]
with $f_{\delta}(s)=\norm{B_{\delta}^{*}S_{\theta,\delta}^{*}\left(s\right)w^{(\delta)}}_{L^{2}(\Lambda_{\delta})}^{2}\I_{\{s\leq t\delta^{-2}\}}
$. Set $f(s)=\norm{B_{0}^{*}e^{\theta(0)s\Delta}\Delta z}_{L^{2}(\R^{d})}^{2}$
and note $\int_{0}^{\infty}f(s)\theta(0)ds=\Psi(\Delta z,\Delta z)$,
substituting $ds'=\theta(0)ds$. By assumption $w^{(\delta)}\r{}\Delta z$
in $L^{2}(\Lambda_{\delta})$ and by Proposition \ref{prop:semigroup_convergence}(ii)
above $S_{\theta,\delta}^{*}(s)\Delta z\rightarrow e^{\theta(0)s\Delta}\Delta z$
in $L^{2}(\R^{d})$. From Proposition \ref{prop:semigroup_bounds}(i)
above we have $\sup_{0\leq s\le t/\delta^{2}}\norm{S_{\theta,\delta}^{*}(s)}_{L^{2}(\Lambda_{\delta})}<\infty$,
as well as $\sup_{0<\delta\leq 1}\norm{B_{\delta}^{*}}_{L^{2}(\R^{d})}<\infty$
by Assumption \ref{assu:B} and the uniform boundedness principle.
We deduce
\begin{align*}
 & \norm{B_{\delta}^{*}S_{\theta,\delta}^{*}(s)w^{(\delta)}-B_{0}^{*}e^{\theta(0)s\Delta}\Delta z}_{L^{2}(\R^{d})}\\
 & \le\norm{B_{\delta}^{*}}\left(\norm{S_{\theta,\delta}^{*}(s)}_{L^{2}(\Lambda_{\delta})}\norm{w^{(\delta)}-\Delta z}_{L^{2}(\R^{d})}+\norm{S_{\theta,\delta}^{*}(s)\Delta z-e^{\theta(0)s\Delta}\Delta z}_{L^{2}(\R^{d})}\right)\\
 & \quad+\norm{(B_{\delta}^{*}-B_{0}^{*})e^{\theta(0)s\Delta}\Delta z}_{L^{2}(\R^{d})}\rightarrow0,
\end{align*}
which implies $f_{\delta}(s)\r{}f(s)$ pointwise. Lemma \ref{lem:S_delta_bounds}(ii)
yields $\left|f_{\delta}\left(s\right)\right|\lesssim1\wedge s^{-\alpha'-d/2}$.
Since $\alpha'>1/2$, $\sup_{0<\delta\leq 1}f_{\delta}(\cdot)\in L^{1}([0,\infty))$,
for any fixed $t$, and the result follows from the dominated convergence
theorem.

(ii). By (i) and Fatou's lemma we obtain
\[
\liminf_{\delta\to0}\delta^{-2}\int_{0}^{T}\Var(\sc{\tilde{X}(t)}{w_{\delta}^{(\delta)}})dt\ge T\theta(0)^{-1}\Psi(\Delta z,\Delta z).
\]
On the other hand, $\Var(\sc{\tilde{X}(t)}{w_{\delta}^{(\delta)}})$ is increasing
in $t$, cf. (\ref{eq:cov_function}). The result follows from
\begin{align*}
\limsup_{\delta\to0}\delta^{-2}\int_{0}^{T}\Var(\sc{\tilde{X}(t)}{w_{\delta}^{(\delta)}})dt & \le\lim_{\delta\to0}\delta^{-2}T\Var(\sc{\tilde{X}(T)}{w_{\delta}^{(\delta)}})\\
 & =T\theta(0)^{-1}\Psi(\Delta z,\Delta z).
\end{align*}
(iii). Revisiting the derivations in (i) and (ii), we obtain
\begin{align*}
 & \delta^{-2}\Cov(\sc{\tilde{X}(t)}{w_{\delta}^{(\delta)}},\sc{\tilde{X}(t)}{u_{\delta}^{(\delta)}})=\int_{0}^{\infty}f_{\delta}(s)\,ds,\\
 & \text{with }\,\,f_{\delta}(s):=\scapro{B_{\delta}^{*}S_{\theta,\delta}^{*}\left(s\right)w^{(\delta)}}{B_{\delta}^{*}S_{\theta,\delta}^{*}\left(s\right)u^{(\delta)}}_{L^{2}(\Lambda_{\delta})}\I_{\{s\leq t\delta^{-2}\}}.
\end{align*}
Putting $f(s):=\scapro{B_{0}^{*}e^{\theta(0)s\Delta}\Delta z}{B_{0}^{*}e^{\theta(0)s\Delta}u}$,
we obtain as in (i), (iii) that $f_{\delta}(s)\to f(s)$ holds pointwise
for $\delta\to0$ by the $L^{2}$-continuity of the scalar product.
Furthermore, the Cauchy-Schwarz inequality and Lemma \ref{lem:S_delta_bounds}(i,ii)
yield the bound
\begin{align}
\abs{f_{\delta}(s)} & \lesssim\norm{S_{\theta,\delta}^{*}\left(s\right)w^{(\delta)}}_{L^{2}(\Lambda_{\delta})}\norm{S_{\theta,\delta}^{*}\left(s\right)u^{(\delta)}}_{L^{2}(\Lambda_{\delta})}\I_{\{s\le t/\delta^{2}\}}\nonumber \\
 & \lesssim e^{CT}(1\wedge s^{-\alpha'/2-d/2})\lesssim1\wedge s^{-5/4}\label{eq:covbound}
\end{align}
for $d\geq2$. Since this bound is integrable in $s\ge0$, we conclude
that
\[
\delta^{-2}\Cov(\sc{\tilde{X}(t)}{w_{\delta}^{(\delta)}},\sc{\tilde{X}(t)}{u_{\delta}^{(\delta)}})\rightarrow\int_{0}^{\infty}f(s)ds=\theta(0)^{-1}\Psi(\Delta z,u),
\]
meaning in particular that $\Psi(\Delta z,u)$ is well defined. What
is more, the bound (\ref{eq:covbound}) also shows that the covariance
is uniformly bounded in $t\in[0,T]$ so that another application of
the dominated convergence theorem shows that the integral over $t\in[0,T]$
converges to $T\theta(0)^{-1}\Psi(\Delta z,u)$.
\end{proof}

\begin{prop}
\label{prop:var_with_delta} Grant Assumption \ref{assu:B}. Consider functions $z,m\in H^{2}(\R^{d})$,
$u\in L^{2}(\R^{d})$, $(w^{(\delta)})_{\delta>0}$, $(u^{(\delta)})_{\delta>0}\subset L^{2}(\R^{d})$ with compact support in $\Lambda_{\delta'}$ for some $\delta'>0$. Assume for $0<\delta\leq\delta'$ that $\norm{w^{(\delta)}-\Delta m}_{L^1\cap L^2(\R^d)}\leq C\delta^{\alpha'}$ for $\alpha'>1/2$,
$\norm{u^{(\delta)}-u}_{L^{2}(\R^{d})}\rightarrow0$ as $\delta\rightarrow0$.
Then:
\begin{enumerate}
\item For $d\geq2$, $\Var(\int_{0}^{T}\langle \tilde{X}(t),(\Delta z)_{\delta}\rangle\langle \tilde{X}(t),u_{\delta}^{(\delta)}\rangle dt)=O(\delta^{6}\ell_{d,2}(\delta)^{3})$.
\item $\Var(\int_{0}^{T}\langle \tilde{X}(t),(\Delta z)_{\delta}\rangle\langle \tilde{X}(t),w_{\delta}^{(\delta)}\rangle dt)=o(\delta^{4})$.
\item For $d\geq2$, $\Var(\int_{0}^{T}\langle \tilde{X}(t),u_{\delta}^{(\delta)}\rangle^{2}dt)=O(\delta^{4}\ell_{d,2}(\delta)^{2})$.
\item Let $d\geq2$, or let $\theta\in C^{1+\alpha'}(\overline{\Lambda})$ and $\int_{\R^{d}}m(x)dx=0$.
Then with $\Psi$ from (\ref{eq:Psi})
\[
\delta^{-6}\Var(\int_{0}^{T}\langle \tilde{X}(t),w_{\delta}^{(\delta)}\rangle^{2}dt)\rightarrow4T\theta(0)^{-3}\int_{0}^{\infty}\Psi(e^{s\Delta}\Delta m,\Delta m)^{2}ds.
\]
\end{enumerate}
\end{prop}

\begin{proof}

We first make some preliminary remarks. For $v,\tilde{v}\in L^{2}(\Lambda_{\delta})$
set $\xi(t)=\langle \tilde{X}(t),v_{\delta}\rangle$, $\tilde{\xi}(t)=\langle \tilde{X}(t),\tilde{v}_{\delta}\rangle$.
The random variables $\{\xi(t)\,|\,t\ge0\}\cup\{\tilde{\xi}(t)\,|\,t\ge0\}$
are jointly Gaussian and centered and so it follows from Wick's formula
(\citet[Theorem 1.28]{Janson:1997uy}) that

\begin{align}
 & \delta^{-6}\Var(\int_{0}^{T}\xi(t)\tilde{\xi}(t)dt)=\delta^{-6}\int_{0}^{T}\int_{0}^{T}\text{Cov}(\xi(t)\tilde{\xi}(t),\xi(s)\tilde{\xi}(s))dtds\nonumber \\
 & =\delta^{-6}\int_{0}^{T}\int_{0}^{T}(\text{Cov}(\xi(t),\tilde{\xi}(s))^{2}+\text{Cov}(\xi(t),\tilde{\xi}(s))\Cov(\tilde{\xi}(t),\xi(s)))dtds\nonumber \\
 & =:2V_{1}+2V_{2},\label{eq:var}
\end{align}
with $V_{1}=V(v,v,\tilde{v},\tilde{v})$, $V_{2}=V(v,\tilde{v},\tilde{v},v)$,
where for $k,\tilde{k}\in L^{2}(\Lambda_{\delta})$
\begin{align}
V(v,\tilde{v},k,\tilde{k}) & :=\delta^{-6}\int_{0}^{T}\int_{0}^{t}\text{Cov}(\sc{\tilde{X}(t)}{v_{\delta}},\sc{\tilde{X}(s)}{\tilde{v}_{\delta}})\nonumber \\
 & \qquad\cdot\text{Cov}(\sc{\tilde{X}(t)}{k_{\delta}},\sc{\tilde{X}(s)}{\tilde{k}_{\delta}})dsdt,\label{eq:V}
\end{align}
and $V(v):=V(v,v,v,v)$. It is thus enough to study $V_{1},V_{2}$.
Set 
\[f_{\delta}((s,v),(s',\tilde{v}))=\langle B_{\delta}^{*}S_{\theta,\delta}^{*}(s)v,B_{\delta}^{*}S_{\theta,\delta}^{*}(s')\tilde{v}\rangle_{L^{2}(\Lambda_{\delta})}\I_{\{0<s,s'\leq T\delta^{-2}\}},\,\,s,s'\geq0.
\] 
Then by \eqref{eq:scaled_cov_function}, $V(v,\tilde{v},k,\tilde{k})$ equals
\begin{align}
 & \delta^{-2}\int_{0}^{T}\int_{0}^{t}\Big(\int_{0}^{s\delta^{-2}}f_{\delta}((t\delta^{-2}-r',v),(s\delta^{-2}-r',\tilde{v}))dr'\Big)\nonumber \\
 & \qquad\cdot\Big(\int_{0}^{s\delta^{-2}}f_{\delta}((t\delta^{-2}-r'',k),(s\delta^{-2}-r'',\tilde{k}))dr''\Big)dsdt\nonumber \\
 & =\int_{0}^{T}\int_{0}^{t\delta^{-2}}\Big(\int_{0}^{t\delta^{-2}-s}f_{\delta}((s+s',v),(s',\tilde{v}))ds'\Big)\nonumber \\
 & \qquad\cdot\Big(\int_{0}^{t\delta^{-2}-s}f_{\delta}((s+s'',k),(s'',\tilde{k}))ds''\Big)dsdt,\label{eq:f_g_delta}
\end{align}
substituting $ds'=s\delta^{-2}-dr'$, $ds''=s\delta^{-2}-dr''$ and $ds'''=\delta^{-2}(t-ds)$, but writing again $s$ for $s'''$. With this we prove now the Proposition.

(i). Let $v=\Delta z$, $\tilde{v}=u^{(\delta)}$. By the rescaling
with $\delta^{-6}$ it is enough to show $V_{i}=O(\ell_{d,2}(\delta)^{3})$,
$i=1,2$. Observe by Lemma \ref{lem:S_delta_bounds}(i,ii) that
\begin{align}
|f_{\delta}((s+s'',u^{(\delta)}),(s'',u^{(\delta)}))| & \lesssim\norm{S_{\theta,\delta}^{*}(s+s'')u^{(\delta)}}_{L^{2}(\Lambda_{\delta})}\norm{S_{\theta,\delta}^{*}(s'')u^{(\delta)}}_{L^{2}(\Lambda_{\delta})}\nonumber \\
 & \lesssim(1\wedge(s+s'')^{-d/4})(1\wedge(s'')^{-d/4})\lesssim1\wedge(s'')^{-d/2},\label{eq:f_delta_bounds_0}\\
\left|f_{\delta}\left((s+s',\Delta z),(s',\Delta z)\right)\right| & \lesssim\norm{S_{\theta,\delta}^{*}(s+s')\Delta z}_{L^{2}(\Lambda_{\delta})}\norm{S_{\theta,\delta}^{*}(s')\Delta z}_{L^{2}(\Lambda_{\delta})}\nonumber \\
 & \lesssim(1\wedge(s+s')^{-1/2-d/4})(1\wedge(s')^{-1/2-d/4})\label{eq:f_delta_bounds}\\
 & \leq(1\wedge s{}^{-1/2-d/4})(1\wedge(s')^{-1/2-d/4}).\nonumber
\end{align}
These bounds yield in (\ref{eq:f_g_delta}) when $d\geq2$ for $V(\Delta z,\Delta z,u^{(\delta)},u^{(\delta)})$
up to a constant the upper bound
\begin{align}
 & \Big(\int_{0}^{T\delta^{-2}}\big(1\wedge(s'')^{-d/2}\big)ds''\Big)\Big(\int_{0}^{T\delta^{-2}}\big(1\wedge s^{-1/2-d/4}\big)ds\Big)\nonumber \\
 & \qquad\cdot\Big(\int_{0}^{T\delta^{-2}}\big(1\wedge(s')^{-1/2-d/4}\big)ds'\Big)\lesssim\ell_{d,2}(\delta)^{3}.\label{eq:f_delta_bounds_1}
\end{align}
Similarly, $\left|f_{\delta}((s+s'',\Delta z),(s'',u^{(\delta)}))\right|\lesssim(1\wedge s{}^{-d/4})(1\wedge(s'')^{-1/2-d/4})$,
$\left|f_{\delta}\left((s+s',u^{(\delta)}),(s',\Delta z)\right)\right|\lesssim(1\wedge s{}^{-d/4})(1\wedge(s')^{-1/2-d/4})$,
implying the upper bound $\ell_{d,2}(\delta)^{3}$ also for $V(\Delta z,u^{(\delta)},u^{(\delta)},\Delta z)$.
In all, we find $|V_{1}|$, $|V_{2}|\lesssim\ell_{d,2}(\delta)^{3}$.

(ii). Let $v=\Delta z$, $\tilde{v}=w^{(\delta)}$. By the rescaling of $V$ it is
enough to show $\delta^2 V_{i}\rightarrow 0$, $i=1,2$. We have
by Lemma \ref{lem:S_delta_bounds}(ii) for any $v_{1},v_{2}\in\{v,\tilde{v}\}$
\begin{align*}
\left|f_{\delta}((s+s'',v_{1}),(s'',v_{2}))\right| & \lesssim(1\wedge(s+s'')^{-\alpha'/2-d/4})(1\wedge(s'')^{-\alpha'/2-d/4}),\\
\left|f_{\delta}((s+s',v_{1}),(s',v_{2}))\right| & \lesssim(1\wedge s^{-\alpha'/2-d/4})(1\wedge(s')^{-\alpha'/2-d/4}).
\end{align*}
Therefore, as in (\ref{eq:f_delta_bounds_1}) but this time for all
$d\geq1$, $|V(v_{1},v_{2},v_{3},v_{4})|\lesssim\delta^{4(\alpha'-1)}$,
$v_{3},v_{4}\in\{v,\tilde{v}\}$.

(iii). Let $v=\tilde{v}=u^{(\delta)}$. The claim is a direct consequence
of (\ref{eq:f_delta_bounds_0}) and (\ref{eq:f_delta_bounds_1}) for
$d\geq2$, with the $ds$-integral of order $O(\delta^{-2})$ this
time.

(iv). Let $v=\tilde{v}=w^{(\delta)}$. Since $V_{1}=V_{2}=V(w^{(\delta)})$,
it is enough to show
\begin{equation}
V_{1}=V(w^{(\delta)})\rightarrow T\theta(0)^{-3}\int_{0}^{\infty}\Psi(e^{s\Delta}\Delta m,\Delta m)^{2}ds.\label{eq:V_lim}
\end{equation}
We argue by dominated convergence. Set
\[
f((s,\Delta m),(s',\Delta m))=\langle B_{0}^{*}e^{\theta(0)s\Delta}\Delta m,B_{0}^{*}e^{\theta(0)s'\Delta}\Delta m\rangle _{L^{2}(\R^{d})}.
\]
Exactly as in the proof of Proposition \ref{prop:cov_asymp}(i) we
get pointwise by polarisation
\[
f_{\delta}((s+s',w^{(\delta)}),(s',w^{(\delta)}))\r{}f((s+s',\Delta m),(s',\Delta m))
\]
In order to conclude observe for $d\geq2$ by (\ref{eq:f_delta_bounds})
(with $w^{(\delta)}$ instead of $\Delta z$) that
\[
\left|f_{\delta}\left((s+s',w^{(\delta)}),(s',w^{(\delta)})\right)\right|\lesssim(1\wedge s^{-1/2-d/8})(1\wedge(s')^{-1/2-3d/8}).
\]
If $\theta\in C^{1+\alpha'}(\overline{\Lambda})$ for $\alpha'>1/2$ and $\int_{\R^{d}}m(x)dx=0$,
the improved bound in Lemma \ref{lem:S_delta_bounds}(iii) for (\ref{eq:f_delta_bounds})
gives
\[
\left|f_{\delta}\left((s+s',w^{(\delta)}),(s',w^{(\delta)})\right)\right|\lesssim(1\wedge s^{-1/2-\alpha'/2-d/4})(1\wedge(s')^{-1/2-\alpha'/2-d/4}).
\]
In both cases (\ref{eq:V_lim}) follows from dominated convergence,
noting 
\[
\int_{0}^{\infty}f((s+s',\Delta m),(s',\Delta m))ds'=\theta(0)^{-1}\Psi(e^{\theta(0)s\Delta}\Delta m,\Delta m).\qedhere
\]
\end{proof}
The next result improves on Proposition \ref{prop:cov_asymp}(ii)
when $B$ is a multiplication operator, by making lower order terms
explicit. This is necessary for the proof of Theorem \ref{thm:ProxyMLE_CLT}.
The main difficulty is to work around not having a rate of convergence
in Proposition \ref{prop:semigroup_convergence}(ii).
\begin{prop}
\label{prop:cov_asymp_special}Let $z\in H^{4}(\R^{d})$ have compact
support in $\Lambda_{\delta'}$ for some $\delta'>0$ and suppose that $B=M_{\sigma}$ with $\sigma\in C^{1}(\overline{\Lambda})$.
If $d=1$, then assume $\theta\in C^{1+\alpha'}(\overline{\Lambda})$ for $\alpha'>1/2$ and
$\int_{\R}z(x)dx=0$. Then for $0<\delta\leq\delta'$ and $\delta\rightarrow 0$
\begin{align*}
 & \delta^{-2}\int_{0}^{T}\Var(\sc{\tilde{X}(t)}{(\Delta z)_{\delta}})dt=\frac{T\sigma^{2}(0)}{2\theta(0)}\norm{\nabla z}_{L^{2}(\R^{d})}^{2}\\
 & \qquad+\frac{\delta T}{2}\sc{\sc{\nabla\left(\tfrac{\sigma^{2}}{\theta}\right)(0)}x_{\R^{d}}}{|\nabla z|^{2}}_{L^{2}(\R^{d})}+o(\delta).
\end{align*}
\end{prop}

\begin{proof}

Let $\sc{\overline{X}(t)}{\cdot}$ be defined as $\sc{\tilde{X}(t)}{\cdot}$
in (\ref{eq:weak_solution_l}), but with semigroup  $(\overline{S}_{\theta}(t))_{t\geq0}$ on $L^{2}(\Lambda)$ generated by $A_{\theta}=\Delta_{\theta}$. As before, $(\overline{S}_{\theta,\delta}(t))_{t\geq0}$ is the corresponding semigroup on $L^2(\Lambda_{\delta})$ generated by $\Delta_{\theta(\delta\cdot)}$. Note that the $\overline{S}_{\theta,\delta}(t)$
are self-adjoint. With $v^{(\delta)}:=\delta^{-1}(\Delta_{\theta(\delta\cdot)}-\theta(0)\Delta)z$ set for $0\leq t\leq T$
\begin{align*}
T_{1}(t) & =-\frac{2\delta^{-1}}{\theta^{2}(0)}\Cov(\langle\overline{X}(t),(\Delta_{\theta(\delta\cdot)}z)_{\delta}\rangle,\sc{\overline{X}(t)}{v_{\delta}^{(\delta)}}),\\
T_{2}(t) & =\frac{\delta^{-2}}{\theta^{2}(0)}\Var(\sc{\overline{X}(t)}{(\Delta_{\theta(\delta\cdot)}z)_{\delta}})-\frac{\sigma^2(0)}{\theta^{2}(0)}\int_{0}^{t\delta^{-2}}\norm{\overline{S}_{\theta,\delta}\left(s\right)\Delta_{\theta(\delta\cdot)}z}_{L^{2}(\Lambda_{\delta})}^{2}ds,\\
T_{3}(t) & =\frac{\sigma^2(0)}{\theta^{2}(0)}\int_{0}^{t\delta^{-2}}\norm{\overline{S}_{\theta,\delta}\left(s\right)\Delta_{\theta(\delta\cdot)}z}_{L^{2}(\Lambda_{\delta})}^{2}ds-\frac{\sigma^{2}(0)}{2\theta(0)}\norm{\nabla z}_{L^{2}(\R^{d})}^{2},\\
R_{1}(t) & =\delta^{-2}\Cov(\sc{\tilde{X}(t)}{(\Delta z)_{\delta}}-\sc{\overline{X}(t)}{(\Delta z)_{\delta}},\sc{\tilde{X}(t)}{(\Delta z)_{\delta}}+\sc{\overline{X}(t)}{(\Delta z)_{\delta}}),\\
R_{2}(t) & =\frac{1}{\theta^{2}(0)}\Var(\sc{\overline{X}(t)}{v_{\delta}^{(\delta)}}),
\end{align*}
and introduce the decompositions
\begin{align*}
 \delta^{-2}\int_{0}^{T}\Var(\sc{\tilde{X}(t)}{(\Delta z)_{\delta}})dt & =\delta^{-2}\int_{0}^{T}\Var(\sc{\overline{X}(t)}{(\Delta z)_{\delta}})dt+\int_{0}^{T}R_{1}(t)dt\\
 \quad =  \frac{\delta^{-2}}{\theta^2(0)}\int_{0}^{T}\Var(\sc{\overline{X}(t)}{&(\Delta_{\theta(\delta\cdot)} z)_{\delta}})dt +  \int_{0}^{T}(T_{1}(t)+R_{2}(t))dt + o(\delta)\\
 \quad=\frac{T\sigma^{2}(0)}{2\theta(0)}\norm{\nabla z}_{L^{2}(\R^{d})}^{2}+&\int_{0}^{T}(T_{1}(t) +T_{2}(t)+T_{3}(t))dt + o(\delta),
\end{align*}
where we use for the remainder terms that $\int_{0}^{T}R_{i}(t)dt=o(\delta)$, $i=1,2$, by Lemmas \ref{lem:variationOfParameters} and \ref{lem:R_2}
below. The claim follows from Lemmas \ref{lem:T_1}, \ref{lem:T_2}
and \ref{lem:T_3} below, which show
\begin{align*}
\delta^{-1}\int_{0}^{T}(T_{1}(t)+T_{2}(t)+T_{3}(t))dt & \rightarrow\frac{T}{2}\sc{\sc{\tfrac{\theta(0)}{\sigma^{2}(0)}\nabla\sigma^{2}(0)-\nabla\theta(0)}x_{\R^{d}}}{|\nabla z|^{2}}_{L^{2}(\R^{d})}.
\end{align*}
\end{proof}

\begin{lem}
\label{lem:T_1} In Proposition \ref{prop:cov_asymp_special} we have
\[
\delta^{-1}\int_{0}^{T}T_{1}(t)dt\rightarrow-\frac{T\sigma^2(0)}{\theta^2(0)}\left\langle \sc{\nabla\theta(0)}x_{\R^{d}},|\nabla z|^{2}\right\rangle _{L^{2}(\R^{d})}.
\]
\end{lem}

\begin{proof}
Lemma \ref{lem:v0} above with $A_{\theta}=\Delta_{\theta}$ yields $v^{(\delta)}\rightarrow v:=\Delta(\scapro{\nabla\theta(0)}{x}_{\R^{d}}z)-\scapro{\nabla\vartheta(0)}{\nabla z}_{\R^{d}}$ in $L^{2}$. 
Moreover, since $\theta \in C^1(\overline{\Lambda})$, we have $\norm{\Delta_{\theta(\delta\cdot)}z-\theta(0)\Delta z}_{L^1\cap L^2(\R^d)}\leq C \delta$. If $d\geq2$, then Proposition \ref{prop:cov_asymp}(iii) with $w^{(\delta)}=\theta(0)^{-1}\Delta_{\theta(\delta\cdot)}z$, $u^{(\delta)}=v^{(\delta)}$ already implies $\delta^{-1}\int_{0}^{T}T_{1}(t)dt\rightarrow-2T\theta^{-2}(0)\Psi(\Delta z,v)$, and the claim follows from Lemma \ref{lem:laplace_to_nabla}
above, recalling the identity $\Psi(\Delta z,v)=-\frac{\sigma^{2}(0)}{2}\sc zv_{L^{2}(\R^{d})}$
from (\ref{eq:Psi_Delta_example}). For $d=1$, on the other hand, the properties $\theta\in C^{1+\alpha'}(\overline{\Lambda})$ for $\alpha'>1/2$ and $\int_{\R}z(x)dx=0$ ensure by Lemma \ref{lem:v0}(ii) that $v=\Delta m$
for a compactly supported $m\in H^{2}(\R)$ with $\norm{v^{(\delta)}-\Delta m}_{L^1\cap L^2(\R)}\leq C \delta^{\alpha'}$. By polarisation and
Proposition \ref{prop:cov_asymp}(ii) (with $w^{(\delta)}=\theta(0)^{-1}\Delta_{\theta(\delta\cdot)}z$ and $w^{(\delta)}=v^{(\delta)}$) $\delta^{-1}\int_{0}^{T}T_{1}(t)dt$
converges again to the claimed limit.
\end{proof}

\begin{lem}
\label{lem:T_2} In Proposition \ref{prop:cov_asymp_special} we have
\begin{align*}
\delta^{-1}\int_{0}^{T}T_{2}(t)dt & \r{}\frac{T}{2\theta(0)}\left\langle \left\langle \nabla\sigma^{2}(0),x\right\rangle _{\R^{d}},|\nabla z|^{2}\right\rangle _{L^{2}(\R^{d})}.
\end{align*}
\end{lem}

\begin{proof}
Using \eqref{eq:scaled_cov_function}, we have $\theta^2(0)T_{2}(t)=\int_{0}^{t\delta^{-2}}f_{\delta}(s)ds$
for
\begin{align*}
f_{\delta}(s) & =\sc{(\sigma^{2}(\delta\cdot)-\sigma^{2}(0))\overline{S}_{\theta,\delta}(s)\Delta_{\theta(\delta\cdot)}z}{\overline{S}_{\theta,\delta}(s)\Delta_{\theta(\delta\cdot)}z}_{L^{2}(\Lambda_{\delta})}\\
 & =\delta\int_{0}^{1}\sc{\sc{\nabla\sigma^{2}(\delta rx)}x_{\R^{d}}\overline{S}_{\theta,\delta}(s)\Delta_{\theta(\delta\cdot)}z}{\overline{S}_{\theta,\delta}(s)\Delta_{\theta(\delta\cdot)}z}_{L^{2}(\Lambda_{\delta})}dr.
\end{align*}
By the Cauchy-Schwarz inequality and the semigroup bounds in Proposition
\ref{prop:semigroup_bounds}(i,ii) above this means
\begin{align}
\left|\delta^{-1}f_{\delta}\left(s\right)\right| & \lesssim\norm{\left|x\right|\overline{S}_{\theta,\delta}(s)\Delta_{\theta(\delta\cdot)}z}_{L^{2}(\Lambda_{\delta})}\left(1\wedge s^{-1}\right)\nonumber\\ 
 & \lesssim\norm{\left|x\right|^{2}\overline{S}_{\theta,\delta}(s)\Delta_{\theta(\delta\cdot)}z}_{L^{2}(\Lambda_{\delta})}^{1/2}(1\wedge s^{-3/2})\nonumber \\ 
 & \lesssim\norm{|x|^{2}e^{\theta(0)s\Delta}|\Delta_{\theta(\delta\cdot)}z|}_{L^{2}(\R^{d})}^{1/2}(1\wedge s^{-3/2}) \nonumber \\
 & \lesssim(1\vee s)^{1/2}(1\wedge s^{-d/4})^{1/2}(1\wedge s^{-3/2})\lesssim(1\wedge s^{-1-d/8}),\label{eq:f_delta_bound_2}
\end{align}
where we used Proposition \ref{prop:semigroup_convergence}(i) for an approximating sequence of $\Delta_{\theta(\delta\cdot)}z$ with continuous functions and Lemma \ref{lem:LaplaceSemigroup_props}(iii) in the last two lines. 
We conclude from Proposition \ref{prop:semigroup_convergence}(ii) that
\[
\delta^{-1}f_{\delta}(s)\rightarrow f(s):=\theta^{2}(0)\sc{\sc{\nabla\sigma^{2}(0)}x_{\R^{d}}e^{\theta(0)s\Delta}\Delta z}{e^{\theta(0)s\Delta}\Delta z}_{L^{2}(\R^{d})}.
\]
Combining this with \eqref{eq:f_delta_bound_2}, the dominated convergence theorem shows $\theta^{2}(0)\delta^{-1}\int_{0}^{T}T_{2}(t)dt\rightarrow T\int_{0}^{\infty}f(s)ds$.
For the result note that by Lemmas \ref{lem:LaplaceSemigroup_props}(ii)
and \ref{lem:laplace_to_nabla} (here we need $z\in H^{4}(\R^{d})$)
$\theta^{-2}(0)\int_{0}^{\infty}f(s)ds$ equals
\begin{align*}
 & \int_{0}^{\infty}\langle\langle\nabla\sigma^{2}(0),-2\theta(0)s\nabla e^{\theta(0)s\Delta}\Delta z+e^{\theta(0)s\Delta}(x\Delta z)\rangle_{\R^{d}},e^{\theta(0)s\Delta}\Delta z\rangle_{L^{2}(\R^{d})}ds\\
 & \quad=\int_{0}^{\infty}\left\langle \langle\nabla\sigma^{2}(0),x\rangle_{\R^{d}}\Delta z,e^{2\theta(0)s\Delta}\Delta z\right\rangle _{L^{2}(\R^{d})}ds\\
 & \quad=-\frac{1}{2\theta(0)}\left\langle \langle\nabla\sigma^{2}(0),x\rangle_{\R^{d}}z,\Delta z\right\rangle _{L^{2}(\R^{d})}=\frac{1}{2\theta(0)}\left\langle \langle\nabla\sigma^{2}(0),x\rangle_{\R^{d}},\left|\nabla z\right|^{2}\right\rangle _{L^{2}(\R^{d})}.\qedhere
\end{align*}
\end{proof}

\begin{lem}
\label{lem:T_3}In Proposition \ref{prop:cov_asymp_special} we have
\[
\delta^{-1}\int_{0}^{T}T_{3}(t)dt\rightarrow\frac{T\sigma^2(0)}{2\theta^2(0)}\left\langle \sc{\nabla\theta(0)}x_{\R^{d}},|\nabla z|^{2}\right\rangle _{L^{2}(\R^{d})}.
\]
\end{lem}

\begin{proof}
Since $\overline{S}_{\theta,\delta}(s)$ is self-adjoint, we have
\begin{align*}
\frac{2\theta^2(0)}{\sigma^2(0)}T_{3}(t) & =\left\langle \Delta_{\theta(\delta\cdot)}\overline{S}_{\theta,\delta}\left(2t\delta^{-2}\right)z,z\right\rangle _{L^{2}(\Lambda_{\delta})}-\left\langle (\Delta_{\theta(\delta\cdot)}-\theta(0)\Delta)z,z\right\rangle _{L^{2}(\R^{d})}.
\end{align*}
Integrating over $0\leq t\leq T$ and using the semigroup bound in
Proposition \ref{prop:semigroup_bounds}(i) the first term is of order
$O(\delta^{2})$. Since $\theta\in C^{1}(\overline{\Lambda})$, the result follows from Lemma \ref{lem:laplace_to_nabla} and
\begin{align*}
(\Delta_{\theta(\delta\cdot)}-\theta(0)\Delta)z & =\delta(\int_{0}^{1}\sc{\nabla\theta(s\delta\cdot)}x_{\R^{d}}ds)\,\Delta z+\sc{\nabla\theta(\delta\cdot)}{\nabla z}_{\R^{d}}\\
 & =\delta\sc{\nabla\theta(0)}x_{\R^{d}}\Delta z+\sc{\nabla\theta(0)}{\nabla z}_{\R^{d}}+o(\delta).\qedhere
\end{align*}
\end{proof}
\begin{lem}
\label{lem:variationOfParameters}In Proposition \ref{prop:cov_asymp_special}
we have $\delta^{-1}\int_{0}^{T}R_{1}(t)dt\rightarrow 0$.
\end{lem}

\begin{proof}
By \eqref{eq:scaled_cov_function} write $R_{1}(t)=\int_{0}^{t\delta^{-2}}f_{\delta}(s)ds$ with $f_{\delta}(s)=\sc{g_{\delta}(s)}{h_{\delta}(s)}_{L^{2}(\Lambda_{\delta})}$ for $s\geq0$, where
\begin{align*}
g_{\delta}\left(s\right) & =(S_{\theta,\delta}^{*}\left(s\right)-\overline{S}_{\theta,\delta}(s))\Delta z,\\
h_{\delta}\left(s\right) & =\sigma^{2}(\delta\cdot)(S_{\theta,\delta}^{*}\left(s\right)+\overline{S}_{\theta,\delta}(s))\Delta z.
\end{align*}
An application of the dominated convergence theorem then proves the result, if
\begin{align}
\left|\delta^{-1}f_{\delta}(s)\right| & \lesssim1\wedge s^{-3/2},\,\,\,\,0<s\leq t\delta^{-2},\label{eq:T_4_1}\\
\delta^{-1}f_{\delta}\left(s\right) & \r{}0.\label{eq:T_4_2}
\end{align}
In order
to show (\ref{eq:T_4_1}) and (\ref{eq:T_4_2}) we use the \emph{variation
of parameters formula }(\citet[p. 162]{engel1999one}): The function
$[0,s]\ni s'\mapsto S_{\theta,\delta}^{*}\left(s'\right)\overline{S}_{\theta,\delta}(s-s')\Delta z$
has derivative $S_{\theta,\delta}^{*}\left(s'\right)(A_{\theta,\delta}^{*}-\Delta_{\theta(\delta\cdot)})\overline{S}_{\theta,\delta}\left(s-s'\right)\Delta z$,
implying
\begin{align*}
g_{\delta}\left(s\right) & =\int_{0}^{s}S_{\theta,\delta}^{*}\left(s'\right)(A_{\theta,\delta}^{*}-\Delta_{\theta(\delta\cdot)})\overline{S}_{\theta,\delta}\left(s-s'\right)\Delta z\,ds'.
\end{align*}
Since the operator $A_{\theta,\delta}^{*}-\Delta_{\theta(\delta\cdot)}=A_{0,\delta}^{*}$
is not bounded, a careful analysis is required. Decomposing it into
first and zero order terms we have
\begin{align}
g_{\delta}\left(s\right) & =-\delta\int_{0}^{s}S_{\theta,\delta}^{*}\left(s'\right)\sc{a(\delta\cdot)}{\nabla\overline{S}_{\theta,\delta}\left(s-s'\right)\Delta z}_{\R^{d}}ds'\nonumber \\
 & +\delta^{2}\int_{0}^{s}S_{\theta,\delta}^{*}\left(s'\right)(b(\delta\cdot)-(\text{div}a)(\delta\cdot))\overline{S}_{\theta,\delta}\left(s-s'\right)\Delta z\,ds'.
\end{align}
The semigroup bounds in Proposition \ref{prop:semigroup_bounds}(ii)
and in Lemma \ref{lem:S_delta_bounds}(i,ii,iii) above, subject to
$d\geq2$ or  $d=1$ with $\theta\in C^{1+\alpha'}(\overline{\Lambda})$ for $\alpha'>1/2$ and $\int_{\R}z(x)dx=0$,
show for sufficiently small $\delta$ and $0\leq s'<s\leq t\delta^{-2}$
that
\begin{align*}
 & \norm{\Delta_{\theta,\delta}\overline{S}_{\theta,\delta}(s-s')\Delta z}_{L^{2}(\Lambda_{\delta})}=\norm{\Delta_{\theta,\delta}\overline{S}_{\theta,\delta}(\frac{s-s'}{2})\overline{S}_{\theta,\delta}(\frac{s-s'}{2})\Delta z}_{L^{2}(\Lambda_{\delta})}\nonumber \\
 & \lesssim(s-s')^{-1}\norm{\overline{S}_{\theta,\delta}(\frac{s-s'}{2})\Delta z}_{L^{2}(\Lambda_{\delta})}\lesssim(s-s')^{-1}(1\wedge(s-s')^{-1}).
\end{align*}
Hence, the computations in (\ref{eq:A0_upper_bound}) above show
\begin{align*}
 & \norm{\sc{a(\delta\cdot)}{\nabla\overline{S}_{\theta,\delta}\left(s-s'\right)\Delta z}_{\R^{d}}}_{L^{2}(\Lambda_{\delta})}\nonumber \\
 & \qquad\lesssim\norm{\Delta_{\theta(\delta\cdot)}\overline{S}_{\theta,\delta}\left(s-s'\right)\Delta z}_{L^{2}(\Lambda_{\delta})}^{1/2}\norm{\overline{S}_{\theta,\delta}\left(s-s'\right)\Delta z}_{L^{2}(\Lambda_{\delta})}^{1/2}\nonumber \\
 & \qquad\lesssim(s-s')^{-1/2}(1\wedge(s-s')^{-1}).
\end{align*}
Because of Proposition \ref{prop:semigroup_bounds}(i) this yields
for $0<s\leq t\delta^{-2}$ 
\begin{align*}
 \norm{\delta^{-1}g_{\delta}\left(s\right)}_{L^{2}(\Lambda_{\delta})} & 
 \lesssim\int_{0}^{s}\left((s-s')^{-1/2}(1\wedge(s-s')^{-1})+\delta(1\wedge(s-s')^{-1})\right)ds' \\
 & \lesssim\int_{0}^{s}(s'){}^{-1/2}(1\wedge(s')^{-1})ds'+t^{1/4}\delta^{1/2}\int_{0}^{s}(1\wedge(s')^{-1-1/4})ds'.
\end{align*}
In all, this is of order $1\wedge s^{-1/2}$. Since also $\norm{h_{\delta}(s)}_{L^{2}(\Lambda_{\delta})}\lesssim1\wedge s^{-1}$ by Lemma \ref{lem:S_delta_bounds}(ii,iii), subject to $d\geq 2$ or the conditions in $d=1$, we obtain (\ref{eq:T_4_1}). With respect to (\ref{eq:T_4_2}) fix $s$ and observe by the convergence of
the semigroups in Proposition \ref{prop:semigroup_convergence}(ii)
above and $\sigma^{2}(\delta\cdot)\rightarrow \sigma^{2}(0)$
that $h_{\delta}(s)\rightarrow2\sigma^{2}(0)e^{\theta(0)s\Delta}\Delta z$
in $L^{2}(\R^{d})$. Therefore, $f_{\delta}\left(s\right)=-2\delta\sigma^2(0)f_{\delta}^{(1)}\left(s\right)+o\left(\delta\right)$, uniformly in $s$, for
\begin{align*}
f_{\delta}^{(1)}(s) & =\int_{0}^{s}\sc{S_{\theta,\delta}^{*}\left(s'\right)\sc{a(\delta\cdot)}{\nabla\overline{S}_{\theta,\delta}\left(s-s'\right)\Delta z}_{\R^{d}}}{e^{\theta(0)s\Delta}\Delta z}_{L^{2}(\Lambda_{\delta})}ds'\\
 & =\sum_{i=1}^{d}\int_{0}^{s}\sc{\partial_{i}\overline{S}_{\theta,\delta}\left(s-s'\right)\Delta z}{a_{i}(\delta\cdot)S_{\theta,\delta}\left(s'\right)(e^{\theta(0)s\Delta}\Delta z)|_{\Lambda_{\delta}}}_{L^{2}(\Lambda_{\delta})}ds'.
\end{align*}
In the same way, since $a_{i}(\delta\cdot)S_{\theta,\delta}\left(s'\right)(v|_{\Lambda_{\delta}})\rightarrow a_{i}(0)e^{\theta(0)s'\Delta}v$
for $v\in L^{2}(\R^{d})$ by Proposition \ref{prop:semigroup_convergence}(ii),
we have $f_{\delta}^{(1)}(s)=f_{\delta}^{(2)}(s)+o(\delta)$ for
\begin{align*}
f_{\delta}^{(2)}(s) & = \sum_{i=1}^{d}a_{i}(0)\int_{0}^{s}\sc{\partial_{i}\overline{S}_{\theta,\delta}\left(s-s'\right)\Delta z}{e^{\theta(0)(s'+s)\Delta}\Delta z}_{L^{2}(\Lambda_{\delta})}ds'\\
 & =-\sum_{i=1}^{d}a_{i}(0)\int_{0}^{s}\sc{\overline{S}_{\theta,\delta}\left(s-s'\right)\Delta z}{e^{\theta(0)(s'+s)\Delta}\Delta\partial_{i}z}_{L^{2}(\Lambda_{\delta})}ds'.
\end{align*}
Noting that $\overline{S}_{\theta,\delta}\left(s-s'\right)\Delta z\rightarrow e^{\theta(0)(s-s')\Delta}\Delta z$
in $L^{2}(\R^{d})$, we finally obtain
\begin{align*}
\delta^{-1}f_{\delta}(s) & \rightarrow2\sigma^{2}(0)\sum_{i=1}^{d}a_{i}(0)\int_{0}^{s}\sc{e^{\theta(0)(s-s')\Delta}\Delta z}{e^{\theta(0)(s'+s)\Delta}\Delta\partial_{i}z}_{L^{2}(\Lambda_{\delta})}ds'\\
 & =\frac{2\sigma^{2}(0)s}{\theta(0)}\sum_{i=1}^{d}a_{i}(0)\sc{\Delta z}{e^{2\theta(0)s\Delta}\Delta\partial_{i}z}_{L^{2}(\R^{d})},
\end{align*}
which vanishes according to Lemma \ref{lem:laplace_to_nabla}.
\end{proof}
\begin{lem}
\label{lem:R_2} In Proposition \ref{prop:cov_asymp_special} we have
$\delta^{-1}\int_{0}^{T}R_{2}(t)dt\rightarrow 0$.
\end{lem}

\begin{proof}
By \eqref{eq:scaled_cov_function} 
we have for $0\leq t\leq T$
\begin{align*}
R_{2}(t) & \lesssim\delta^{2}\int_{0}^{t\delta^{-2}}\norm{\overline{S}_{\theta,\delta}(s)v^{(\delta)}}_{L^{2}(\Lambda_{\delta})}^{2}ds.
\end{align*}
Recall from Lemma \ref{lem:T_1} that $v^{(\delta)}$ converges
in $L^{2}(\R^{d})$. For $d\geq 2$, Lemma \ref{lem:S_delta_bounds}(i) then shows uniformly in $0\leq t\leq T$ that
\[
\delta^{-1}R_{2}(t)\lesssim\delta\int_{0}^{t\delta^{-2}}1\wedge s^{-1}ds\lesssim\delta^{1/2}\int_{0}^{\infty}1\wedge(s^{-1-1/4})ds\rightarrow 0,
\]
implying the claim in this case. For $d=1$ it is enough to recall from Lemma \ref{lem:T_1} that $\norm{v^{(\delta)}-\Delta m}_{L^{1}\cap L^{2}(\R)}\leq C\delta^{\alpha'}$, and so the claim follows from Lemma \ref{lem:S_delta_bounds}(iii).
\end{proof}
\begin{prop}
\label{prop:eightMom} Grant Assumption \ref{assu:B}.
Let $z\in H^{2}(\R^{d})$ have compact support in $\Lambda_{\delta'}$ for some $\delta'>0$ and for $d=1$ assume
$\theta\in C^{1+\alpha'}(\overline{\Lambda})$ for $\alpha'>1/2$, $\int_{\R}z(x)dx=0$.
For $0<\delta\leq\delta'$ set $\xi_{\delta}(t)=\sc{\tilde{X}(t)}{(\Delta z)_{\delta}}$.
Then the fourth moment of $\delta^{-3}\int_{0}^{T}(\xi_{\delta}(t)^{2}-\E[\xi_{\delta}(t)^{2}])dt$
converges, with $\Psi$ from (\ref{eq:Psi}), for $\delta\rightarrow 0$ to
\[
3\left(4T\theta(0)^{-3}\int_{0}^{\infty}\Psi(e^{s\Delta}\Delta z,\Delta z)^{2}ds\right)^{2}.
\]
\end{prop}

\begin{proof}

In view of Proposition \ref{prop:var_with_delta}(iv) it is enough
to show
\begin{align*}
R_{\delta} & :=\E\left[\left(\int_{0}^{T}(\xi_{\delta}(t)^{2}-\E[\xi_{\delta}(t)^{2}])dt\right)^{4}\right]-3\Var\left(\int_{0}^{T}\xi_{\delta}(t)^{2}dt\right)^{2}\\
 & =o\left(\Var\left(\int_{0}^{T}\xi_{\delta}(t)^{2}dt\right)^{2}\right)=o(\delta^{12}).
\end{align*}
Abbreviating $c_{\delta}(t,s)=\text{Cov}(\sc{\tilde{X}(t)}{(\Delta z)_{\delta}},\sc{\tilde{X}(s)}{(\Delta z)_{\delta}})$,
recall from (\ref{eq:var}) that
\begin{align*}
 & 3\Var\left(\int_{0}^{T}\xi_{\delta}(t)^{2}dt\right)^{2}=3\left(\int_{0}^{T}\int_{0}^{T}2c_{\delta}(t,s)^{2}dtds\right)^{2}.
\end{align*}
Wick's formula (\citet[Theorem 1.28]{Janson:1997uy}) for 8th centered
Gaussian moments $\E[\prod_{i=1}^{8}Z_{i}]=\sum_{\pi\in\Pi_{2}(8)}\prod_{(i,j)\in\pi}\E[Z_{i}Z_{j}]$
applied to $Z_{i}=\xi_{\delta}(t_{i})$ for $0\leq t_{i}\leq T$,
with $\Pi_{2}(8)$ being the set of all partitions $\pi$ of $\{1,\ldots,8\}$
into 2-tuples, therefore yields, using the symmetry of the integrand
in $(t_{1},t_{2},t_{3},t_{4})$,
\begin{align*}
R_{\delta} & =48\int_{0}^{T}\int_{0}^{T}\int_{0}^{T}\int_{0}^{T}c_{\delta}(t_{1},t_{2})c_{\delta}(t_{2},t_{3})c_{\delta}(t_{3},t_{4})c_{\delta}(t_{4},t_{1})\,dt_{1}dt_{2}dt_{3}dt_{4}.
\end{align*}
The calculations in the proof of Proposition \ref{prop:var_with_delta}(iv)
show for $s\leq t$, both when $d\geq 2$ and when $d=1$, $\theta\in C^{1+\alpha'} (\overline{\Lambda})$ for $\alpha'>1/2$,
$\int_{\R}z(x)dx=0$, that
\begin{align}
\abs{c_{\delta}(t,s)} & =\left|\int_{0}^{s}f_{\delta}\left((\delta^{-2}(t-s'),\Delta z),(\delta^{-2}(s-s'),\Delta z)\right)ds'\right|\label{EqCovBound}\\
 & =\delta^{2}\left|\int_{0}^{s\delta^{-2}}f_{\delta}\left((\delta^{-2}(t-s)+s',\Delta z),(s',\Delta z)\right)ds'\right|\nonumber \\
 & \lesssim\delta^{2}\int_{0}^{s\delta^{-2}}\left(1\wedge(\delta^{-2}|t-s|+s')^{-1}\right)\left(1\wedge(s')^{-1}\right)ds'\nonumber \\
 & \lesssim\delta^{2}\ell_{d,2}(1\wedge(\delta^{-2}|t-s|){}^{-1})=:\bar{c}_{\delta}(\delta^{-2}(t-s)),\nonumber
\end{align}
so that as in (\ref{eq:f_g_delta}), substituting $s_{i}=\delta^{-2}(t_{i+1}-t_{i})$,
\begin{align*}
R_{\delta} & \lesssim\delta^{6}\int_{0}^{T}\int_{0}^{T\delta^{-2}}\int_{0}^{T\delta^{-2}}\int_{0}^{T\delta^{-2}}\bar{c}_{\delta}(s_{1})\bar{c}_{\delta}(s_{2})\bar{c}_{\delta}(s_{3})\\
 & \qquad\cdot\bar{c}_{\delta}(s_{1}+s_{2}+s_{3})dt_{1}ds_{1}ds_{2}ds_{3}\\
 & \lesssim T\delta^{14}\ell_{d,2}^{4}\left(\int_{0}^{T\delta^{-2}}(1\wedge s{}^{-1})ds\right)^{3}=o(\delta^{12}).\qedhere
\end{align*}
\end{proof}

\bibliographystyle{apalike2}
\bibliography{refs_SPDE}

\end{document}